\documentclass[a4paper,oneside,reqno]{amsart}
\usepackage{amscd,amsfonts,amssymb,amsmath,amsthm,bbm}
\usepackage{graphicx,psfrag}
\usepackage{color}
\usepackage[english]{babel}
\usepackage[headinclude,DIV13]{typearea}

\areaset{15.1cm}{25.0cm}
\renewcommand {\familydefault}{bch}
\numberwithin{equation}{section}

\newtheorem{theorem}{Theorem}[section]
\newtheorem{lemma}[theorem]{Lemma}
\newtheorem{proposition}[theorem]{Proposition}
\newtheorem{corollary}[theorem]{Corollary}
\newtheorem{claim}[theorem]{Claim}

\theoremstyle{definition}

\newtheorem{definition}[theorem]{Definition}

\theoremstyle{remark}
\newtheorem*{remark}{Remark}

 \newcommand{\be}{\begin{equation}}
 \newcommand{\ee}{\end{equation}}

 \newcommand{\ba}{\begin{array}}
 \newcommand{\ea}{\end{array}}

 \newcommand{\bea}{\begin{eqnarray}}
 \newcommand{\eea}{\end{eqnarray}}

 \newcommand{\bl}{\begin{lemma}}
 \newcommand{\el}{\end{lemma}}

 \newcommand{\br}{\begin{remark}}
 \newcommand{\er}{\end{remark}}

 \newcommand{\bt}{\begin{theorem}}
 \newcommand{\et}{\end{theorem}}

 \newcommand{\bd}{\begin{definition}}
 \newcommand{\ed}{\end{definition}}

 \newcommand{\bcl}{\begin{claim}}
 \newcommand{\ecl}{\end{claim}}

 \newcommand{\bp}{\begin{proposition}}
 \newcommand{\ep}{\end{proposition}}

 \newcommand{\bc}{\begin{corollary}}
 \newcommand{\ec}{\end{corollary}}

 \newcommand{\bpr}{\begin{proof}}
 \newcommand{\epr}{\end{proof}}

 \newcommand{\bi}{\begin{itemize}}
 \newcommand{\ei}{\end{itemize}}

 \newcommand{\ben}{\begin{enumerate}}
 \newcommand{\een}{\end{enumerate}}

 \def \D {{\mathbb D}}

 \def \R {{\mathbb R}}
 
 \def \N {{\mathbb N}}
 \def \P {{\mathbb P}}
 \def \E {{\mathbb E}}

 \def \cD {\mathcal{D}}

 \def \cL {\mathcal{L}}

 \def \cR {\mathcal{R}}
 
 \def \cT {\mathcal{T}}

 \def \cX {\mathcal{X}}

 \def \a {{\alpha}}
 \def \b {{\beta}}
 \def \d {{\delta}}
 \def \e {{\varepsilon}}

  \def \g {{\gamma}}
 
 \def \s {{\sigma}}
 \def \k {{\kappa}}
 \def \z {{\zeta}}
 \def \m {{\mu}}
 
 \def \t {{\tau}}
 \def \th {{\theta}}
 \def \l {{\lambda}}

 \def \D {{\Delta}}
 
 \def \G {{\Gamma}}

\def \à {{\`{a}}}
\def \ì{{\`{\i}}}
\def \ò{{\`{o}}}
\def \è{{\`{e}}}
\def \ù{{\`{u}}}

\def \gL{\mathcal{L}_{\mathcal R}^{^*}}
\def \lL{\mathcal{L}_{\mathcal R,\l}^{^*}}
\def\ti{\tilde}
\def \XR{{\cX\!\setminus\!\cR}}

 \def \1{\mathbbm{1}} % funziona solo con il package bbm. sennò usare quello sotto
 \def\no{\noindent}
 
 \def\var{\hbox{\rm Var}}

\def\sfrac#1#2{{\textstyle{\frac{#1}{#2}}}}

\begin{document}

\title[Metastability and quasi-stationary measures]
{Metastable states,\\
quasi-stationary  distributions and soft measures}
\thanks{Supported by GDRE 224 GREFI-MEFI, by
the European Research Council through the Advanced Grant PTRELSS 228032,
 and by the FIRB Project RBFR10N90W (MIUR, Italy). A. Bianchi acknowledges
the University of Padua that provided partial financial support through
the Project “Stochastic Processes and Applications to Complex Systems” (CPDA123182)}
\author[A. Bianchi]{Alessandra Bianchi}
\address{A. Bianchi \\
Dip. di Matematica - Universit\`a  di Padova\\
Via Trieste, 63 - 35121 Padova, Italy}
\email{bianchi@math.unipd.it}
\author[A. Gaudilli\`ere]{Alexandre Gaudilli\`ere}
\address{A. Gaudilli\`ere \\
Aix Marseille Universite, CNRS, Centrale Marseille, I2M, UMR 7373, 13453 Marseille, France}
\email{alexandre.gaudilliere@math.cnrs.fr}
\date{June 11, 2015}

\begin{abstract}
We establish metastability in the sense of Lebowitz and Penrose
under practical and simple hypotheses for Markov chains on a
finite configuration space in some asymptotic regime.
By comparing restricted ensemble and quasi-stationary measures,
and introducing soft measures as interpolation between the two,
we prove asymptotic exponential exit law and, on a generally
different time scale, asymptotic exponential transition law.
By using potential-theoretic tools,
and introducing ``$(\kappa, \lambda)$-capacities'',
we give sharp estimates on relaxation time,
as well as mean exit time and transition time.
We also establish local thermalization on shorter time scales.
\end{abstract}

\maketitle

\section{Metastability after Lebowitz and Penrose}

\subsection{Phenomenology and modelization}

Lebowitz and Penrose characterized {\em metastable thermodynamic
 states} by the following properties~\cite{LP}:
\begin{itemize}
\item[(a)] only one thermodynamic phase is present,
\item[(b)] a system that starts in this state is likely to
take a long time to get out,
\item[(c)] once the system has gotten out, it is unlikely
to return.
\end{itemize}
We can think, for example, about freezing fog made of small droplets
in which only one phase is present (liquid phase) that remains
for a long time in such a state (until collision with ground
or trees, forming then hard rime) and that once frozen
will typically not return to liquid state before pressure
or temperature have changed.

To model such a state they considered in \cite{LP}
a deterministic dynamics with equilibrium measure $\mu$.
First, they associated with the metastable phase
a subset $\cR$ of the configuration space, and described this
metastable state by the {\em restricted ensemble}
$\mu_\cR = \mu(\cdot|\cR).$
Second, they proved that the escape rate from $\cR$
of the system started in $\mu_\cR$ is maximal
at time $t = 0$, and that this initial escape
rate is very small.
Last, they used standard methods of equilibrium statistical
mechanics to deal with (c). As an estimate of the returning probability
to the metastable state they used the fraction
of members of the {\em equilibrium} ensemble that have configurations
in $\cR$ and they noted~(\cite{LP}, Section~8):
\begin{quote}
This amounts to assuming that a system whose dynamical state
has just left $\cR$ is no more likely to return to it than
one whose dynamical state was never anywhere near $\cR$.
The validity of this assumption, at least in the short run, is
dubious, but at least it provides us with some indication of what
to expect.
\end{quote}

In this paper we want to give a different model
for the same phenomenology
that overcomes the last difficulty.
We will work with stochastic processes
rather than deterministic dynamics,
but the Lebowitz-Penrose modelization
will be our guideline.
We will try to recover this phenomenology
under simple and practical hypotheses only.
Since the study of metastability has been considerably
enriched after the Lebowitz and Penrose work,
we want also to incorporate in our modeling
as much as possible of what was previously achieved.
We will then make a brief and partial review of
these achievements.
Our goals and starting ideas will depend on this review
but not our proofs, since we want to make this
paper as self-contained as possible.
Our model and results are presented in Section~\ref{model_and_results}.
Examples of applications are given in Section~\ref{trotta}.

\subsection{A partial review}

Since the Lebowitz and Penrose paper, an enormous amount of work
has been done to describe the metastability phenomenon.
In particular Cassandro, Galves, Olivieri and Vares introduced the path-wise
approach, which focused, in the context of stochastic processes, on
time averages associated with an asymptotic exponential law~\cite{CGOV}.
This was further developed by the pioneering works of Neves and
Schonmann who studied the typical paths for stochastic Ising model
in a given volume in the low temperature regime~(\cite{NS1},
\cite{NS2}). This work was then extended to higher dimensions,
infinite-volume and fixed-temperature regimes, locally conservative
dynamics and probabilistic cellular automata~(\cite{DS}, \cite{BC}, \cite{SS},
\cite{GOS}, \cite{CNS},~\cite{CM}).

As developed in~\cite{OV}, a crucial role was played by large deviation tools inherited
from Wentzell and Freidlin in their reduction procedure from continuous stochastic processes
to finite-configuration-space Markov chains with exponentially small transition rates~\cite{WF}.
This is especially true for very-low-temperature regimes,
but the same kind of reduction procedure made it possible to deal
in various cases with large-volume rather than low-temperature limits
(see~\cite{MP} for the Curie-Weiss model under random magnetic field, see~\cite{OV}
for further examples).

Then, using potential-theoretic rather than large deviations tools,
Bovier, Eckhoff, Gay\-rard and Klein developed a set of general techniques
to compute sharp asymptotics of the expected value of asymptotic
exponential laws associated with the metastability phenomenon,
and revisited (after~\cite{W}, \cite{S2}) the relation between spectrum of the generator
of the stochastic dynamics and metastability~(\cite{BEGK1},
\cite{BEGK2},~\cite{BEGK3}). This allowed, for example, to go beyond logarithmic
asymptotics for stochastic Ising models in the low-temperature regime~(\cite{BM},~\cite{BdHN})
and to prove the first rigorous results in the fully conservative case~(\cite{BdHS}),
to deal with metastability for the random hopping-time dynamics
associated with the Random Energy Model~(\cite{BBG}), to make a detailed
analysis of Sinai's random walk spectrum~(\cite{BF}), and to extend the study
of the disordered Curie-Weiss model to the case of continuous
magnetic field distribution~(\cite{BBI1}, ~\cite{BBI2}).
We refer to~\cite{BdH} for a comprehensive account of this approach.

We then reached an essentially complete comprehension
of the metastability phenomenon in at least two classes
of models: very low temperature dynamics in finite fixed volumes
and large volume or continuous-configuration space
dynamics that can be reduced via a Wentzell-Freidlin procedure
to the previous case. Of course, specific and often
nontrivial computations have to be made for each specific model,
but there exists a general approach to the problem
that is developed in~\cite{OV} and, as far as the potential-theoretic
part is concerned,~\cite{BEGK1}, \cite{BEGK2}, and \cite{BEGK3}
together with
~\cite{BL3}, \cite{BL4}
that bridges between potential theory
and typical path description by reinforcing and generalizing
the results of~\cite{S1} (and it is worth noting that~\cite{BL3},
after~\cite{BL1,BL2}, contemplates also the case of polynomially
small rather than only exponentially small transition probabilities).
For both classes of models, like one-dimensional metastable
systems as considered in~\cite{BF} or~\cite{BBF},
a recurrence property for
a very localized subset of the configuration space
(single configurations identified to metastable states in the first case,
small neighborhoods of the dynamics attractors in the second case)
plays an important role.

Beyond these two classes of models there are many limit cases,
special cases, and partial results.
For example, in~\cite{BBG} we are far from a finite-fixed-volume situation
but single configurations can still be identified with metastable states
and have still enough mass at equilibrium for potential-theoretic or renewal
techniques to work. This is not the case in~\cite{BdHS}, where potential-theoretic
tools give only expected values of some hitting times
when the system is started from some specific harmonic measures
that are very different from what one would expect to be a ``metastable state''
(here, like in the sequel and following Lebowitz and Penrose, we mean
a whole measure when referring to a metastable ``state'' and not to
a single configuration of the configuration space). Any kind of
exponential law is presently also lacking in this case.
The same difficulty is faced in~\cite{BBI1}, but it is overcome
in~\cite{BBI2} by mean of a specific coupling argument
that gives point-wise estimates and opens the way to the exponential law.
We also note that~\cite{BL4} develops some general martingale ideas
to deal with the same issues within the framework built from~\cite{BL1, BL3}.
In fact, though working with a different setup,
we share some of the leading ideas developed in~\cite{BL4},
which uses some of the key objects that we introduced through this work.
Such ideas also inspired~\cite{FMNSS}
where the non-reversible situation is contemplated.
Finally, the beautiful paper by Schonmann and Shlosman~\cite{SS}
achieves the {\em tour de force} of using essentially
equilibrium statistical mechanics
computations to deal with the dynamical problem of metastability.
In this case also the exponential law is lacking as well as sharp estimates
on the relaxation time, and even the simple formulation of such
properties is not completely obvious in this fixed-temperature and
vanishing-magnetic-field regime.

\subsection{Starting ideas}

In the present paper we want to elaborate some tools to describe
the metastability phenomenon beyond the case
of a dynamics with a recurrence property for a very localized subset of
the configuration space.
We will focus on exponential laws and sharp asymptotics
of their expected values.
We note that the exponential law itself suggests some kind
of recurrence property.
If it is not a recurrence property for a very localized
subset, it has to be in some sense a recurrence
property to a whole ``spread measure''.
And this measure should coincide with our metastable state.
Now, following Lebowitz and Penrose,
if we associate the metastable state with some subset
$\cR$ of the configuration space $\cX$, then, considering
property (b),  we have at least
two candidates to describe our metastable state:
one is the restricted ensemble
$\mu_\cR = \mu(\cdot|\cR)$,
the other is the quasi-stationary distribution
\begin{equation}\label{q-statdef}
\mu_\cR^*= \lim_{t\rightarrow +\infty} P_{\mu_\cR}(X(t) \in \cdot | \tau_{\cX\setminus\cR} > t)
\end{equation}
where $X(t)$ is the configuration of the system at time $t$
and $\tau_{\cX\setminus\cR}$ is the exit time of $\cR$ (we will be
more precise in the next section).
Notice that Eq. (\ref{q-statdef}) provides
the stationarity of $\mu_{\mathcal{R}}^*$ for the process conditioned
to not having exit $\cR$,
\begin{equation}
	P_{\mu_{\mathcal{R}}^*}\left(
		X(t) \in \cdot
		| \tau_{\mathcal{X} \setminus \mathcal{R}} > t
	\right) = \mu_{\mathcal{R}}^*
	\,,
\end{equation}
and thus explains the name of quasi-stationary distributions.

The main advantage of $\mu_\cR$ is that $\mu_\cR$ is
often an explicit measure one can compute with, while
$\mu_\cR^*$ is only implicitly defined.
The main advantage of $\mu_R^*$ is that the exit law of $\cR$
for the system started in $\mu_R^*$ {\em is} an exponential law.
Our first results will then start, as in~\cite{AB}, with a comparison between
$\mu_R$ and $\mu_R^*$. We will give simple and practical
hypotheses to ensure that they are close in some sense, then we will be able
to prove some kind of recurrence property for $\mu_R^*$.
In doing so we will also answer some problems left open in~\cite{AB}
(see our comment after formula~(\ref{banana})).
All this will be done in the simplest possible setup:
considering a Markov process on a finite configuration space
in some asymptotic regime (including the possibility
of sending to infinity the cardinality of the configuration space).

In the present work we will essentially build on the ideas of four different papers:~\cite{LP}
for the formulation of the problem,~\cite{CGOV} for the focus on exponential
laws,~\cite{BEGK1} for the introduction of potential-theoretic techniques
in the metastability field to get sharp estimates on some mean hitting times,
and Miclo's work~\cite{Mi1} where some concepts of local equilibrium,
and ``hitting times'' of such equilibriums, are introduced.
 As far as this last paper is concerned,
it will only work as a source of inspiration: we will not require a full spectrum
knowledge, and we will not introduce any notion of dependence of a local equilibrium on the initial condition.
Finally, we note that the idea of considering quasi-stationary measures as metastable states
was already contemplated in~\cite{HMS}. Even though some of our results
echo some of~\cite{HMS}, we were not able to make any clear comparison, essentially
because of the much more analytical point of view of~\cite{HMS} and the many hypotheses
introduced in the results of~\cite{HMS}.
We note that~\cite{HMS} deals with a much more general setup than ours
since the authors consider non-reversible Markov processes on a continuous configuration space,
while we look at reversible Markov processes on finite configuration space.
However, the reason why we assume reversibility is to be able
to use potential-theoretic results to get sharp estimates on mean times
via variational principles, a question that is not considered
in~\cite{HMS}.

\subsection{Two new objects}

In this section we provide a brief explanation
on the two main new objects of this paper,
that we will progressively describe
in the sequel : $(\kappa, \lambda)$-\emph{capacities}
and \emph{soft-measures}.
To understand their meaning, beyond their definitions,
we can start from the main formula introduced
in the context of metastable dynamics by~\cite{BEGK1}.
Given a reversible and irreducible Markov process $X : t \mapsto X(t) \in \mathcal{X}$,
and for any two disjoint and non-empty subsets $A$ and $B$ of $\mathcal{X}$, it holds
\begin{equation} \label{clipper}
	{\mathbb E}_{\nu_A}\left[
		\tau_B
	\right]
	= \frac{\mu(V_{A, B})}{C(A, B)}
	\,,
\end{equation}
where $\nu_A$ is the so-called harmonic measure on $A$
(which actually depends also on $B$),
$\tau_B$ is the hitting time of $B$,
$\mu(V_{A, B})$ is the mean value,
w.r.t. the equilibrium measure $\mu$,
of the ``equilibrium potential'' between $A$ and $B$,
and $C(A, B)$ is the capacity between $A$ and $B$.
This formula had in particular two crucial advantages.
First, it allowed to describe the metastability phenomenon
essentially only through the computation of mean hitting times.
Second, the most relevant part in the right-hand side
is the capacity  appearing in the denominator,
which has the key property of satisfying two variational principles
which, in turn, can be used to get sharp estimates just by using test functions
to obtain upper bounds and test flows to obtain lower bounds.
Using this formula one has however to cope with three interlinked difficulties,
which, depending on the considered model,
can or cannot be easily overcame:
\begin{enumerate}
\item[i)]
	the choice of the family of sets $A$ and $B$ can be delicate;
\item[ii)]
	there is in general no variational principle
	to help in estimating the mean potential at the numerator of the right-hand side;
\item[iii)]
	the harmonic measure $\nu_A$ is in general
	very different from the natural measures
	associated with metastability, say $\mu_\mathcal{R}$
	or $\mu^*_\mathcal{R}$.
\end{enumerate}
Let us rapidly explain these three points.
Formula~\eqref{clipper}
is in its very nature associated with the Markov process
{\em stopped at time $\tau_B$}.
Our previous discussion explains
why it will not be sufficient just to choose $B = \mathcal{X} \setminus \mathcal{R}$.
Thus, in general, one has to consider a family of sets $B$
that are ``deep inside $\mathcal{X} \setminus \mathcal{R}$'',
and for a symmetric reason, the family of sets $A$ should be chosen ``deep inside $\mathcal{R}$'' too.
But then, the deeper are these chosen sets, the harder turns
the estimation of the mean potential.
Moreover, while  $\mu_\mathcal{R}$ and $\mu^*_\mathcal{R}$
are usually concentrated deep inside $\mathcal{R}$ (and $A$),
the harmonic measure $\nu_A$ of formula ~\eqref{clipper} has support on the {\em border} of $A$.
In general, this makes  uneasy the comparison between the Markov process
started from $\nu_A$ and the system started from a ``metastable equilibrium''.
We point out that this last difficulty
is actually the exact counterpart in $A$
of the fact that~\eqref{clipper} deals with a Markov process stopped in $B$
(this possibly not obvious fact can be well understood
by looking at the proof of~\eqref{clipper}).

The two objects that we introduce in this work, are partially intended
to deal with these difficulties.
The $(\kappa, \lambda)$-capacities
are capacities computed in an extended network
that is associated with a Markov process
{\em stopped at rate $\lambda$ in $B$}
and for which $\kappa$ plays a symmetrical role in $A$
(just like, when discussing~\eqref{clipper},
we noted that the fact that $\nu_A$
was concentrated on the border of $A$
was the counterpart of the fact
that~\eqref{clipper} was dealing with
a process stopped in $B$).
We will then be able to build on~\eqref{clipper}
with a Markov process that {\em can} penetrate $B$.
This will allow to simply choose $B = \mathcal{X} \setminus \mathcal{R}$,
rather than a family of subsets of $\mathcal{X} \setminus \mathcal{R}$,
and to compute the ``mean transition times'' from metastable to stable states
by estimating  the $(\kappa, \lambda)$-capacities.
Symmetrically, the parameter $\kappa$ will be used to deal
with measures that are concentrated deep inside $A = \mathcal{R}$.
In addition, the parameter $\lambda$ will be used to interpolate
between the restricted ensemble $\mu_\mathcal{R}$ (at $\lambda = 0$)
and the quasi-stationary measure $\mu^*_\mathcal{R}$ (at $\lambda = + \infty$).
These interpolating measures will be our soft-measures;
they are the quasi-stationary measures of the trace on $\mathcal{R}$
of the process killed outside $\mathcal{R}$ at rate~$\lambda$.
In some sense, they are intended to keep the idea of characterizing
metastability through the computation of mean hitting times,
for which we can benefit of the classical potential theory
and of its variational principles.
Though formula~\eqref{clipper} will be used  in our proofs,
we will derive new (asymptotic) equations expressing these mean hitting times
in terms of quantities that satisfy two-sided variational principles only,
and do not involve mean potentials.

Finally, we stress that the difficulties arising
when using~\eqref{clipper}
to describe a metastable dynamics
will not magically disappear by using soft-measures
or $(\kappa, \lambda)$-capacities instead of standard capacities.
They are actually deferred into the estimation of the local relaxation times,
called $\gamma_\mathcal{R}^{-1}$ and $\gamma_{\mathcal{X} \setminus \mathcal{R}}^{-1}$
in the sequel.
However, in doing so, we can benefit from the huge  mathematical literature dealing with
the computation of rates of convergence to equilibrium.
In this paper we will also prove a new Poincar\'e inequality, adding one more tool in this respect.
And at this point, we should stress that the hypotheses that these {\em local} relaxation times
should satisfy in order to apply our results,  do not require sharp estimates.
Rough estimates will be enough to find large windows
in which choosing our parameters $\kappa$ and $\lambda$
to obtain, through the use of variational principles, sharp estimates
on the {\em global} relaxation time ($\gamma^{-1}$ in the sequel).

\section{Model and results}\label{model_and_results}

\subsection{Quasi-stationary measure and restricted ensemble}

We consider a continuous-time Markov process $X$
on a finite set $\cX$ with generator defined by
\begin{equation}
\cL f(x) = \sum_{y\in \cX} p(x,y) (f(y) - f(x)) \label{mk}
\end{equation}
for $x$ in $\cX$ and $f:\cX\rightarrow {\mathbb R}$,
and where $p$ is such that
\begin{equation} \label{lampione}
	\sum_y p(x,y) = 1
	\,.
\end{equation}
Since $\cX$ is finite, any generator can be written
like in~(\ref{mk}) up to time rescaling.
We assume that $X$ is irreducible and reversible
with respect to some probability measure $\mu$,
we denote by $\langle\cdot\,,\cdot\rangle$
the scalar product in $\ell^2(\mu)$,
by $\|\cdot\|$ the associated 2-norm,
by $\cD$ the Dirichlet form defined
by
\begin{equation}
\cD(f)  = \langle f,-\cL f\rangle
=\frac{1}{2}\sum_{x,y\in\cX}c(x,y)\left[f(x) - f(y)\right]^2
\end{equation}
where each conductance $c(x,y)$ is equal to
\begin{equation}
c(x,y) = \mu(x)p(x,y),
\end{equation}
and by $\gamma$ the spectral gap
\begin{equation}
\gamma = \min_{\var_\mu(f)\neq 0}\frac{\cD(f)}{\var_\mu(f)}.
\end{equation}
For $\cR\subset\cX$ we define
in each $x\in\cR$ the escape probability (or rate)
\begin{equation}
e_\cR(x) = \sum_{y\not\in\cR} p(x,y)
\end{equation}
and we denote by $X_\cR$
the {\em reflected process}
(or {\em restricted process})
with generator given by
\begin{equation}
\cL_\cR f(x) = \sum_{y\in \cR} p_\cR(x,y) (f(y) - f(x))
\end{equation}
for $x$ in $\cR$ and $f:\cR\rightarrow {\mathbb R}$,
and where, for all $x$, $y$ in $\cR$,
\begin{equation}
p_\cR(x,y) = \left\{\begin{array}{ll}
    p(x,y) & \mbox{if $x\neq y$,}\\
    p(x,x) + e_\cR(x) & \mbox{if $x = y$}.
  \end{array}\right.
\end{equation}
We will only consider subsets $\cR$
such that both $X_\cR$ and $X_{\cX\setminus\cR}$
are irreducible and we note that $X_\cR$
inherits from $X$ the reversibility property
with respect to the restricted ensemble
\begin{equation}
  \mu_\cR = \mu(\cdot |\cR).
\end{equation}
We identify $\ell^2(\mu_\cR)$
with the subset of $\ell^2(\mu)$ of
functions $f:\cX\rightarrow{\mathbb R}$
such that $f|_{\cX\setminus\cR} \equiv 0$
and we denote by $\langle\cdot\,,\cdot\rangle_\cR$,
$\|\cdot\|_\cR$, $\cD_\cR$, $c_\cR(x,y)$ and $\gamma_\cR$
the associated scalar product, 2-norm,
Dirichlet form, conductances for $x$, $y$ in $\cR$
and spectral gap.

We denote by $p_\cR^*$
the sub-Markovian kernel  on $\cR$
such that, for all $x$, $y$ in $\cR$,
\begin{equation}
p_\cR^*(x,y) = p(x,y).
\end{equation}
We know from~\cite{DSe} and the Perron-Frobenius theorem
that there exists $\phi_\cR^*> 0$ such that
$1-\phi_\cR^*$ is the spectral radius of $p_\cR^*$
and that there is a unique {\em quasi-stationary measure} $\mu_\cR^*$
such that $\mu_\cR^*p_\cR^* = (1-\phi_\cR^*)\mu_\cR^*$.
In addition we have, for all $x$, $y$ in $\cR$ and $t\geq 0$,
with $\tau_{\mathcal{X} \setminus \mathcal{R}}$ the exit time from $\mathcal{R}$,
i.e., the hitting time of $\mathcal{X} \setminus \mathcal{R}$,
\begin{eqnarray}
&\lim_{t\rightarrow+\infty} P_x(X(t) = y | \tau_{\cX\setminus\cR} > t)
= \mu_\cR^*(y), &\label{Yaglom}\\
&P_{\mu_\cR^*}(\tau_{\cX\setminus\cR} > t) = e^{-\phi_\cR^* t},&\\
&\mu_\cR^*(e_\cR) = \phi_\cR^*.\label{gioia}
\end{eqnarray}
The limit in~\eqref{Yaglom}
is called a {\em Yaglom limit}
after Yaglom showed the existence
of such limits in the case of branching processes~\cite{Y}.
In our context of  finite state spaces,
the existence of such a limit, that does not depend
on the starting point $x$, simply follows from the Perron-Frobenius theorem.
In Sections~\ref{soft_therm_trans_mix}
and~\ref{soft_p} these properties
will be rederived in a slightly more general context.

Our first result states that if $1/\phi_\cR^*$, the mean exit time
for the system started in $\mu_\cR^*$, is large with respect to $1/\gamma_\cR$,
the relaxation time of the reflected process, then the quasi-stationary measure
$\mu_\cR^*$ is close to the restricted ensemble $\mu_\cR$.
This is similar to Lemma 10 (b) in \cite{AB}.
More precisely,  for all $x$ in $\cR$, let us define
\begin{eqnarray}
\varepsilon_\cR^* &=& \frac{\phi_\cR^*}{\gamma_\cR}\\
h_\cR^*(x) &=& \frac{\mu_\cR^*(x)}{\mu_\cR(x)}
\end{eqnarray}
and notice that $h_\cR^*$ is a right eigenvector of $p_\cR^*$ with eigenvalue
$1-\phi_\cR^*$. We prove the following.
\begin{proposition}\label{prop:varH}
If $\e_\cR^*<1$, then
\begin{equation}
\var_{\m_\cR}(h_\cR^*)
= \|h_\cR^* - \1_\cR\|_\cR^2
\leq \frac{\e_\cR^*}{1-\e_\cR^*}
\end{equation}
\end{proposition}
\begin{proof}
See Section~\ref{quasi_restr_p_1}.
\end{proof}
\begin{remark}
When proving that $\varepsilon_\cR^*$
goes to~0 in some asymptotic regime
(for example  when the cardinality of the configuration space
goes to infinity like in~\cite{CGOV}, when some parameter
of the dynamics goes to~0 like the temperature
in~\cite{NS1} or when both happen like in~\cite{GdHNOS})
one has to give upper bounds on $\phi_\cR^*$
and lower bounds on $\gamma_\cR$.
$\phi_\cR^*$ satisfies a variational principle
through which one can get such upper bounds
using suitable test functions.
In particular, since one can often easily compute
with $\mu_\cR$, and $e_\cR$ is often explicit,
one can usually estimate
\begin{equation}\label{caramellino}
\phi_\cR = \mu_\cR(e_\cR)
\end{equation}
and then bound $\phi_\cR^*$ with $\phi_\cR$.
In some cases, for example in the low-temperature
regime, this estimate will already be good enough.
More generally and precisely,
we have the following lemma, that we prove
in Section \ref{quasi_restr_p_2}.
\end{remark}
\begin{lemma}\label{lem:phiphi}
$\displaystyle
\phi_{\cR}^* = \min_{
	\stackrel{\scriptstyle f \neq 0}
		{\scriptstyle  f_{|_\XR} = 0}
	}
\frac{\cD(f)}{\|f\|^2}
\leq \frac{1}{\E_{\mu_\cR}\left[\tau_{\cX \setminus \cR}\right]}
\leq \phi_{\cR}$.
\end{lemma}
\noindent Lower bounds on $\gamma_\cR$
can be more difficult to obtain.
However we note, first, that rough lower bounds
will often be sufficient to our ends, second,
that the new Poincar\'e inequality we will prove in this paper
(Theorem~\ref{th:gap&cap}) can be used to this purpose
(see Section~\ref{naso}).

As a consequence of this first result we can control
the convergence rate of the Yaglom limit in~(\ref{Yaglom}).
We note that, by the reversibility of $X$
with respect to $\mu$, $p_\cR^*$ is a self-adjoint operator on $\ell^2(\m_\cR)$
and has real eigenvalues.
By the Perron-Frobenius theorem, this implies the existence of a spectral gap
$\g_\cR^*>0$ equal to the difference between
the first and the second largest eigenvalue of $p_\cR^*$.
\begin{proposition}\label{prop:gaps}
If $\e_\cR^*<\frac{1}{3}$, then
\begin{equation}\label{gaps}
\frac{1}{\g_\cR^*}\leq \frac{1}{\g_\cR}\left\{\frac{1-\e_\cR^*}{1-3\e_\cR^*}\right\}\,.
\end{equation}
\end{proposition}
\begin{proof}
See Section~\ref{quasi_restr_p_3}.
\end{proof}
\begin{remark}
Since, after the static study made in~\cite{GSSV},
we intend to apply our results to the dynamical study
of the cavity algorithm introduced in \cite{ISS},
for which finite-volume effects are of first importance,
we need to give asymptotics
with quantitative control of corrective terms.
This produces quite long formulas
and to simplify the reading we put between curly brackets
any terms that go to 1 in a suitable asymptotic regime.
\end{remark}
Then we set
\begin{equation}\label{zeta}
	\z_\cR^*
	= \min_{x\in\cR}\m_\cR (x){h_\cR^*}^2(x)
	= \min_{x\in\cR}\m_\cR^* (x){h_\cR^*}(x)
	\,,
\end{equation}
which is the mass of the smallest atom
of the measure $\mu_{\mathcal{R}}$ biased by ${h_{\mathcal{R}}^*}^2$,
and define, if $\e_\cR^*<1 / 3$ and for any $\d\in]0,1[$,
\begin{equation}\label{Tdelta}
T_{\d,\cR}^*= \frac{1}{\g_\cR}
\left(\ln\frac{2}{\d(1-\d)\z_\cR^*}\right)
\left\{
	\left(
		1+\sqrt{\frac{\e_\cR^*}{1-\e_\cR^*}}
	\right) \left(
		\frac{
			1 - \varepsilon^*_{\mathcal{R}}
		}{
			1 - 3 \varepsilon^*_{\mathcal{R}}
		}
	\right)
\right\}
\end{equation}
to get point-wise mixing estimates for Yaglom limits.
\begin{theorem}[Mixing towards quasi-stationary measure]\label{th:dinamica}
If $\e_\cR^*<1/3$, then for all $x,y\in\cR$ and $\d\in]0,1[$,
\begin{equation}\label{dinamica}
\left|\frac{\P_x(X(t)=y\,|\, \t_{\XR}>t)}{\m_\cR^*(y)}-1\right|< \d
\quad \mbox{ as soon as }\quad t>T_{\d,\cR}^*
\end{equation}
\end{theorem}
\begin{proof}
See Section~\ref{sec:recurrence}.
\end{proof}
\begin{remark}
In words, this says that either the system leaves $\cR$
before time $T_{\delta,\cR}^*$, or it is described after that time
by $\mu_\cR^*$ in the strongest possible sense.
This theorem is useful only if one can provide upper bounds
on $T_{\delta,\cR}^*$.
Bounding $T_{\delta, \mathcal{R}}^*$ depends on the control
we have on $\varepsilon^*_{\mathcal{R}}$ and on this new parameter $\zeta_\cR^*$.
As far as the latter  is concerned,
we note that it only appears in the formula
through its logarithm.
Crude or very crude estimates of $\zeta_\cR^*$
will then often be sufficient.
One has for example the following lemma.
\end{remark}
\begin{lemma}\label{lem:zeta}
\par\noindent
\begin{itemize}
\item[i)]
With $\zeta_\cR = \min_{x \in \cR} \mu_\cR(x)$ and $\alpha_\cR = \max_{x \in \cR} e_\cR(x)$,
it holds
\begin{equation}\label{zidane}
	\ln \frac{1}{\zeta_\cR^*}
	\leq \ln \frac{4}{\zeta_\cR}
	+ \frac{\alpha_\cR}{\gamma_\cR} \left[
		\ln \frac{4\varepsilon^*_\cR}{(1 - \varepsilon^*_\cR) \zeta_\cR}
	\right]_+
	\,,
\end{equation}
where the brackets $[\cdot]_+$ stand for the positive part.
\item[ii)]
If $p(x,x)>0$ for all $x\in\cR$, then
\begin{equation}\label{tartaruga}
	\ln \frac{1}{\zeta_\cR^*}
	\leq \ln \frac{1}{\min_{x\in\cR} {\mu_\cR^*}^2(x)}
	\leq 2\D_\cR D_\cR
	\,,
\end{equation}
where  $$\ba{l}
\D_\cR= \max\{-\ln p_\cR(x,y)\,:\,p_\cR(x,y)>0\,,\,\forall x,y\in\cR\}\\
D_\cR= \min\{k\geq 0:\,p_\cR^k(x,y)>0\,,\,\forall x,y\in\cR\}\ea\,.$$
\end{itemize}
\end{lemma}
\begin{proof} See Appendix~\ref{crude}. \end{proof}
Also, since $h_\cR^*$ is superharmonic on $\cR$ with respect to $\cL$
(see Appendix~\ref{crude}), it reaches its minimum on
the internal border of $\cR$,
\begin{equation}
\partial_- \cR = \left\{x \in \cR :\: \exists y \not \in \cR, p(x, y) > 0\right\}.
\end{equation}
Then we always have
\begin{equation}
	\zeta_\cR^* \geq \left(\min_{x \in \cR} \mu_\cR(x)\right) \left(\min_{x \in \partial_- \cR} h_\cR^*(x)\right)^2,
\end{equation}
and in the special case when $\partial_- \cR$ reduces to a singleton,
this gives, by (\ref{gioia}) and (\ref{caramellino}),
\begin{equation} \label{brown}
	\zeta_\cR^* \geq \min_{x \in \cR} \mu_\cR(x)
	\left(
		\frac{\phi_\cR^*}
			{\phi_\cR}
	\right)^2
	\,.
\end{equation}
Bounding $\zeta_\cR^*$ from below
essentially reduces in this case
to giving a lower bound on $\phi_R^*$,
which is one of the main goals of this paper
(see Theorem \ref{th:phi&cap}).

We will make a special choice for the parameter $\delta$
in~(\ref{Tdelta}): we define
\begin{equation}
T_\cR^*= T_{\varepsilon_\cR^*, \cR}^*.
\end{equation}
We then have
\begin{equation} \label{delta*phi*}
\phi_\cR^* T_\cR^* \leq \e_\cR^* \left(
	\ln \frac{3}{\varepsilon_\cR^* \zeta^*_\cR}
\right) \left\{
	\left(
		1 + \sqrt{
			\frac{\varepsilon_\cR^*}{1 - \e_\cR^*}
		}
	\right) \left(
		\frac{1 - \e_\cR^*}{1 - 3\varepsilon_\cR^*}
	\right)
\right\}
\end{equation}
as soon as $\varepsilon_\cR^* < 1/3$.
We will sometimes refer in the sequel to the regime $\phi_\cR^*T_\cR^*\ll1$.
Equation (\ref{delta*phi*})
provides a sufficient and practical condition
for being in such a regime.
We close this section with a first asymptotic exponential law
in this particular regime.
\begin{theorem}[Asymptotic exit law]\label{th:tempo}
For  any  probability measure $\nu$ on $\cR$,
define $\pi_\cR(\nu)= \P_\nu(\t_{\XR}<T_\cR^*)\,$.
If $\e_\cR^*<1/3$, then,
for all $t\geq \phi_\cR^*T_\cR^*$,
$$\left\{
\ba{l}
\P_\nu (\t_{\XR}>\sfrac{t}{\phi_\cR^*})\leq
(1-\pi_\cR(\nu))e^{-t}\left\{e^{\phi_\cR^* T_\cR^*}(1+\e_\cR^*)\right\}\\
\P_\nu (\t_{\XR}>\sfrac{t}{\phi_\cR^*})\geq
(1-\pi_\cR(\nu))e^{-t}\left\{e^{\phi_\cR^* T_\cR^*}(1-\e_\cR^*)\right\}
\ea
\right.\,.\qquad\qquad
$$
\end{theorem}
\begin{proof}
See Section~\ref{exp_law_p_2}
\end{proof}
\begin{remark}
The theorem gives more than an asymptotic exponential exit law.
It says that, provided $\pi_\cR(\nu)$ converges to some limit
and in the regime $\phi_\cR^* T_R^* \ll 1$, the normalized mean exit
time $\phi_\cR^*\t_{\XR}$ converges in law to a convex combination
between a Dirac mass in~0 and an exponential law with mean~1.
\end{remark}

As an example of an application we can consider the case of the restricted ensemble.
\begin{lemma}\label{lem:piR}
It holds
\begin{equation}\label{piR}
\pi_\cR(\m_\cR)= \P_{\m_\cR}(\t_{\XR}\leq T_\cR^*)
\leq\frac{1}{2}
\sqrt{
	\frac{\varepsilon_\cR^*}
	{1 - \varepsilon_\cR^*}
} + \phi_\cR^* T_\cR^*\,.
\end{equation}
\end{lemma}
\begin{proof}
See Section~\ref{exp_law_p_3}.
\end{proof}
This shows an asymptotic exponential exit law in the regime $\phi^*_\cR T_\cR^*\ll1$ for the system started
in the restricted ensemble.

Another consequence of Theorem \ref{th:tempo}
is that, in the regime $\phi_\cR^* T_\cR^* \ll 1$,
$\mu_\cR^*$ asymptotically maximizes the mean
exit time on the set of all possible starting measures.
This can be seen in a different way by following \cite{AB}.
Consider, for any $t > 0$ the natural coupling up to time
$t \wedge \tau_\XR$ between $X$ started
from a measure $\nu$ and a process that starts from
$X(0)$, follows the law of the reflected process up to time
$t$, and then the same law as the original process.
This last process cannot escape
from ${\mathcal R}$ before $X$ and we get,
\begin{equation}
	{\mathbb E}_{\nu}\left[
		\tau_\XR
	\right]
	\leq t + \left(
		1 + \frac{e^{-\gamma_\cR t}}{\zeta_R}
	\right) {\mathbb E}_{\mu_\cR}\left[
		\tau_\XR
	\right]
	\,,
\end{equation}
with, as previously, $\zeta_\cR = \min_{x \in \cR} \mu_\cR(x)$.
Using Lemma \ref{lem:phiphi} and optimizing in $t$
one gets
\begin{equation}\label{banana}
	{\mathbb E}_\nu \left[
		\tau_\XR
	\right]
	\leq \frac{1}{\phi_\cR^*}
	\left\{
		1 + \varepsilon_\cR^* + \e_\cR^* \ln \frac{1}{\e_\cR^* \zeta_R}
	\right\}
	\,.
\end{equation}

We already mentioned that $\phi_\cR^*$ can be estimated from above
by using test functions in a variational principle (see Lemma~\ref{lem:phiphi}).
One of the questions raised in \cite{AB}
is that of upper bounds on mean exit times,
i.e., that of lower bounds for $\phi_\cR^*$.
This is the question we will now adress.

\subsection{$(\kappa,\lambda)$-capacities, mean exit times and a new Poincar\'e inequality}
In this section we introduce a new object
which extends the notion of capacity between sets.
For any $\kappa,\l>0$ and $A,B\subset \cX$,
we first extend the electrical network $(\cX,c)$,
with $c(x,y)=\m(x)p(x,y)=\m(y)p(y,x)$ for all distinct
$x,y\in\cX$, into a larger electrical network $(\tilde\cX,\tilde c)$
by attaching a dangling edge $(a,\bar a)$ with conductance
$\kappa\m(a)$ to each $a\in A$ and a dangling edge $(b,\bar b)$ with conductance
$\l\m(b)$ to each $b\in B$
(this extension is related with some Markov chain modification considered in~\cite{Mi2}.
More precisely, we add $|A|+|B|$ nodes and edges to the network by setting
$$\tilde\cX= \cX\cup\{\bar a\,:\,a\in A\}\cup\{\bar b\,:\,b\in B\}$$
and, for all distinct $\ti x, \ti y\in\ti\cX$ we define
\begin{equation}\label{conductance}
\ti c(\ti x,\ti y)=
\left\{
\ba{ll}
c(x,y) & \mbox{ if } (\ti x,\ti y)=(x,y)\in\cX\times\cX\\
\kappa\m(a) & \mbox{ if } (\ti x,\ti y)=(a,\bar a)\,\, \mbox{ for some }a\in A\\
\l\m(b) & \mbox{ if } (\ti x,\ti y)=(b,\bar b)\,\, \mbox{ for some }b\in B\\
0 & \mbox{ otherwise }
\ea
\right.\,.
\end{equation}
This extended network is naturally associated
with \emph{a family} of ``two level Markov processes''
that evolve like $X$ in $\mathcal{X}$,
``go down'' from $A$ and $B$ in $\mathcal{X}$
to $\bar A$ and $\bar B$ at rate $\kappa$
and $\lambda$ respectively, and ``go up'' from $\bar A$
and $\bar B$ to $A$ and $B$ in $\mathcal{X}$ at some rates
tuning the equilibrium measure of such processes in $\tilde{\mathcal{X}}$.
(We will use in the proof of our results
this liberty in choosing the rates of this family of processes
associated with this unique extended electrical network.)

\begin{definition}\label{def:capacity}
The  $(\kappa,\l)$-capacity, $C_\kappa^\l(A,B)$, is defined as the capacity
between the sets $\bar A$ and $\bar B$ in the electrical network
$(\ti\cX, \ti c)$, and then is given, according to Dirichlet principle, by
\begin{equation}\label{capacity}
\begin{split}
C_\kappa^\l(A,B)&=
\min_{\ti f:\ti\cX\mapsto\R}
\left\{\frac{1}{2}\sum_{\ti x,\ti y\in\ti\cX}\ti c(\ti x,\ti y)[\ti f(\ti x)-\ti f(\ti y)]^2;\,
\ti f_{|_{\bar A}}=1\,,\ti f_{|_{\bar B}}=0\right\}\\
&=\min_{f:\cX\mapsto\R}\cD(f) + \kappa\sum_{a\in A}\m(a)[f(a)-1]^2
+ \l\sum_{b\in B}\m(b)[f(b)-0]^2\\
&=\min_{f:\cX\mapsto\R}\cD(f) + \kappa\m(A)\E_{\m_A}\left[(f_{|_A}-1)^2\right]
+ \l\m(B)\E_{\m_B}\left[(f_{|_B}-0)^2\right]\,.
\end{split}
\end{equation}
\end{definition}

\no\emph{Remarks.}
\begin{itemize}
\item[i)]Since all the points of $\bar A$ and $\bar B$
are at potential $1$ and $0$ respectively in formula
(\ref{capacity}), they are electrically equivalent
and we could have defined the $(\kappa,\l)$-capacity between
$A$ and $B$ by adding just two nodes to the electrical
network $(\cX,c)$.
However, our definition with dangling edges will be more useful
in the sequel.
\item[ii)]
A $(\kappa,\l)$-capacity is in some sense easy to estimate since
it satisfies a two-sided variational principle.
On one hand, by definition, it is the infimum of some
functional, and any test function will provide an upper bound.
On the other hand it is the supremum of another functional on
flows from $\bar A$ to $\bar B$, which are antisymmetric functions
of oriented edges with null divergence in $\cX$, i.e., on functions
$\ti\psi:\ti\cX\times\ti\cX\mapsto\R$ such that for all
$x\in\ti \cX \setminus(\bar A\cup\bar B)$,
$\mbox{div}_x \ti\psi=\sum_{\ti x\in\ti\cX}\ti\psi(x,\ti x)=0$.
This is Thomson's principle that goes back to~\cite{TT}, Chapter~1,
Appendix~A (see also lecture notes~\cite{G} for a more modern presentation
or textbook~\cite{N} for the proof of an almost equivalent result).
Letting
$$\ti\cD(\ti\psi)=\sfrac{1}{2}\sum_{\ti x,\ti y\in\ti\cX}
\frac{\ti\psi(\ti x,\ti y)^2}{\tilde c(\ti x,\ti y)}\,,$$
be the energy dissipated by the flow $\ti\psi$ in the network
$(\ti\cX,\ti c)$, and $\ti\Psi_1(\bar A,\bar B)$ the set of unitary flows
from $\bar A$ to $\bar B$, that is, the set of flows $\ti\psi$
from $\bar A$ to $\bar B$ such that
\begin{equation}\label{unitflow}
\sum_{\bar a\in\bar A}\mbox{div}_{\bar a}\ti\psi=
\sum_{\bar a\in\bar A}\sum_{\ti x\in\ti\cX}
\ti\psi(\bar a,\ti x)= 1 = -\sum_{\bar b\in\bar B}\mbox{div}_{\bar b}\ti\psi=
-  \sum_{\bar b\in\bar B}\sum_{\ti x\in\ti\cX}
\ti\psi(\bar b,\ti x)\,,
\end{equation}
we have
\begin{equation}
C_\kappa^\l(A,B)=
\max_{\ti\psi\in\ti\Psi_1(\bar A,\bar B)}\ti\cD(\ti\psi)^{-1}
\,.
\end{equation}
If $A \cap B = \emptyset$
this gives
\bea\label{cuore}
\qquad\quad C_\kappa^\l(A,B)&=&
\max_{\psi\in\Psi_1(A,B)}\left(\cD(\psi)
+\sum_{a\in A}\frac{(\mbox{div}_a\psi)^2}{\kappa\m(a)}
+\sum_{b\in B}\frac{(\mbox{div}_b\psi)^2}{\lambda\m(b)}\right)^{-1}\\
&=&
\max_{\psi\in\Psi_1(A,B)}\left(\cD(\psi)
+\sfrac{1}{\kappa \mu(A)}\E_{\m_A}\left[\left(\sfrac{\mbox{div}\psi}{\m_A}\right)^2\right]
+\sfrac{1}{\lambda \mu(B)}\E_{\m_B}\left[\left(\sfrac{\mbox{div}\psi}{\m_B}\right)^2\right]\right)^{-1}\,,
\nonumber
\eea
where $\Psi_1(A,B)$ is the set of unitary flows $\psi$
from $A$ to $B$ and
$$\cD(\psi)=\sfrac{1}{2}\sum_{x,y\in\cX}
\frac{\psi(x,y)^2}{c(x,y)}.$$
Then, any test flow provides a lower bound on $C_\kappa^\l(A,B)$.
\item[iii)]
We know (\cite{N}, \cite{G}) that the infimum and supremum in (\ref{capacity}) and (\ref{cuore}),
are realized, respectively,
by the equilibrium potential
$V_\kappa^\l=\P_{(\cdot)}(\ell^{-1}_A(\s_\kappa)<\ell^{-1}_B(\s_\l))$,
where $\ell^{-1}_A$ and $\ell^{-1}_B$ are the
right continuous inverses of the local times in $A$ and $B$,
while $\s_\kappa$ and $\s_\l$ are independent exponential times with rates
$\kappa$ and $\l$, and by its associated normalized current
\begin{equation}\label{current}
-\sfrac{c\nabla V_\kappa^\l}{C_\kappa^\l(A,B)}:(x,y)\in\cX\times\cX
\longmapsto \sfrac{c(x,y)}{C_\kappa^\l(A,B)}(V_\kappa^\l(x)-V_\kappa^\l(y))\,.
\end{equation}
We will say more on such quantities in the next section.
\item[iv)]
The previous definitions and observations extend to the case when $\kappa$ and $\l$
are equal to $+\infty$. In that case we identify $\bar A$ with $A$ in the extended
network if $\kappa=+\infty$, or $\bar B$ with $B$ if $\l=+\infty$, and we
drop the infinite upper or lower index in the notation, so that, for example,
$C_\kappa (A,B)= C_\kappa^{\infty}(A,B)$.
However, when $\kappa$ and $\l$ are both equal to $+\infty$, to avoid any ambiguity
we need to require that $A\cap B=\emptyset$. In that case the notation becomes
$C(A,B)=C_{\infty}^{\infty}(A,B)$ and we recover indeed the usual notion
of capacity.
\end{itemize}

We then get sharp asymptotics on mean exit times for the system started
in the quasi-stationary measure.
\begin{theorem}[Mean exit time estimates]\label{th:phi&cap}
For all $\kappa>0$, it holds
\begin{equation}\label{phi&cap}
\sfrac{C_\kappa (\cR,\XR)}{\m(\cR)}
\left\{1-\e_\cR^*-\sfrac{\kappa}{\g_\cR}\right\}
\leq\phi_\cR^*\leq
\sfrac{C_\kappa (\cR,\XR)}{\m(\cR)}
\left\{1-\sfrac{C_\kappa (\cR,\XR)}{\kappa \m(\cR)}\right\}^{-2}
\end{equation}
\end{theorem}
\begin{proof} See Section~\ref{klc}.\end{proof}
\par\noindent{\em Remarks.}
\begin{itemize}
\item[i)] In the regime $\varepsilon_\cR^*\ll1$,
one can choose $\kappa$ such that $\phi_\cR^*\ll \kappa \ll \gamma_\cR$
and infer, by the lower bound  in (\ref{phi&cap}),
that $\kappa\gg \sfrac{C_\kappa (\cR,\XR)}{\m(\cR)}$.
In turns, this yields an asymptotical matching upper bound.
\item[ii)] Both bounds are in some sense easy to estimate since
capacities satisfy a two-sided variational principle.
Moreover, compared with the formula for mean exit time
provided by potential-theoretic techniques (see, e.g., \cite{BEGK1}),
the above inequalities  require no residual average potential estimates.
(Such estimates, as well as some harmonic measures
will only play a role in the {\em proof} of the theorem.)
\end{itemize}

Our $(\kappa,\lambda)$-capacities provide also spectral gap estimates
and a new general Poincar\'e inequality.
For $\kappa,\l>0$ and $A,B\subset \cX$ we set
\begin{equation}\label{phi(A,B)}
\phi_\kappa^\l(A,B)=\frac{C_\kappa^\l(A,B)}{\m(A)\m(B)}=\phi_\l^\kappa (B,A)\,.
\end{equation}
\begin{theorem}[Relaxation time estimates]\label{th:gap&cap}
For all $\kappa,\l>0$ and any $\cR\subset \cX$
such that $X_\cR$ and $X_{\XR}$ are
both irreducible Markov processes,
\begin{equation}\label{gap&cap}
\left\{
\ba{l}
\frac{1}{\g}\geq
\sfrac{1}{\phi_\kappa^\l(\cR,\XR)}
\left\{1-\sfrac{C_\kappa (\cR,\XR)}{\kappa\m(\cR)}-
\sfrac{C^\l(\cR,\XR)}{\l\m(\XR)} \right\}^2
\\
\frac{1}{\g}\leq
\sfrac{1}{\phi_\kappa^\l(\cR,\XR)}
\left\{1+\max\left(\sfrac{\kappa+\phi_\kappa^\l(\cR,\XR)}{\g_\cR},
\sfrac{\l+\phi_\kappa^\l(\cR,\XR)}{\g_{\XR}}\right)\right\}
\ea
\right.\,.
\end{equation}
\end{theorem}
\begin{proof} See Section~\ref{klc}.\end{proof}
\par\noindent{\em Remarks.}
\begin{itemize}
\item[i)]
Without loss of generality, we can assume $\mu(\cR)\leq\mu(\XR)$
so that, by (\ref{phi(A,B)}),
$\phi_\kappa^\lambda\leq 2 C_\kappa^\lambda(\cR,\XR)/\mu(\cR)$.
Then, as a consequence of the previous theorem and
of the monotonicity in $\kappa$ and $\lambda$
of $(\kappa,\lambda)$-capacities,
we get matching bounds on $1/\g$
in the regime
$\varepsilon_\cR^* + \varepsilon_{\XR}^* + \phi_\cR^*/\gamma_{\XR} \ll 1$.
One can indeed choose $\kappa$
such that $\phi_\cR^*\ll \kappa \ll \gamma_\cR$,
just as for Theorem \ref{th:phi&cap} (Remark i)), and $\lambda$ such that
$\phi_\cR^*, \phi_{\XR}^* \ll \lambda \ll \gamma_{\XR}$.
In addition and like previously,
all the relevant quantities can be estimated
by two-sided variational principles.
\item[ii)]
The lower bound is a generalization
of the classical isoperimetrical estimate
that is recovered for $\kappa = \lambda  = +\infty$.
\item[iii)]
The upper bound is a new Poincar\'e inequality.
This inequality, or an easy-to-derive version when one divides
the configuration space into more than two subsets,
echoes Poincar\'e inequalities given in~\cite{JSTV}.
We are not able to compare in full generality our result
with that of~\cite{JSTV} but we note that because of the presence
of some global parameter called $\gamma$ in~\cite{JSTV}
one gets generally in our metastable situation an extra {\em factor}
 $1/ \min(\gamma_\cR, \gamma_{\XR})$ by applying the results of~\cite{JSTV}.
\item[iv)]
The proof of this upper bound, when considering more than two subsets,
extends verbatim to obtain the following result.
\begin{lemma}\label{giallo}
	If $\mathcal{R}_1$, $\mathcal{R}_2$, \dots, $\mathcal{R}_m$
	form a partition of $\mathcal{X}$ for which each of the restricted processes $X_{\mathcal{R}_i}$
	is irreducible, if $\kappa_1$, $\kappa_2$, \dots, $\kappa_m$ are positive real numbers
	and if we write $\gamma_i$ for $\gamma_{\mathcal{R}_i}$ and
	$\phi(i,j) = C_{\kappa_i}^{\kappa_j}(\mathcal{R}_i, \mathcal{R}_j) / (\mu(\mathcal{R}_i) \mu(\mathcal{R}_j))$,
	then
	\begin{equation}
		\frac{1}{\gamma}
		\leq \left(
			\frac{1}{2} \sum_{i \neq j} \frac{1}{\phi(i, j)}
		\right)\left\{
			1 + \frac{
				\max_i \frac{1}{\gamma_i} \left\{
					1 + \sum_{j \neq i} \frac{\kappa_i}{\phi(i, j)}
				\right\}
			}{
				\frac{1}{2}\sum_{i \neq j} \frac{1}{\phi(i, j)}
			}
		\right\}
		\,.
	\end{equation}
\end{lemma}
\end{itemize}

\subsection{Soft measures, local thermalization, transition and mixing times}
\label{soft_therm_trans_mix}

We address now the difficulty raised by Lebowitz and Penrose.
Whatever the measure we choose to describe our metastable state,
restricted ensemble or quasi-stationary measure, it is associated
with some subset $\cR$ of the configuration space.
Then there is an ambiguity when one looks at property (b):
what is ``getting out'' of the metastable state?
One is tempted to say that it corresponds in our model to the exit
from $\cR$. But doing so we are very unlikely to modelize in any satisfactory
way property (c): we can expect that ``on the edge'',  when the system just exited $\cR$,
it has probabilities of the same order to ``proceed forward'' and thermalize in $\XR$
and to go ``backward'' and thermalize in $\cR$.
Thus we would like to define what would be a ``true escape'' from $\cR$.
Theorem~\ref{th:dinamica} suggests an answer in the regime $\phi_{\XR}^*T^*_{\XR}\ll 1$.
We could define the true escape as the first excursion of length $T_{\XR}^*$ inside $\XR$.
Since time randomization is almost always a good idea, we are led to
consider a random timer, which is independent of the dynamics and has exponential distribution
in order to keep the Markovianity of the process.
The timer starts when the dynamics exits $\cR$, but if it does not ring before returning to $\cR$,
the excursion to $\XR$ is ignored in the sense that it is not considered
a ``true escape'' from $\cR$. A ``true escape'' happens only when the timer rings during
one of the excursion outside $\cR$.
This will lead to an extension of the concept of quasi-stationary
distribution that interpolate between $\mu_{\mathcal{R}}^*$ and $\mu_{\mathcal{R}}$
and we will see (Theorem~\ref{mix} below) that the system will actually be close
to equilibrium the first time when the timer will ring during an excursion outside $\mathcal{R}$:
it will have truly escaped from metastability.

For any $A\subset \cX$ we call
\begin{equation}
\ell_A(t) = \int_0^t\1_A(X(s))ds
\end{equation}
the local time spent in $A$
up to time t and we denote by $\ell^{-1}_A$
the right-continuous inverse of $\ell_A$:
\begin{equation}
\ell^{-1}_A(t) = \inf\{s\geq 0 :\: \ell_A(s) > t\}.
\end{equation}
Recall that the process $X$ can be seen
as the process updated, according to its discrete version
with transition probability matrix $p$,
at each ring of a Poissonian clock with intensity~1.
Let us then call $\tau$ the first ringing time.
For $\sigma_\lambda$
an exponential time with mean $1/\lambda$
that is independent from $X$, we define for all $x$ and
$y$ in $\cR$
\begin{equation}
p_{\cR, \lambda}^*(x,y) = \P_x\left(X(\tau_\cR^+) = y, \ell_{\XR}(\tau_\cR^+)\leq \sigma_\lambda\right)
\end{equation}
with $\tau_\cR^+$ the return time in $\cR$ after the first clock ring,
i.e., $\tau_\cR^+ = \tau + \tau_\cR\circ\theta_\tau$ with $\theta$ the usual shift operator.
We also define, for all $x$ in $\cR$,
\begin{equation}
e_{\cR,\lambda} (x) = \P_x(\ell_{\XR}(\tau_\cR^+) > \sigma_\lambda) = 1 - \sum_{y\in\cR} p_{\cR,\lambda}^*(x,y)
\end{equation}
and for all $x$ and $y$ in $\cR$
\begin{equation}
p_{\cR,\lambda}(x, y) = \left\{\begin{array}{ll}
    p_{\cR,\lambda}^*(x,y) & \mbox{if $x\neq y$},\\
    p_{\cR,\lambda}^*(x,x) + e_{\cR,\lambda}(x) & \mbox{if $x =  y$}.
  \end{array}
\right.
\end{equation}
The Markov process $X_{\cR,\lambda}$ on $\cR$
with generator defined by
\begin{equation}
\cL_{\cR,\lambda}f(x) = \sum_{y\in \cR} p_{\cR,\lambda}(x, y)(f(y) - f(x))
\end{equation}
is reversible with respect to $\mu_\cR$ and has spectral gap
\begin{equation}
\gamma_{\cR,\lambda} = \min_{\var_{\mu_\cR}(f)\neq 0} \frac{\cD_{\cR,\lambda}(f)}{\var_{\mu_\cR}(f)}
\end{equation}
where
\begin{equation}
\cD_{\cR,\lambda} (f) = \frac{1}{2} \sum_{x,y} c_{\cR,\lambda}(x,y)(f(x) - f(y))^2
\end{equation}
with
\begin{equation}
c_{\cR,\lambda} (x,y) = \mu_\cR(x) p_{\cR,\lambda}(x, y) = p_{\cR,\lambda}(y,x) \mu_\cR(y).
\end{equation}
In addition we define
\begin{equation}
\mathcal T :=\ell_{\mathcal X\setminus \mathcal R}^{-1}(\sigma_\lambda)
\end{equation}
\begin{equation}
\tau_{\XR,\lambda} = \ell_\cR(\ell^{-1}_{\XR}(\sigma_\lambda)).
\end{equation}
We will refer to $\tau_{\XR, \lambda}$ as the {\em transition time},
since, for suitable choices of $\lambda$, this is the time spent
by the process in $\mathcal{R}$ before ``truly escaping'' from $\mathcal{R}$,
as seen by formula~\eqref{giggiola} in Theorem~\ref{mix} below.

\begin{remark}\label{referee}
While $\mathcal T$ is the global time such that the time spent in $\mathcal X\setminus \mathcal R$,
during possibly many excursions, is equal to $\sigma_\lambda$,
the time $\tau_{\mathcal X\setminus \mathcal R,\lambda}$
is the local time on $\mathcal R$ associated to $\mathcal T$.
On one hand, it may thus look natural to address the study toward the characterization
of the global time $\cT$.
On the other hand, the transition time $\tau_{\mathcal X\setminus \mathcal R,\lambda}$
is a natural generalization of the exit time $\tau_{\mathcal X\setminus \mathcal R}$,
given in such a way that when the time
$\tau_{\mathcal X\setminus \mathcal R,\l}$ is reached, not only the dynamics has exited $\mathcal R$
but it has also spent in $\mathcal X\setminus \mathcal R$ a time equal to $\sigma_\lambda$.
With a little effort, we will then derive for $\tau_{\mathcal X\setminus \mathcal R,\lambda}$
 similar results to those we have obtained for $\tau_{\mathcal X\setminus \mathcal R}$,
and in particular its asymptotic exponential law (see Theorem\ref{th:t7}).
At this point one may then think to derive information on $\cT$ by the identity
$$\mathcal T= \sigma_\lambda + \tau_{\mathcal X\setminus \mathcal R,\lambda}\,,$$
but since the random variables $\sigma_\lambda$ and  $\tau_{\mathcal X\setminus \mathcal R,\lambda}$
are not independent, this representation of $\cT$ is not immediately useful.
However, we will show  that for a suitable range of $\l$
the global time $\mathcal T$ and the local time $\tau_{\mathcal X\setminus \mathcal R,\lambda}$
are asymptotically of the same order, which is also the order of the relaxation time
(see Theorem \ref{mix} and remark below).
\end{remark}

We know by the Perron-Frobenius theorem
that the spectral radius of $p_{\cR,\lambda}^*$
is a simple positive eigenvalue that is smaller than or equal to 1 and has
left and right eigenvectors with positive coordinates.
We call it $1-\phi_{\cR,\lambda}^*$ and denote by $\mu_{\cR,\lambda}^*$
the unique associated left eigenvector that is also a probability measure
on $\cR$. We then have the following lemma.
\begin{lemma}\label{lem:softmeasure}
It holds
\begin{itemize}
\item[i)] $\phi_{\cR,\l}^*=\m_{\cR,\l}^*(e_{\cR,\l})$ ;
\item[ii)] $\P_{\m_{\cR,\l}^*}(\t_{\XR,\l}>t)=e^{-t\phi_{\cR,\l}^*}\,,\quad\forall t\geq 0$ ;
\item[iii)] $\displaystyle\lim_{t\to\infty}\P_x(X\circ \ell^{-1}_\cR(t)=y\,|\,\t_{\XR,\l}>t)=
\m_{\cR,\l}^*(y)\,,\quad\forall x,y\in\cR$.
\end{itemize}
\end{lemma}
\begin{proof}
See Section~\ref{soft_p_1}.
\end{proof}
We say that $\mu_{\cR,\lambda}^*$
is a quasi-stationary measure associated with a soft barrier,
or a soft quasi-stationary measure,
or, more simply,  a  {\em soft measure}.
Indeed, $\mu_{\cR,\lambda}^*$ is the limiting distribution
of the process conditioned to survival
when it is killed at rate $\lambda$ outside~$\cR$.
So, the hardest quasi-stationary measure associated with $\cR$,
corresponding to $\lambda = +\infty$,
is the quasi-stationary measure $\mu_\cR^*$,
while the softest measure, corresponding to $\lambda = 0 $,
is the restricted ensemble $\mu_\cR$ ($\phi_{\cR,0}^* = 0$
and $\mu_{\cR,0}^*$ is the equilibrium measure associated with
$p_{\cR, 0}^* = p_{\cR, 0}$, which is reversible with respect to $\mu_{\cR}$).
More precisely we have the following.
\begin{lemma}\label{interpolation}
The function $\l\in[0,+\infty]\mapsto \m_{\cR,\l}^*\in\ell^2(\m_\cR^*)$
is a continuous interpolation between the restricted ensemble $\m_\cR$
and the quasi-stationary distribution $\m_\cR^*$.
In particular, for any $\l_0\in[0,+\infty]$ and $y\in\cR$, we have
\begin{equation}
\lim_{\l\to\l_0}\m_{\cR,\l}^*(y)=\m_{\cR,\l_0}^*(y)
\end{equation}
and for all $x\in\cR$ it holds the  limit commutation property
\begin{equation}\label{commutativity}
\lim_{\l\to\l_0}\lim_{t\to\infty} \P_x(X\circ \ell^{-1}_\cR(t)=y\,|\,\t_{\XR,\l}>t)
=
\lim_{t\to\infty}\lim_{\l\to\l_0}\P_x(X\circ \ell^{-1}_\cR(t)=y\,|\,\t_{\XR,\l}>t)\,.
\end{equation}
\end{lemma}
\begin{proof} See Section~\ref{soft_p_2}.\end{proof}
Analogously to what was done in the case $\lambda = +\infty$
we set $\varepsilon_{\cR,\lambda}^* = \phi_{\cR,\lambda}^*/ \gamma_{\cR,\lambda}$,
$h_{\cR,\lambda}^* = \mu_{\cR,\lambda}^*/\mu_{\cR}$ and we call $\gamma_{\cR,\lambda}^*$
the gap between the largest and the second eigenvalue of $p_{\cR,\lambda}^*$
(since $p_{\cR,\lambda}^*$ is self-adjoint with respect to $\langle\cdot,\cdot\rangle_\cR$
it has only real eigenvalues). We also define $\phi_{\cR,\lambda} = \mu_\cR(e_{\cR,\lambda})$.
\begin{proposition}\label{prop:contin-gap}
The parameters $\g_{\cR,\l}$, $\phi_{\cR,\l}^*$, $\e_{\cR,\l}^*$ and $\phi_{\cR,\l}$
depend continuously on $\l$. In addition, when $\l$ decreases to $0$, so do
$\phi_{\cR,\l}^*$, $\e_{\cR,\l}^*$ and $\phi_{\cR,\l}$, while $\g_{\cR,\l}$ increases.
\end{proposition}
\begin{proof} See Section~\ref{soft_p_3}.\end{proof}
The proofs of Sections~\ref{quasi_restr_p} carry over this more general
setup, and we get, by defining the analogous $T^*_{\delta, \mathcal{R}, \lambda}$
and $\alpha_{\mathcal{R}, \lambda}$ (while $\zeta_R$, which is associated
with $\mu_{\mathcal{R}}$ rather than $\mu^*_{\mathcal{R}}$, has no ``$\lambda$-extension''),
the following theorem.
\begin{theorem}[Mixing towards soft measures]\label{th:generalization}
For all $\l\geq0$, $\phi_{\cR,\l}^*\leq\phi_{\cR}^*$, $\g_{\cR,\l}\geq\g_\cR$
and $\e_{\cR,\l}^*\leq\e_{\cR}^*$, Proposition \ref{prop:varH}, Proposition \ref{prop:gaps},
Theorem \ref{th:dinamica} and Lemma \ref{lem:zeta} hold with an extra index $\l$
and writing $X\circ \ell^{-1}_\cR$ instead of $X$.
\end{theorem}
\begin{remark}
By continuity and monotonicity, the hypotheses
$\e_{\cR,\l}^*<1$ and $\e_{\cR,\l}^*<1/3$ are always
satisfied for $\l$ small enough.
\end{remark}
We are now ready to deal with local thermalization:
we will identify a ``short'' time scale on which,
for any given starting point,
the system will relax towards a mixture
of ``local equilibria'' that are our quasi-stationary measures
with soften barriers.

For a given $\kappa\geq 0$, let $\sigma_\kappa$ be
an exponential time with mean $1/\kappa$
which is independent from $X$ and from $\s_\l$.
We think to $\s_k$ as to the random time which enters in  the construction of soft measures over $\XR$,
in the same way the random time $\s_\l$ entered in the construction of soft measure over $\cR$.
We define inductively, for $\kappa, \lambda \geq 0$, the stopping times $\t_i$ for $i\geq 0$:
\begin{eqnarray}
\t_0 &=& 0,\\
\t_1 &=& \ell^{-1}_\cR(\sigma_\kappa)\wedge \ell^{-1}_{\XR}(\sigma_\lambda),\\
\t_{i+1} &=& \t_i + \t_1\circ\theta_{\tau_i}
\end{eqnarray}
Then for $\delta \in (0,1)$ we call $i_0$ the smallest $i\geq 1$ such that one of the two following conditions
holds,
\begin{eqnarray}
i) && X(\t_i)\in\cR \mbox{ and } \ell_\cR(\t_i) - \ell_\cR(\t_{i-1}) > T_{\delta,\cR,\lambda}^*,\label{b1}\\
ii)&& X(\t_i)\not\in\cR \mbox{ and } \ell_{\XR}(\t_i) - \ell_{\XR}(\t_{i-1}) > T_{\delta,\XR,\k}^*,\label{b2}
\end{eqnarray}
and we set $\t_\delta = \t_{i_0}$.
\begin{theorem}[Local thermalization]\label{th:t6}
For any $\d\in(0,1)$ and any probability measure $\nu$
on $\cX$, if $\e_{\cR,\l}^*<1/3$ and $\e_{\XR,\kappa}^*<1/3$, then
\begin{equation}\label{t6prima}
\max\left(
\max_{x\in\cR}\left|\frac{\P_\nu(X(\t_\d)=x\,|\,X(\t_\d)\in\cR)}{\m_{\cR,\l}^*(x)}-1\right|,
\max_{x\in\XR}\left|\frac{\P_\nu(X(\t_\d)=x\,|\,X(\t_\d)\not\in\cR)}{\m_{\XR,\kappa}^*(x)}-1\right|
\right)<\d\,.
\end{equation}
Moreover if
$\xi= \max\left(e^{\kappa T_{\d,\cR,\l}^*}-1, e^{\l T_{\d,\XR,\kappa}^*}-1 \right) < 1$,
it holds
\begin{equation}\label{t6seconda}
\P_\nu\left(\t_\d>t\left(\sfrac{1}{\kappa} +\sfrac{1}{\l}\right) \right)
\leq e^{-t}\left\{\frac{1}{1-\xi} \right\}\,.
\end{equation}
\end{theorem}
\begin{proof} See Section~\ref{soft_p_4}.\end{proof}
\begin{remark}
For $\kappa$ and $\l$ small enough, we have $\e_{\cR,\l}^*<1/3$
and $\e_{\XR,\l}^*<1/3$. Then, when $\kappa$ and $\l$ decrease
to $0$, we have non-increasing upper bounds on $T_{\d,\cR,\l}^*$
and $T_{\d,\XR,\k}^*$. As a consequence, the condition $\xi<1$
will always be satisfied for $\kappa$ and $\l$ small enough
and the theorem says that starting from any configuration the
system is close to a random mixture of two states ($\mu_{\cR,\lambda}^*$
and $\mu_{\XR,\kappa}^*$, close to $\mu_\cR$ and $\mu_{\XR}$ respectively)
after a time of order $T_{\delta,\cR,\lambda}^* + T_{\delta,\XR,\kappa}^*$.
\end{remark}

As previously we  make special choices for the parameter $\delta$ and we set
\begin{equation}
T_{\cR,\lambda}^* = T_{\e_{\cR,\lambda}^*,\cR,\lambda}^*\quad\mbox{and} \quad
T_{\XR,\k}^*= T_{\e_{\XR,\k}^*,\XR,\k}^*
\end{equation}
We then have, as soon as $\e_{\cR,\l}^* < 1/3$,
\begin{equation}
\phi_{\cR,\l}^* T_{\cR,\l}^*
 \leq \e_{\cR,\l}^* \left(
	\ln \frac{3}{\varepsilon_{\cR,\l}^* \zeta^*_{\cR,\l}}
\right) \left\{
	\left(
		1 + \sqrt{
			\frac{\varepsilon_{\cR,\l}^*}{1 - \e_{\cR,\l}^*}
		}
	\right) \left(
		\frac{1 - \e_{\cR,\l}^*}{1 - 3\varepsilon_{\cR,\l}^*}
	\right)
\right\}
\,.
\end{equation}

Now the proofs of Section~\ref{exp_law_p} carry over this more general setup and we get
asymptotic exponential laws for the transition time $\tau_{\XR,\lambda}$.
\begin{theorem}[Asymptotic transition law]\label{th:t7}
For all $\l\geq 0$, Theorem \ref{th:tempo}, Lemma~\ref{lem:piR}
and inequality~(\ref{banana})
hold with an extra index $\l$.
\end{theorem}
We can also give sharp estimates
on the mean transition time and asymptotics of the mixing time.
\begin{theorem}[Mean transition time estimates]\label{th:t8}
For all $\kappa, \lambda>0$,
setting $\phi_\kappa^\l=\phi_\kappa^\l(\cR,\XR)$ (recall~(\ref{phi(A,B)})),
it holds
\begin{equation}
\left\{
\ba{l}
	\phi_{\cR,\l}^*
	\geq
	\frac{C_\kappa^\l(\cR,\XR)}{\m(\cR)}\left\{
		\frac{1 - \mu(\mathcal{R}) - 2 \phi^*_{\mathcal{R}, \lambda} / \lambda}{1 - \mu(\mathcal{R})}
	\right\}
	\left\{
		1 - \max\left(
			\frac{\kappa+\phi_\kappa^\l}{\g_\cR}, \frac{\l+\phi_\kappa^\l}{\g_{\XR}}
		\right)
	\right\} \,, \\
	{\phi^*_{\mathcal{R}, \lambda}}
	\leq
	\frac{
		C_\kappa^\lambda(\mathcal{R}, \mathcal{X}\setminus\mathcal{R})
	}{
		\mu(\mathcal{R})
	}\left\{
		1 + \varepsilon^*_{\mathcal{R}, \lambda}
		+ \varepsilon^*_{\mathcal{R}, \lambda} \ln \frac{
			1
		}{
			\varepsilon^*_{\mathcal{R}, \lambda} \zeta_{\mathcal{R}}
		}
		+ \frac{
			\phi^*_{\mathcal{R}, \lambda}
		}{
			\kappa
		}
	\right\}
	\,.
\ea
\right.
\end{equation}
\end{theorem}
\begin{proof}
See Section~\ref{soft_p_5}.
\end{proof}
\par\noindent
\emph{Remarks}
\begin{itemize}
\item[i)] In the regime
	$\varepsilon_\cR^* + \varepsilon_{\XR}^* + \phi_\cR^*/\gamma_{\XR} \ll 1$
	and assuming $\mu(\cR)\leq \mu(\XR)$
	one can choose $\kappa$ and $\lambda$ in such a way that
	$\phi_\cR^*\ll \kappa \ll \gamma_\cR$ and
	$\phi_\cR^*, \phi_{\XR}^*\ll \lambda \ll \gamma_{\XR}$,
and	then we get matching bounds provided
	$\varepsilon_{\cR,\lambda}^*\ll \ln(1/\zeta_{\mathcal{R}})$.
	Once again, all the relevant quantities can be estimated
	via a two-sided variational principle.
\item[ii)] This logarithmic term in the upper bound looks spurious.
	An upper bound without such a term should hold
	but we were not able to derive it.
\end{itemize}

\begin{theorem}[Mixing time asymptotics]\label{mix}
For $\kappa,\l > 0$ and any $x \in \mathcal{X}$,
we define $\cT = \ell^{-1}_{\XR}(\sigma_\lambda)$
and $\nu_x = \P_x(X(\cT) = \cdot)$. Then,
if $\varepsilon^*_{\mathcal{X} \setminus \mathcal{R}, \kappa} < 1 / 3$
\begin{eqnarray}
\|\nu_x - \mu_{\XR}\|_{TV} &\leq & \frac{1}{2}\varepsilon_{\XR, \kappa}^* + \lambda T_{\XR, \kappa}^*, \label{giggiola}\\
\|\nu_x - \mu\|_{TV}
&\leq & \mu(\cR) + \sqrt{
	\frac{
		\varepsilon_{\XR, \kappa}^*
	}{
		1 - \varepsilon^*_{\mathcal{X} \setminus \mathcal{R}, \kappa}
	}
} + \lambda T_{\XR,\kappa}^*
\,.
\end{eqnarray}
In addition, if
\begin{equation}
\eta =
\mu(\cR)
+ 2 \left(
	\sqrt{
		\frac{
			\varepsilon_{\XR, \kappa}^*
		}{
			1 - \varepsilon^*_{\mathcal{X} \setminus \mathcal{R}, \kappa}
		}
	} + \lambda T_{\XR,\kappa}^*
\right)
< \frac{1}{2}
\,,
\end{equation}
then, with
\begin{equation}
t_{\rm mix} = \inf_{t\geq 0}\left\{
		\max_{x\in\cX}\|\P_x(X(t) = \cdot) - \mu\|_{\mbox{{\tiny TV}}}
	\leq \frac{1}{2}\left(
		\eta + \frac{1}{2}
	\right)
\right\}
\,,
\end{equation}
we have
\begin{equation}
t_{\rm mix} \leq \frac{2}{
	\phi_{\cR,\lambda}^* \left(
		\frac{1}{2} - \mu(\mathcal{R})
	\right)
}
\left\{
	1
	+ \varepsilon^*_{\mathcal{R}, \lambda}
	+ \varepsilon^*_{\mathcal{R}, \lambda} \ln \frac{1}{\varepsilon^*_{\mathcal{R}, \lambda} \zeta_R}
	+ \frac{\phi^*_{\mathcal{R}, \lambda}}{\lambda}
\right\}
\,.
\end{equation}
\end{theorem}
\begin{proof} See Section~\ref{soft_p_6}.\end{proof}
\begin{remark}
The theorem makes sense
in the regime $\varepsilon_{\cR}^* + \varepsilon^*_{\XR} + \phi_{\cR}^*/\gamma_{\XR} \ll 1$.
One can then choose~$\lambda$ such that $\phi_{\cR, \lambda}^*\ll \lambda \ll \gamma_{\XR,\kappa}$.
If $\lambda T^*_{\XR, \kappa} \ll 1$ then $(1 + 2 \eta)/ 4$
can be made as close as $(1 + 2 \mu(\mathcal{R})) / 4 < 1 / 2$ as we want.
If $\varepsilon^*_{\mathcal{R}, \lambda} \ln (1 / \zeta_{\mathcal{R}}) \ll 1$,
then the theorem provides the correct order for the mixing time,
since the spectral gap goes like $\phi_{\cR,\lambda}^*/\mu(\XR)$
and $\mu(\XR)\geq 1/2$.
\end{remark}

Let us finally summarize our results.
To have a mathematical model of the metastability
phenomenon described by properties (a)-(c),
we first consider a reversible Markov process on a finite state space~$\mathcal{X}$,
and a subset $\mathcal{R}$ of $\mathcal{X}$ such that $\mu(\mathcal{R}) < \mu(\mathcal{X} \setminus \mathcal{R})$,
with $\mu$ the equilibrium measure of the process.
We the describe metastable states by soft measures associated
with $\cR$  in the regime
$\varepsilon_\cR^* + \varepsilon_{\XR}^* + \phi_\cR^*/\gamma_{\XR} \ll 1$.
In this regime all soft measures are close to the restricted ensemble
(Theorem~\ref{th:generalization}).
If we choose $\kappa$ and $\lambda$ such that
$\phi_\cR^* \ll \kappa \ll \gamma_\cR$ and
$\phi_\cR^*, \phi_{\XR}^* \ll \lambda \ll \gamma_{\XR}$
then we can show
\begin{itemize}
\item[i)]
  local thermalization towards the soft measure
  $\mu_{\cR,\lambda}$ or $\mu_{\XR,\kappa}$ starting from any configuration
  in $\cX$ and on a {\em short time scale}
  $\frac{1}{\kappa} + \frac{1}{\lambda}$ (Theorem~\ref{th:t6}),
\item[ii)]
  exponential asymptotic transition time on a {\em long time scale}
  $\frac{1}{\phi_{\cR,\lambda}^*}\sim \frac{\mu(\cR)}{C_\kappa^\lambda(\cR,\XR)}$
  (Theorems~\ref{th:t7} and~\ref{th:t8}),
\item[iii)]
  return time to metastable state on a {\em still longer time scale}
  $\frac{1}{\phi_{\XR,\kappa}^*}\sim \frac{\mu(\XR)}{C_\kappa^\lambda(\cR,\XR)}$
  (Theorem~\ref{th:t8} applied to $\XR$ in place of $\cR$).
\end{itemize}
In addition relaxation and mixing times are of the same order as the mean
transition time (Theorems~\ref{th:gap&cap} and~\ref{mix})
- in particular the relaxation time has the same exact asymptotic up to
a factor $\mu(\XR)$ - while exit times are on long, but generally shorter,
time scale (Theorem~\ref{th:phi&cap}). And we note once again, that all relevant
quantities can be estimated via two-sided variational principles.

\section{Analysis in $\ell^2(\mu_\cR)$}\label{quasi_restr_p}
\subsection{Proof of Proposition~\ref{prop:varH}}\label{quasi_restr_p_1}
We recall that the reflected process $X_\cR$ is reversible w.r.t. $\m_\cR$ with spectral gap $\g_\cR$.
In particular, for any function $f\in\ell^2(\m_\cR)$,
we have the Poincar\'{e} inequality
$\var_{\m_\cR}(f)\leq \frac{1}{\g_\cR}\cD_{\cR}(f)$, where $\cD_\cR(f)$ is the
Dirichlet form of $f$ given by
\begin{equation}\label{D}
\cD_\cR(f)=\langle f, -\cL_\cR f\rangle_\m =
\sum_{x,y\in\cR}\m_\cR(x)f(x)(\d_x(y)-p_\cR(x,y))f(y)\,.
\end{equation}
Applying the Poincar\'{e} inequality to $h_\cR^*$, and
 exploiting the definition of $p_\cR$ and $p_\cR^*$, we get
\begin{equation}\label{varH1}
\begin{split}
\var_{\m_\cR}(h_\cR^*)&\leq
\frac{1}{\g_\cR}\cD_\cR(h_\cR^*)=\frac{1}{\g_\cR}
\sum_{x,y\in\cR}\m_\cR(x)h_\cR^*(x)\left(\d_x(y)-p_\cR(x,y)\right)h_\cR^*(y)\\
&=
\frac{1}{\g_\cR}\left(\m_\cR({h_\cR^*}^2)- \sum_{x,y\in\cR}\m_\cR^*(x)p_\cR(x,y)h_\cR^*(y)\right)\\
&=
\frac{1}{\g_\cR}\left(\m_\cR({h_\cR^*}^2)- \sum_{x,y\in\cR}\m_\cR^*(x)(p(x,y)+\d_x(y)e_\cR(x))
h_\cR^*(y)\right)\\
&\leq
\frac{1}{\g_\cR}\left(\m_\cR({h_\cR^*}^2)-\sum_{x,y\in\cR}\m_\cR^*(x)p_\cR^*(x,y)h_\cR^*(y)\right)\,.
\end{split}
\end{equation}
From the last line, using that $\m_\cR^*$ is a left eigenvector of $p_\cR^*$
with eigenvalue $(1-\phi_\cR^*)$ and that $\m_\cR^*(h_\cR^*)=\m_\cR({h_\cR^*}^2)\,,$
we get
\begin{equation}\label{varH2}
\var_{\m_\cR}(h_\cR^*)\leq
\frac{\phi_\cR^*}{\g_\cR}\m_\cR({h_\cR^*}^2)
=\frac{\phi_\cR^*}{\g_\cR}\left(\var_{\mu_\cR}(h_\cR^*)+1\right)\,.
\end{equation}
Finally, rearranging the terms in the above inequality and from the hypothesis
$\e_\cR^*=\frac{\phi_\cR^*}{\g_\cR}<1$, we obtain the required upper bound.

\subsection{Proof of Lemma~\ref{lem:phiphi}}\label{quasi_restr_p_2}
Let us denote by $\gL$ the sub-Markovian generator associated
to the kernel $p^*_{\cR}$. For any function $f\in\ell^2(\mu_{\cR})$, this is defined as
\begin{equation}\label{L^*}
(\gL f )(x)= -f(x)+\sum_{y\in\cR}p_{\cR}^*(x,y)f(y)\,,
\end{equation}
and we have the following useful lemma:
\begin{lemma}\label{lem:Dirichlet}
For all $f\in\ell^2(\m_\cR)$, it holds
\begin{equation}\label{Dirichletforms}
\cD_\cR(f)\leq\frac{\cD(f)}{\m(\cR)}=\langle f,-\gL f\rangle_\cR \,.
\end{equation}
\end{lemma}
\begin{proof}[Proof of Lemma \ref{lem:Dirichlet} ]
For all $x,y\in\cR$ with $x\neq y$, $p_\cR(x,y)=p(x,y)$. Then we have
\begin{equation}
\begin{split}
\cD_\cR(f)&=
\frac{1}{2}\sum_{x,y\in\cR}\m_\cR(x)p_\cR(x,y)\left[f(x)-f(y)\right]^2\\
&=
\frac{1}{2}\sum_{x,y\in\cR}\m_\cR(x)p(x,y)\left[f(x)-f(y)\right]^2
\,,
\end{split}
\end{equation}
since only the terms in $x\neq y$ matter in this sum.
Thus, extending the sum to all $x,y\in\cX$,
\begin{equation}
\cD_\cR(f)\leq
\frac{1}{2}\sum_{x,y\in\cX}\m_\cR(x)p(x,y)\left[f(x)-f(y)\right]^2\leq \frac{\cD(f)}{\m(\cR)}\,,
\end{equation}
and this provides the stated upper bound.

To prove the equality, we recall that the space $\ell^2(\m_\cR)$
is identified with the subset of functions $f\in\ell^2(\mu)$ with
$f_{|_{\XR}}\equiv0$. Since, for all $x,y\in \cR$, it holds
that $\m_\cR(x)=\m(x)/\m(\cR)$ and $p_\cR^*(x,y)=p(x,y)$, we have
\begin{equation}
\begin{split}
\frac{\cD(f)}{\m(\cR)}&=\frac{1}{\m(\cR)}
\sum_{x,y\in\cX}\m(x)f(x)\left(\d_x(y)-p(x,y)\right)f(y)\\
&=
\sum_{x,y\in\cR}\m_\cR(x)f(x)\left(\d_x(y)-p_\cR^*(x,y)\right)f(y)\\
&= \langle f, -\gL f\rangle_\cR\,,
\end{split}
\end{equation}
which concludes the proof.
\end{proof}

We can now proceed with the proof of Lemma \ref{lem:phiphi}.
Since $1 - \phi_\cR^*$ is the largest eigenvalue
of $p_\cR^*$, we have
\begin{equation}\label{fragola}
\phi_{\cR}^* =\min_{\stackrel{\scriptstyle f:\cR\rightarrow \R}{\scriptstyle f\neq 0}}
\frac{\langle f,-\gL f\rangle_\cR}{\langle f,f\rangle_\cR}\,,
\end{equation}
then the equality in Lemma \ref{lem:phiphi}
is a consequence of Lemma \ref{lem:Dirichlet}.
Taking $f=\1_\cR$ as test function in (\ref{fragola}), we get
\begin{equation}
\begin{split}
\phi_\cR^*&\leq \sum_{x\in\cR}\m_\cR(x)\left(1-\sum_{y\in\cR}p_\cR^*(x,y)\right)\\
&= \sum_{x\in\cR}\m_\cR(x)\left(1-\sum_{y\in\cR}p(x,y)\right)\\
&=\sum_{x\in\cR}\m_\cR(x)e_\cR(x)=\phi_\cR\,,
\end{split}
\end{equation}
and it remains to prove that $\E_{\mu_R}\left[\tau_\XR\right]$
lies between $1/\phi_\cR$ and $1/\phi_\cR^*$.

Since, for any $k\in {\mathbb N}$,
$(1 -\phi_\cR^*)^k$ is the largest eigenvalue of
$(p_\cR^*)^k$, the same argument gives
\begin{equation}
1 - (1 -\phi_\cR^*)^k \leq \sum_{x \in \cR} \mu_\cR(x)\left(
  1 - \sum_{y \in \cR} \left(
    p_\cR^*
  \right)^k (x, y)
\right)
\end{equation}
namely,
\begin{equation}
(1 -\phi_\cR^*)^k \geq \sum_{x \in \cR} \mu_\cR(x)
\sum_{y \in \cR} \left(
  p_\cR^*
\right)^k (x, y)\,.
\end{equation}
By summing on $k \geq 0$
and with $\hat X$ the discrete time version
of $X$, such that $X$ follows $\hat X$
at each ring of a Poissonian clock of intensity $1$,
we have, with obvious notation,
\begin{equation}
\frac{1}{\phi_\cR^*} = \sum_{k\geq 0} (1 - \phi_\cR^*)^k
\geq \sum_{k \geq 1} \P_{\mu_\cR}(\hat \tau_\XR \geq k)
= \E_{\mu_\cR}[\tau_\XR]
\geq \P_{\mu_\cR}(\hat \tau_\XR = 1)
= \frac{1}{\phi_\cR}\,.
\end{equation}

\subsection{Proof of Proposition~\ref{prop:gaps}}\label{quasi_restr_p_3}
The second smallest eigenvalue
of the sub-Markovian generator $\gL$, $\phi_\cR^*+\g_\cR^*$, satisfies the variational formula
\begin{equation}
\begin{split}
\phi_\cR^*+\g_\cR^*
&=
\min\left\{\frac{\langle f,-\gL f\rangle_\cR}{\langle f,f\rangle_\cR}\,:
f\neq 0 \,,\langle f,h_\cR^*\rangle_\cR =0\right\}\\
&=
\min\left\{\langle f,-\gL f\rangle_\cR\,:
\langle f,h_\cR^*\rangle_\cR =0\,,\langle f,f\rangle_\cR=1\right\}
\end{split}
\end{equation}
Let $f$ be a function on $\cR$ that realizes the minimum in the above definition,
with $\langle f,f\rangle_\cR=1$. Since $\langle f,h_\cR^*\rangle_\cR =0$, we have

$$\langle f,h_\cR^*-\1_\cR\rangle_\cR =
-\langle f,\1_\cR\rangle_\cR =-\m_\cR(f)$$
and then, by the Cauchy-Schwartz inequality together
with Proposition \ref{prop:varH},
\begin{equation}\label{stimaProp2}
\m_\cR^2(f)\leq \|f\|_\cR^2 \cdot
\|h_\cR^*-\1_\cR\|_\cR^2\leq\frac{\e_\cR^*}{1-\e_\cR^*}\,.
\end{equation}
Now, writing the orthogonal decomposition $f=\m_\cR(f)+g$,
with $\m_\cR(g)=0$, we have
$$1=\|f\|_\cR^2 =\m_\cR^2(f)+\|g\|_\cR^2$$
and thus, from (\ref{stimaProp2}),
$$\|g\|_\cR^2= 1- \m_\cR^2(f)\geq
1-\frac{\e_\cR^*}{1-\e_\cR^*}=\frac{1-2\e_\cR^*}{1-\e_\cR^*}\,.$$
Using $g$ as a test function in
\begin{equation}\label{gapRefl}
\g_\cR=
\min\left\{\frac{\cD_\cR(h)}{\|h\|_\cR^2}\,:
h\neq 0 \,,\m_\cR(h) =0\right\}\,,
\end{equation}
we get
\begin{equation}\label{stimagap1}
\g_\cR\leq \frac{1-\e_\cR^*}{1-2\e_\cR^*}\cD_\cR(g)=\frac{1-\e_\cR^*}{1-2\e_\cR^*}\cD_\cR(f)\,.
\end{equation}
From Lemma \ref{lem:Dirichlet},
and using that
$f$ was chosen in order to have
$\langle f,-\gL f\rangle_\cR=\phi_\cR^*+\g_\cR^*$,
we get
\begin{equation}\label{stimagap2}
\g_\cR\leq \frac{1-\e_\cR^*}{1-2\e_\cR^*}\langle f,-\gL f\rangle_\cR
=\frac{1-\e_\cR^*}{1-2\e_\cR^*}(\phi_\cR^*+\g_\cR^*)\,.
\end{equation}
Setting $\phi_\cR^*=\e_\cR^*\g_\cR$
and rearranging the terms in the last inequality, we get
$$\left(\frac{1-3\e_\cR^*+{\e_\cR^*}^2}{1-2\e_\cR^*}\right)\g_\cR\leq
\left(\frac{1-\e_\cR^*}{1-2\e_\cR^*}\right)\g_\cR^*\,,
$$
which, under the hypothesis $\e_\cR^*<1/3$, implies
$$\frac{1}{\g_\cR^*}\leq\frac{1}{\g_\cR}
\left\{\frac{1-\e_\cR^*}{1-3\e_\cR^*}\right\}\,.$$

\subsection{Proof of Theorem~\ref{th:dinamica}}\label{sec:recurrence}
The proof is based on a classical trick to control mixing times with
relaxation times.
For any probability measure $\nu$ on $\cR$, any $f:\cR\to \R$ such that
$\m_\cR^*(f)\neq0$ and any $s,t\geq 0$, one can check that
\begin{equation}\label{trick}
\begin{split}
\E_\nu [f(X(s+t))\1_{\{\t_{\XR}>s+t\}}]&
-\m_\cR^*(f)\P_{\nu}(\t_{\XR}>s+t)\\
&=\sum_{y\in\cR}
\left(\P_{\nu}(X(s)=y\,,\,\t_{\XR}>s)-
\P_{\nu}(\t_{\XR}>s)\m_\cR^*(y)\right)\\
&\times
\left( \E_y[f(X(t))\1_{\{\t_{\XR}>t\}}]
- \P_y(\t_{\XR}>t)\m_\cR^*(f)\right)\,.
\end{split}
\end{equation}
Indeed, one can rewrite the right-hand side of the above equality
as the sum of four terms, two of which coincide with the two terms
in the left-hand side by the Markov property, while the other two
terms cancel using the quasi-stationarity property, i.e.
\begin{equation}\label{q-statproperty}
\E_{\m_\cR^*}\left[f(X(t))\,|\,\t_{\XR}>t\right]= \m_\cR^*(f)\,.
\end{equation}
As a consequence, by the Cauchy-Schwartz inequality one gets
\begin{equation}\label{trick1}
\begin{split}
\big|\E_\nu [f(X(s+t))&\1_{\{\t_{\XR}>t\}}]
-\m_\cR^*(f)\P_{\nu}(\t_{\XR}>s+t)\big|
\\
&\leq \left\|\sfrac{\P_{\nu}(X(s)=\cdot\,,\,\t_{\XR}>s)}{\m_\cR(\cdot)}-
\P_{\nu}(\t_{\XR}>s)h_\cR^*(\cdot)\right\|_\cR\\
&\times
\left\| \E_{(\cdot)}[f(X(t))\1_{\{\t_{\XR}>t\}}]
- \P_{(\cdot)}(\t_{\XR}>t)\m_\cR^*(f)\right\|_\cR\,.
\end{split}
\end{equation}
We now estimate these two factors.
Noting that
 $$\P_{\nu}(X(s)=\cdot\,,\,\t_{\XR}>s)= \nu e^{s\gL}(\cdot)\quad
\mbox{ and }\quad
\E_{(\cdot)}[f(X(t))\1_{\{\t_{\XR}>t\}}]= e^{t\gL}f(\cdot)\,,$$
and diagonalizing the self-adjoint operator $\gL$ in an orthonormal basis,
one gets
\begin{equation}\label{trick2}
\left\|\sfrac{\P_{\nu}(X(s)=\cdot\,,\,\t_{\XR}>s)}{\m_\cR(\cdot)}-
\|\sfrac{\nu}{\m_\cR}\|_{_\cR}\sfrac{h_\cR^*}
{\|h_\cR^*\|_{_\cR}}\cos\th_\nu\,e^{-\phi_\cR^*s}\right\|_\cR^2
\leq
\|\sfrac{\nu}{\m_\cR}\|_{_\cR}^2\sin^2\th_\nu e^{-2s(\phi_\cR^*+\g_\cR^*)}\,,
\end{equation}
with $\th_\nu\in[0,\pi/2[$ such that
$\|\sfrac{\nu}{\m_\cR}\|_{_\cR}\|h_\cR^*\|_{_\cR}\cos\th_\nu=
\langle\sfrac{\nu}{\m_\cR},h_\cR^*\rangle=\nu(h_\cR^*)\,,$
and
\begin{equation}\label{trick3}
\left\| \E_{(\cdot)}[f(X(t))\1_{\{\t_{\XR}>t\}}]
- \|f\|_{_\cR}\sfrac{h_\cR^*}
{\|h_\cR^*\|_{_\cR}}\cos\th_f\,e^{-\phi_\cR^*t}\right\|_\cR^2
\leq
\|f\|^2_{_\cR}\sin^2\th_f e^{-2t(\phi_\cR^*+\g_\cR^*)}\,
\end{equation}
with $\th_f\in[0,\pi]\setminus\{\pi/2\}$ such that
$\|f\|_{_\cR}\|h_\cR^*\|_{_\cR}\cos\th_f=
\langle f,h_\cR^*\rangle=\m_\cR^*(f)\,.$\\
Moreover, since
$$\P_{\nu}(\t_{\XR}>s)=
\m_\cR\left( \sfrac{\P_{\nu}(X(s)=\cdot\,,
\,\t_{\XR}>s)}{\m_\cR(\cdot)}\right)\,,$$
 by the Cauchy-Schwartz inequality and using (\ref{trick2}) we get
\begin{equation}\label{trick4}
\begin{split}
\left|\P_{\nu}(\t_{\XR}>s)-
\|\sfrac{\nu}{\m_\cR}\|_{_\cR}\sfrac{\cos\th_\nu}
{\|h_\cR^*\|_{_\cR}}\,e^{-\phi_\cR^*s}\right|&=
\left|\m_\cR\left(\sfrac{\P_{\nu}(X(s)=\cdot\,,\,\t_{\XR}>s)}{\m_\cR(\cdot)}-
\|\sfrac{\nu}{\m_\cR}\|_{_\cR}\sfrac{h_\cR^*}
{\|h_\cR^*\|_{_\cR}}\cos\th_\nu\,e^{-\phi_\cR^*s}\right)\right|\\
&\leq
\left\langle \1_\cR,
\left|\sfrac{\P_{\nu}(X(s)=\cdot\,,\,\t_{\XR}>s)}{\m_\cR(\cdot)}-
\|\sfrac{\nu}{\m_\cR}\|_{_\cR}\sfrac{h_\cR^*}
{\|h_\cR^*\|_{_\cR}}\cos\th_\nu\,e^{-\phi_\cR^*s}\right|
\right\rangle_\cR\\
&\leq \|\sfrac{\nu}{\m_\cR}\|_{_\cR}\sin\th_\nu e^{-s(\phi_\cR^*+\g_\cR^*)}\,.
\end{split}
\end{equation}
Using inequalities (\ref{trick2}) and (\ref{trick4}), we finally get
\begin{equation}\label{trick5}
\begin{split}
\left\|\sfrac{\P_{\nu}(X(s)=\cdot\,,\,\t_{\XR}>s)}{\m_\cR(\cdot)}-
\P_{\nu}(\right.&\left.\t_{\XR}>s)h_\cR^*(\cdot)\right\|_\cR
\\
&\leq
\left\|\sfrac{\P_{\nu}(X(s)=\cdot\,,\,\t_{\XR}>s)}{\m_\cR(\cdot)}-
\|\sfrac{\nu}{\m_\cR}\|_{_\cR}\sfrac{h_\cR^*}
{\|h_\cR^*\|_{_\cR}}\cos\th_\nu\,e^{-\phi_\cR^*s}\right\|_\cR\\
&+
\left\|\left(\P_{\nu}(\t_{\XR}>s)-
\|\sfrac{\nu}{\m_\cR}\|_{_\cR}\sfrac{\cos\th_\nu}
{\|h_\cR^*\|_{_\cR}}\,e^{-\phi_\cR^*s}\right)h_\cR^*\right\|_\cR\\
&\leq
\|\sfrac{\nu}{\m_\cR}\|_{_\cR}(1+\|h_\cR^*\|_\cR)\sin\th_\nu \,e^{-s(\phi_\cR^*+\g_\cR^*)}\,.
\end{split}
\end{equation}
which provides an estimate of the first factor in (\ref{trick1}).

To what concerns the second factor, noting that
$$\P_{(\cdot)}(\t_{\XR}>t)=\E_{(\cdot)}[\1_\cR(X(t))\1_{\{\t_{\XR}>t\}}]
$$
and that, from the definition of $\cos\th_f$ applied to $f=\1_\cR$,
$$\cos \th_{\1_\cR}=\sfrac{1}{\|h_\cR^*\|_{_\cR}}
\quad\mbox{ and }\quad
\sin^2\th_{\1_\cR}=\sfrac{\|h_\cR^*\|^2_{_\cR}-1}{\|h_\cR^*\|^2_{_\cR}}=
\sfrac{\var_{\m_\cR}(h_\cR^*)}{\|h_\cR^*\|^2_{_\cR}}\,,$$
from inequality (\ref{trick3}) we get
\begin{equation}\label{trick6}
\left\|\P_{(\cdot)}(\t_{\XR}>t)-
\sfrac{h_\cR^*}{\|h_\cR^*\|^2_{_\cR}}\,e^{-\phi_\cR^*t}\right\|^2
\leq \sfrac{\var_{\m_\cR}(h_\cR^*)}{\|h_\cR^*\|^2_{_\cR}}e^{-2t(\phi_\cR^*+\g_\cR^*)}\,.
\end{equation}
Then, from inequalities (\ref{trick3}) and (\ref{trick6}),

\begin{equation}\label{trick7}
\begin{split}
\left\| \E_{(\cdot)}[f(X(t))\right.&\left.\1_{\{\t_{\XR}>t\}}]
-\P_{(\cdot)}(\t_{\XR}>t)\m_\cR^*(f)\right\|_\cR
\\
&\leq
\left\|\E_{(\cdot)}[f(X(t))\1_{\{\t_{\XR}>t\}}]-
\|f\|_{_\cR}\sfrac{h_\cR^*}
{\|h_\cR^*\|_{_\cR}}\cos\th_f\,e^{-\phi_\cR^*t}\right\|_\cR\\
&+
\left\|\P_{(\cdot)}(\t_{\XR}>t)\|f\|_{_\cR}\|h_\cR^*\|_{_\cR}\cos\th_f-
\|f\|_{_\cR}\sfrac{h_\cR^*}
{\|h_\cR^*\|_{_\cR}}\cos\th_f\,e^{-\phi_\cR^*t}\right\|_\cR\\
&\leq
\|f\|_{_\cR}\sin\th_fe^{-t(\phi_\cR^*+\g_\cR^*)}+
\|f\|_{_\cR}\|h_\cR^*\|_{_\cR}\sfrac{\sqrt{\var_{\m_\cR}(h_\cR^*)}}
{\|h_\cR^*\|_{_\cR}}\cos\th_f\, e^{-t(\phi_\cR^*+\g_\cR^*)}\\
&=
\left(\|f\|_{_\cR}\sin\th_f+ \sfrac{\m_\cR^*(f)\sqrt{\var_{\m_\cR}(h_\cR^*)}}
{\|h_\cR^*\|_{_\cR}} \right)e^{-t(\phi_\cR^*+\g_\cR^*)}\,.
\end{split}
\end{equation}
which provides an estimate of the second factor in (\ref{trick1}).

Inserting (\ref{trick5}) and (\ref{trick7}) in (\ref{trick1}),
we then obtain
\begin{equation}\label{trick8}
\begin{split}
&\big|\E_\nu [f(X(s+t))\1_{\{\t_{\XR}>t\}}]
-\m_\cR^*(f)\P_{\nu}(\t_{\XR}>s+t)\big|
\\
&\leq \|\sfrac{\nu}{\m_\cR}\|_{_\cR}(1+\|h_\cR^*\|_\cR)\sin\th_\nu
\left(\|f\|_{_\cR}\sin\th_f+ \sfrac{\m_\cR^*(f)\sqrt{\var_{\m_\cR}(h_\cR^*)}}
{\|h_\cR^*\|_{_\cR}} \right)e^{-(s+t)(\phi_\cR^*+\g_\cR^*)}\,.
\end{split}
\end{equation}

To conclude our proof we will make two more steps.
First notice that from (\ref{trick4}) one also gets that, for any $t\geq0$,
$$
\P_{\nu}(\t_{\XR}>t)\geq
\|\sfrac{\nu}{\m_\cR}\|_{_\cR}
\left(\sfrac{\cos\th_\nu}{\|h_\cR^*\|_{_\cR}}e^{-t\phi_\cR^*}
-\sin\th_\nu \,e^{-t(\phi_\cR^*+\g_\cR^*)}\right)\,.
$$
In particular, as soon as the following condition is verified
\begin{equation}
\sin\th_\nu \,e^{-t(\phi_\cR^*+\g_\cR^*)}\leq
\d\sfrac{\cos\th_\nu}{\|h_\cR^*\|_{_\cR}}e^{-\phi_\cR^*t}
\,,
\end{equation}
that is
\begin{equation}\label{conditiondelta}
\|h_\cR^*\|_{_\cR}\tan\th_\nu\, e^{-\phi_\cR^*t}\leq\d\,,
\end{equation}
it holds
\begin{equation}\label{trick4bis}
\P_{\nu}(\t_{\XR}>t)\geq
(1-\d)\|\sfrac{\nu}{\m_\cR}\|_{_\cR}
\sfrac{\cos\th_\nu}{\|h_\cR^*\|_{_\cR}}e^{-\phi_\cR^*t}
\,.
\end{equation}
Now, dividing both terms of (\ref{trick8})
by $\m_\cR^*(f)\P_{\nu}(\t_{\XR}>s+t)$,
we reach an inequality that controls the Yaglom limit
and that, provided condition (\ref{conditiondelta}) holds and
then using the last inequality, reads as
\begin{equation}\label{trick9}
\left|\sfrac{\E_\nu [f(X(t))\,|\,\t_{\XR}>t]}
{\m_\cR^*(f)}-1\right|\leq
\sfrac{1+\|h_\cR^*\|_{_\cR}}{1-\d}
\tan\th_\nu \left(\tan\th_f+ \sqrt{\var_{\m_\cR}(h_\cR^*)}\right)e^{-\g_\cR^* t}\,.
\end{equation}
As a final step we apply this inequality to $\nu=\d_x$ and $f=\d_y$.
For this choice of $\nu$ and $f$, and  by definition
of $\th_\nu$ and $\th_f$, one has
$$\tan\th_\nu \leq \frac{1}{\cos\th_\nu}=
\frac{\|h_\cR^*\|_{_\cR}}{\sqrt{\m_\cR(x)h_\cR^*(x)}}
\quad \mbox{ and }\quad
\tan\th_f \leq \frac{1}{\cos\th_f}=
\frac{\|h_\cR^*\|_{_\cR}}{\sqrt{\m_\cR(y)h_\cR^*(y)}}\,,
$$
Thus, from (\ref{trick9}), we obtain that under condition (\ref{conditiondelta})
\begin{equation}\label{trick10}
\begin{split}
&\left|\sfrac{\P_x \left((X(t)=y)\,|\,\t_{\XR}>t\right)}
{\m_\cR^*(y)}-1\right|
\leq
e^{-\g_\cR^* t}
\sfrac{(1+\|h_\cR^*\|_{_\cR})\|h_\cR^*\|^2_{_\cR}}{(1-\d)\sqrt{\m_\cR(x)h_\cR^*(x)}}
\left(\sfrac{1}{\sqrt{\m_\cR(y)h_\cR^*(y)}}+
 \sfrac{\sqrt{\var_{\m_\cR}(h_\cR^*)}}{\|h_\cR^*\|_{_\cR}}\right)\\
&
\qquad\qquad\qquad
\leq e^{-\g_\cR^* t}\sfrac{1}{1-\d}\left(1+\sqrt{1+\sfrac{\e_\cR^*}{1-\e_\cR^*}}\right)
\left(1+\sfrac{\e_\cR^*}{1-\e_\cR^*}\right)
\sfrac{1}{\sqrt{\z_\cR^*}}\left(\sfrac{1}{\sqrt{\z_\cR^*}}+
\sqrt{\sfrac{\e_\cR^*}{1-\e_\cR^*}}\right)
\,,
\end{split}
\end{equation}
where in the second line we used that $\|h_\cR^*\|_{_\cR}\geq 1$,
the estimate given in Proposition \ref{prop:varH}, and we introduced the
quantity $\z_\cR^*$ defined in (\ref{zeta}).

The right-hand side of the last inequality is
smaller than $\d$ as soon as
\begin{equation}\label{conditionT1}
t\geq \sfrac{1}{\g_\cR^*}\left[\ln{\sfrac{2}{\d(1-\d)\z_\cR^*}} +
\ln\left(\left(\sfrac{1}{2}+\sfrac{1}{2}\sqrt{1+\sfrac{\e_\cR^*}{1-\e_\cR^*}}\right)
\left(1+\sfrac{\e_\cR^*}{1-\e_\cR^*}\right)
\left(1+\sqrt{\sfrac{\z_\cR^*\e_\cR^*}{1-\e_\cR^*}}\right)\right)\right]\,,
\end{equation}
which also implies  (\ref{conditiondelta}).

Finally, from the hypothesis $\e_\cR^*<1/3$,
from the concavity of the logarithm and of the square root function,
and using that $\z_\cR^*\leq 1$, then $\d(1-\d)\leq1/4$
and $\ln 8 \geq 1 + 5/(4\sqrt{2})$,
after some computation one obtains
that the condition (\ref{conditionT1}) is implied by
the stronger condition
\begin{equation}\label{conditionT2}
t\geq \sfrac{1}{\g_\cR^*}\left(\ln{\sfrac{2}{\d(1-\d)\z_\cR^*}}\right)
\left\{1 +\sqrt{\sfrac{\e_\cR^*}{1-\e_\cR^*}}\right\}\,,
\end{equation}
which, in turns, follows from $t > T^*_{\delta, \mathcal{R}}$
by using Proposition~\ref{prop:gaps}.

\section{Around the exponential law}\label{exp_law_p}

\subsection{Proof of Theorem~\ref{th:tempo}}\label{exp_law_p_2}
We write
\begin{equation}\label{probT}
\nonumber
\begin{split}
\P_\nu (\t_{\XR}>\sfrac{t}{\phi_\cR^*})&=\pi_\cR(\nu)
\P_\nu (\t_{\XR}>\sfrac{t}{\phi_\cR^*}\,|\,\t_{\XR}<T_\cR^*)\\
&+(1-\pi_\cR(\nu))
\P_\nu (\t_{\XR}>\sfrac{t}{\phi_\cR^*}\,|\,\t_{\XR}>T_\cR^*)\,.
\end{split}
\end{equation}
If $t\geq\phi_\cR^* T_\cR^*$, the first term in the r.h.s equals zero and we get
$$
\P_\nu \left(\t_{\XR}>\sfrac{t}{\phi_\cR^*}\right)=(1-\pi_\cR(\nu))
\P_\nu \left(\t_{\XR}>\sfrac{t}{\phi_\cR^*}\,|\,\t_{\XR}>T_\cR^*\right)\,.
$$
By Theorem \ref{th:dinamica}, we also have
\begin{equation}
\nonumber
\left| \P_\nu \left(\t_{\XR}>\sfrac{t}{\phi_\cR^*}\,|
\,\t_{\XR}>T_\cR^*\right)
-e^{-\phi_\cR^*(\sfrac{t}{\phi_\cR^*}-T_\cR^*)}\right|
\leq \e_\cR^*e^{-\phi_\cR^*(\sfrac{t}{\phi_\cR^*}-T_\cR^*)}\,,
\end{equation}
which together with the previous equality completes the proof.

\subsection{Proof of Lemma~\ref{lem:piR}}\label{exp_law_p_3}
On the one hand we have
\begin{equation}
	\P_{\mu_\cR^*}(\tau_{\XR} \leq T_\cR^*)
	 = 1 - e^{-\phi_\cR^* T_\cR^*}
	\leq \phi_\cR^* T_\cR^*
	\,.
\end{equation}
On the other hand, denoting by $d_{TV}(\mu,\nu)$
the total variation distance between $\mu$ and $\nu$,
from the $\ell^2(\mu_\cR)$ estimate
given in Proposition~\ref{prop:varH},
together with the Cauchy-Schwartz inequality,
we get
\begin{equation}\label{distanceTV}
\begin{split}
d_{TV}(\m_{\cR},\m_{\cR}^*)
&=
\frac 12\sum_{x\in\cR}\left|\m_\cR(x)- \m_\cR^*(x)\right|
= \frac 12\sum_{x\in\cR}\m_\cR(x)\left|1- h_\cR^*(x)\right|\\
&\leq \frac 12
\var_{\m_\cR} (h_\cR)^{1/2}
\leq \frac 12 \sqrt{\frac{\e_\cR^*}{1-\e_\cR^*}}\,.
\end{split}
\end{equation}
We then derive the desired result by using an optimal coupling.

\section{Working with $(\kappa, \lambda)$-capacities}\label{klc}
\subsection{Proof of the upper bound in Theorem \ref{th:phi&cap}}\label{sole}
Let $\ti X$  denote the continuous-time Markov chain on $\ti\cX$
defined, for $\ti \kappa>0$,by the generator
\begin{equation}\label{generator}
(\ti\cL f)(\ti x)=
\left\{
\ba{ll}
\ti \kappa (f(x)-f(\bar x)) & \mbox{ if } \ti x=\bar x\in\bar\cR\\
(\cL f)(x)+\kappa (f(\bar x)-f(x))& \mbox{ if } \ti x=x\in\cR\\
(\cL f)(x)& \mbox{ if } \ti x=x\in\XR
\ea
\right.\,.
\end{equation}
This is a reversible process with respect to  a measure
$\ti \mu$ defined as
\begin{equation}\label{tilde-mu}
\ti\m(\ti x)=\left\{
\ba{ll}
\m(x)& \mbox{if } \ti x=x\in\cX\\
\sfrac{\kappa}{\ti \kappa}\m(x)& \mbox{if } \ti x=\bar x\in\bar\cR
\ea
\right.\,.
\end{equation}
Note that $\ti\mu$ is not a probability measure.
Let us denote by $\ti\nu_{\bar\cR}$ the harmonic measure
on $\bar\cR$ associated with $\XR$, i.e.,
the probability measure on $\bar\cR$ defined by
\begin{equation}\label{har-measure}
\ti\nu_{\bar\cR}(\bar x)=\frac{-\ti\mu(\bar x)
(\ti\cL\ti V_\kappa)(\bar x)}{C_\kappa (\cR,\XR)}
\end{equation}
and with
$$
\ti V_\kappa (\ti x)=\left\{
\ba{ll}
V_\kappa (x)=\P_x(\s_\kappa<\t_{\XR}) & \mbox{ if } \ti x=x\in\cR\\
1 & \mbox{ if } \ti x=\bar x\in\bar\cR\\
0 & \mbox{ if } \ti x=x\in\XR
\ea
\right.\,.
$$
With obvious notation, we then have
\begin{equation}\label{media-arm}
\E_{\ti\nu_{\bar\cR}}[\ti\t_{\XR}]=
\frac{\ti\m(\ti V_\kappa)}{C_\kappa (\cR,\XR)}\,.
\end{equation}
Such kind of formula was introduced into the study
of metastability in~\cite{BEGK1}. We refer
to lecture notes~\cite{G} for a derivation.

Now
setting $\nu(x)=\ti\nu_{\bar\cR}(\bar x)$ for all $x\in\cR$, we can write
$$\E_{\ti\nu_{\bar\cR}}[\ti\t_{\XR}]=
\sfrac{1}{\ti \kappa} +\E_{\nu}[\t_{\XR}]+
\E_{\nu}[\t_{\XR}]\cdot \kappa\cdot\sfrac{1}{\ti \kappa}
=\sfrac{1}{\ti \kappa} +\E_{\nu}[\t_{\XR}](1+\sfrac{\kappa}{\ti \kappa})\,,
$$
where the first of the three summands
stands for the mean time to go from $\bar \cR$
to $\cR$, the second one for the mean time spent
when moving inside $\cR$ before reaching $\XR$,
and the last one for the mean time spent
moving back and forth between $\cR$ and $\bar\cR$.
From (\ref{tilde-mu}) we also have
$$
\frac{\ti\m(\ti V_\kappa)}{C_\kappa (\cR,\XR)}=
\frac{\m(V_\kappa)+\sum_{\bar x\in\bar\cR}\ti\m(\bar x)
}{C_\kappa (\cR,\XR)}
=\frac{\m(V_\kappa)+\frac{\kappa}{\ti \kappa}\m(\cR)}{C_\kappa (\cR,\XR)}\,.
$$
Inserting the above equalities in (\ref{media-arm})
and multiplying by $\ti \kappa$, we then get
\begin{equation}\label{media-arm2}
1+\E_{\nu}[\t_{\XR}](\ti \kappa+\kappa)
=\frac{\ti \kappa\m(V_\kappa)+\kappa\m(\cR)}{C_\kappa (\cR,\XR)}\,.
\end{equation}
Note that $\m(\cR),\m(V_\kappa)$, $C_\kappa (\cR,\XR)$
and $\E_{\nu}[\t_{\XR}]$ do not depend on $\ti \kappa$.
Then, in the limit of a vanishing $\ti \kappa$, it holds
\begin{equation}\label{picche}
1+\kappa\E_{\nu}[\t_{\XR}]
=\frac{\kappa\m(\cR)}{C_\kappa (\cR,\XR)}\,.
\end{equation}
This already provides, by (\ref{banana}), an upper bound on $\phi_\cR^*$.

To get the more practical upper bound stated in (\ref{phi&cap}),
let first note that dividing (\ref{media-arm2}) by $\ti \kappa$,
and then sending $\ti \kappa$ to $+\infty$,
we get
\begin{equation}\label{media-arm3}
\E_{\nu}[\t_{\XR}]
=\frac{\m(V_\kappa)}{C_\kappa (\cR,\XR)}\,.\end{equation}
Together with (\ref{picche}), this implies $
\frac{\m(V_\kappa)}{C_\kappa (\cR,\XR)}=
\frac{\m(\cR)}{C_\kappa (\cR,\XR)}-\frac{1}{\kappa}$
or equivalently
\begin{lemma}\label{lem:mediaV_k}
\begin{equation}\label{mediaV_k}
\m_\cR(V_\kappa)= 1-\frac{C_\kappa (\cR,\XR)}{\kappa\m(\cR)}\end{equation}
\end{lemma}

We now exploit the variational principle for $\phi_\cR^*$ provided by
Lemma \ref{lem:phiphi} and take $V_\kappa$ as test function.
Noting that  $\ti V_\kappa$ is the equilibrium potential
of the electrical network on $\tilde \cX$ defined in (\ref{conductance}),
from (\ref{capacity}) we get
\begin{equation}
\cD(V_\kappa (x)) \leq C_\kappa (\cR, \XR)\,.
\end{equation}
By the Jensen inequality and (\ref{mediaV_k}),
\bea
\|V_\kappa\|^2 &= & \m(\cR) \sum_{x\in\cR}
\m_\cR(x)\P_x(\s_\kappa<\t_\XR)^2\nonumber\\
&\geq& \m(\cR)\m_\cR (V_\kappa)^2  \nonumber\\
&=& \m(\cR) \left(1- \sfrac{C_\kappa (\cR, \XR)}{\kappa \m(\cR)}\right)^2\,.
\eea
Finally inserting these inequalities in the variational principle for $\phi_\cR^*$,
we get the stated upper bound.

\subsection{Proof of the upper bound in Theorem \ref{th:gap&cap}}
For any $f\in\ell^2(\m)$, we have
\begin{equation}\label{splitVar}
\begin{split}
\var_\m(f)&=
\m(\var_\m(f|\1_\cR)) + \var_\m(\m(f|\1_\cR))\\
&=\m(\cR)\var_{\m_\cR}(f_{|_\cR}) + \m(\XR)
\var_{\m_{\XR}}(f_{|_{\XR}})\\
&+
\m(\cR)\m(\XR)\left(\m_\cR(f_{|_{\cR}})-
\m_{\XR}(f_{|_{\XR}})\right)^2\,.
\end{split}
\end{equation}
Now, using the test function
$$\ti f= \frac{f-\m_{\XR}(f_{|_{\XR}})}{\m_{\cR}(f_{|_{\cR}})-\m_{\XR}(f_{|_{\XR}})}$$
in the definition (\ref{capacity}) of $(\kappa,\l)$-capacity, we get
\begin{equation}\nonumber
\begin{split}
&C_\kappa^\l(\cR,\XR)\leq \cD(\ti f)+\kappa\m(\cR)\var_{\m_\cR}(\ti f_{|_\cR})+
\l\m(\XR)\var_{\m_{\XR}}(\ti f_{|_{\XR}})\\
&=\left(\m_\cR(f_{|_{\cR}})-\m_{\XR}(f_{|_{\XR}})\right)^{-2}\!\!
\left(\cD(f)+\kappa\m(\cR)\var_{\m_\cR}(f_{|_\cR})+
\l\m(\XR)\var_{\m_{\XR}}(f_{|_{\XR}})\right)\,,
\end{split}
\end{equation}
which provides an upper bound on
$\left(\m_\cR(f_{|_{\cR}})-\m_{\XR}(f_{|_{\XR}})\right)^{2}$.
Applying that bound in Eq. (\ref{splitVar}), and from the definition
of $\phi_\kappa^\l(A,B)$, we get
\bea\label{fiori}
\var_\m(f)&\leq& \m(\cR)\var_{\m_\cR}(f_{|_\cR}) + \m(\XR)
\var_{\m_{\XR}}(f_{|_{\XR}})\nonumber\\
&+&\left(\cD(f)+\kappa\m(\cR)\var_{\m_\cR}(f_{|_\cR})+
\l\m(\XR)\var_{\m_{\XR}}(f_{|_{\XR}})\right){\phi_\kappa^\l(\cR,\XR)}^{-1}\\
&\leq&
\sfrac{\cD(f)}{\phi_\kappa^\l(\cR,\XR)}+
\sfrac{\m(\cR)\cD_\cR(f_{|_\cR})}{\g_\cR}
\left(1+\sfrac{\kappa}{\phi_\kappa^\l(\cR,\XR)}\right)
+\sfrac{\m(\XR)\cD_{\XR}(f_{|_{\XR}})}{\g_{\XR}}
\left(1+\sfrac{\l}{\phi_\kappa^\l(\cR,\XR)}\right)\nonumber\\
&\leq&
\sfrac{\cD(f)}{\phi_\kappa^\l(\cR,\XR)}
\left\{1 + \max\left(\sfrac{\phi_\kappa^\l(\cR,\XR)+\kappa}{\g_\cR};
\sfrac{\phi_\kappa^\l(\cR,\XR)+\l}{\g_{\XR}}\right)\right\}\,,\nonumber
\eea
where in the last step we used that
$$\cD(f)\leq \m(\cR)\cD_\cR(f_{|_\cR})+\m(\XR)\cD_{\XR}(f_{|_{\XR}})\,.
$$
The upper bound in (\ref{gap&cap}) follows directly.

\subsection{Proof of the lower bound of Theorem \ref{th:phi&cap}}
From inequality (\ref{fiori}) applied to $f=h_\cR^*$ and  $\l=+\infty$,
and since ${h_\cR^*}_{|_{\XR}}=0$, we get
\begin{equation}\label{varH}
\var_{\m}(h_\cR^*)\leq
\m(\cR)\var_{\m_\cR}(h_\cR^*)+\frac{\cD(h_\cR^*)}{C_\kappa(\cR,\XR)}\m(\cR)(1-\m(\cR))
\left\{1+\sfrac{\kappa}{\g_\cR}\right\}\,.
\end{equation}
On the other hand, by (\ref{splitVar}),
$$
\var_{\m}(h_\cR^*)=
\m(\cR)\var_{\m_\cR}(h_\cR^*)+ \m(\cR)(1-\m(\cR))\,.
$$
Inserting this formula in (\ref{varH}), the term  $\m(\cR)\var_{\m_\cR}(h_\cR^*)$
becomes zero and then,  dividing by $\m(\cR)(1-\m(\cR))$, we have
\begin{equation}
1\leq
\frac{\cD(h_\cR^*)}{C_\kappa(\cR,\XR)}
\left\{1+\sfrac{\kappa}{\g_\cR}\right\}\,,
\end{equation}
or equivalently
\begin{equation}
C_\kappa(\cR,\XR)\left\{1+\sfrac{\kappa}{\g_\cR}\right\}^{-1}\leq\cD(h_\cR^*)
\,.
\end{equation}
Now, dividing by $\m(\cR)\|h_\cR^*\|_\cR^2$ and using that,
by Prop. \ref{prop:varH},
$\|h_\cR^*\|_\cR^2=\var_{\m_\cR}(h_\cR^*)+1\leq 1/(1-\e_\cR^*)$,
 we get
\begin{equation}
\frac{C_\kappa(\cR,\XR)}{\m(\cR)}\left\{\frac{1-\e_\cR^*}{1+\sfrac{\kappa}{\g_\cR}}\right\}
\leq\frac{C_\kappa(\cR,\XR)}{\m(\cR)\|h_\cR^*\|_\cR^2}\left\{1+\sfrac{\kappa}{\g_\cR}\right\}^{-1}
\leq\frac{\cD(h_\cR^*)}{\|h_\cR^*\|_\cR^2\m(\cR)}=\phi_\cR^*\,,
\end{equation}
where the last equality comes from Lemma \ref{lem:Dirichlet}
and the fact that
$$\langle h_\cR^*,-\gL h_\cR^*\rangle_{\cR}= \phi_\cR^*\,. $$
Finally, using the convexity of the function $x\mapsto\sfrac{1}{1+x}$, we obtain
\begin{equation}
\frac{C_\kappa(\cR,\XR)}{\m(\cR)}\left\{1-\e_\cR^*-\sfrac{\kappa}{\g_\cR}\right\}\leq
\phi_\cR^*\,,
\end{equation}
which concludes the proof of the lower bound in (\ref{phi&cap}).

\subsection{Proof of the lower bound in Theorem \ref{th:gap&cap}}
We use the test function $V_\kappa^\l$, for which we know that (see (\ref{capacity}))
$$\cD(V_\kappa^\l)+\kappa\m(\cR)\E_{\m_\cR}\left[\left({V_\kappa^\l}_{|_\cR}-1\right)^2\right]
+\l\m(\XR)\E_{\m_{\XR}}
\left[\left({V_\kappa^\l}_{|_{\XR}}-0\right)^2\right]=C_\kappa^\l(\cR,\XR)
$$
so that
\begin{equation}\label{gap&cap-low1}
\cD(V_\kappa^\l)\leq C_\kappa^\l(\cR,\XR)
\,.
\end{equation}
We then look for a lower bound on $\var_\m(V_\kappa^\l)$.
From (\ref{splitVar}) we have
$$\var_\m(V_\kappa^\l)\geq \m(\cR)\m(\XR)\left(\m_\cR
({V_\kappa^\l}_{|_\cR})-\m_{\XR}({V_\kappa^\l}_{|_{\XR}})\right)^2\,,$$
thus we need to estimate $\m_\cR({V_\kappa^\l}_{|_\cR})$
and $\m_{\XR}({V_\kappa^\l}_{|_{\XR}})$.

By the monotonicity in $\l$, for all $x\in\cR$ we get
$$V_\kappa^\l(x)=\P_x(\ell^{-1}_\cR(\s_\kappa))<\ell^{-1}_{\XR}(\s_\l))
\geq \P_x(\s_\kappa<\t_{\XR})=V_\kappa(x)\,,$$
which implies, together with Lemma \ref{lem:mediaV_k},
$$
\m_\cR({V_\kappa^\l}_{|_{\cR}})\geq 1-\sfrac{C_\kappa(\cR,\XR)}{\kappa\m(\cR)}\,.
$$
In the same way we have
$$
\m_{\XR}({V_\kappa^\l}_{|_{\XR}})\leq \sfrac{C^\l(\cR,\XR)}{\l\m(\XR)}\,.
$$
Altogether, we finally get
\begin{equation}
\g\leq \sfrac{\cD(V_\kappa^\l)}{\var_\m(V_\kappa^\l)}\leq
\sfrac{C_\kappa^\l(\cR,\XR)}{\m(\cR)\m(\XR)}
\left(1-\sfrac{C_\kappa(\cR,\XR)}{\kappa\m(\cR)}-
\sfrac{C^\l(\cR,\XR)}{\l\m(\XR)}\right)^{-2}\,.
\end{equation}

\section{Working with soft measures}\label{soft_p}
\subsection{Proof of Lemma~\ref{lem:softmeasure}}\label{soft_p_1}
If $\l=0$ the first statement holds trivially since,  in that
case, $\phi_{\cR,\l}^*=0=\m_{\cR,\l}^*(e_{\cR,\l})$.
If $\l>0$, we can write
$$
\P_{\m_{\cR,\l}^*}(\t_{\XR,\l}\leq t)=\sum_{k\geq 1}
\P_{\m_{\cR,\l}^*}(N_{\cR}(t)\geq k)(1-\phi_{\cR,\l}^*)^{k-1}
\m_{\cR,\l}^*(e_{\cR,\l})\,,
$$
where $N_\cR(t)$ is the number of clock rings inside $\cR$
for the Poissonian clock associated to $X$.
Taking the limit as $t\to\infty$ in the above equation,
we get that
$$1=\m_{\cR,\l}^*(e_{\cR,\l})/\phi_{\cR,\l}^* \,,$$
which provides identity i).\\
Let us now define the operator $\lL$ on $\ell^2(\m_\cR)$ as
\begin{equation}\label{def:generator2}
(\lL f)(x)= -f(x)+\sum_{y\in\cR}p_{\cR,\l}^*(x,y)f(y)
\qquad \forall x\in\cR\,, f\in\ell^2(\m_\cR)
\end{equation}
and notice that, for any probability measure $\nu$ on $\cX$, it holds
\begin{equation}\label{media&semigrupppo}
\E_\nu\left[f(X\circ \ell^{-1}_\cR(t))\1_{\{\t_{\XR,\l}>t\}}\right]
=\nu\left(e^{t\lL}f\right)\,.
\end{equation}
 The exponential law given in ii) follows from the above
 identity applied to $\nu=\mu_{\cR,\l}^*$ and $f=\1_\cR$.

Finally, since $1-\phi_{\cR,\l}^*$ is a simple eigenvalue
equal to the spectral radius of $p_{\cR,\l}^*$, for any $x,y\in\cR$
and in the large $t$ regime,  we have
 \begin{equation}\label{quasi-stat2}
\P_x(X\circ \ell^{-1}_\cR(t)=y\,,\,\t_{\XR,\l}>t)\sim c_x \m_{\cR,\l}^*(y)e^{-t\phi_{\cR,\l}^*}\,,
 \end{equation}
where $c_x\m_{\cR,\l}^*$ is the canonical projection of $\d_x$
on the one-dimensional eigenspace associated with $\m_{\cR,\l}^*$
($c_x$ is strictly positive as a consequence of the positivity of $\mu_{\cR, \lambda}^*$).
From (\ref{quasi-stat2}), and taking the limit when $t$ goes to infinity, it follows
$$
\lim_{t\to\infty}\P_x(X\circ \ell^{-1}_\cR(t)=y\,|\,\t_{\XR,\l}>t)=\m_{\cR,\l}^*(y)
\,.$$

\subsection{Proof of Lemma~\ref{interpolation}}\label{soft_p_2}
The result is once again a consequence of the Perron-Frobenius
theorem. Let $\chi_\l(y)$ denote  the characteristic polynomial
of $\lL$, which can be written as $\chi_\l(y)=(y+\phi_{\cR,\l}^*)a(y)$.
If $a(y)= (y+\phi_{\cR,\l}^*)q(y)+a(-\phi_{\cR, \lambda}^*)$ is the Euclidian division
of $a(y)$ by $(y-\phi_{\cR,\l}^*)$,
we have the B\'{e}zout identity
\begin{equation}
\sfrac{1}{a(-\phi_{\cR, \lambda}^*)} a(y)-\sfrac{1}{a(-\phi_{\cR, \lambda}^*)} q(y)(y+\phi_{\cR,\l}^*)=1\,.
\end{equation}
In particular, for any $x\in\cR$,
$\sfrac{1}{a(-\phi_{\cR, \lambda}^*)}\d_x a(\lL)= c_x \m_{\cR,\l}^*$
is the canonical projection of $\d_x$ on the eigenspace
associated to $\m_{\cR,\l}^*$, and since $c_x>0$ as previously noticed,
we  have
 \begin{equation}
\m_{\cR,\l}^*= \frac{\d_x a(\lL)}{\sum_{y\in\cR}\d_x a(\lL)\1_{\{y\}}}\,.
 \end{equation}
Since $a(y)=\frac{\chi_\l(y)}{(y+\phi_{\cR,\l}^*)}$, the above equation
expresses the map $\l\mapsto\m_{\cR,\l}^*$ as a composition of continuous functions of $\l$.

\subsection{Proof of Proposition~\ref{prop:contin-gap}}\label{soft_p_3}
As far as $\phi_{\cR,\l}$ is concerned, continuity and monotonicity
follow from continuity and monotonicity of $e_{\cR,\l}(x)$ for any $x\in\cR$.
We then consider the other parameters.
The continuity follows from the continuity of the eigenvalues
as root of the characteristic polynomial.
To prove the monotonicity, we notice that
 when $\l$ decreases to zero,
$p_{\cR,\lambda}^*(x,y)$ grows for all $x$ and $y$ in $\cR$
as well as $c_{\cR,\l}(x,y)$ for any distinct $x,y\in\cR$.
From the variational characterization of $\phi_{\cR,\l}^*$, i.e.
\begin{eqnarray}
\phi_{\cR,\l}^* & = & \min\left\{\langle f,-\lL f \rangle_\cR
\,:\,
\langle f,f \rangle_\cR=1\,,\,f>0\right\}\\
& = &\min_{
  \stackrel%
  {{\scriptstyle \langle f,f \rangle_\cR=1}}%
  {{\scriptstyle f>0}}%
}
  \sum_{x,y\in\cR}\mu_\cR(x)f(x)\left(f(x)-\sum_{y\in\cR}p_{\cR,\lambda}^*(x,y)f(y)\right)
\label{phi-lambda}
\end{eqnarray}
where the restriction $f>0$ comes from the fact that, by the Perron-Frobenius theorem,
the right eigenvector has positive coordinates,
we see that $\phi_{\cR,\l}^*$ is decreasing in $\l$.
Similarly, using
\begin{equation}\label{gap-lambda}
\g_{\cR,\l}=\min\left\{\sfrac{1}{2}\sum_{x,y\in\cR}c_{\cR,\l}(x,y)
(f(x)-f(y))^2\,:\, \var_{\m_\cR}(f)=1\right\}\,,
\end{equation}
we see that $\g_{\cR,\l}$ is increasing in $\l$.
As a consequence $\e_{\cR,\l}^*$ is decreasing in $\l$, and we have
\begin{equation}
\e_{\cR,0}^*=\frac{\phi_{\cR,0}^*}{\g_{\cR,0}}=
\frac{\m_{\cR,0}^*(e_{\cR,0})}{\g_{\cR,0}}=0\,.
\end{equation}

\subsection{Proof of Theorem~\ref{th:t6}}\label{soft_p_4}
Proof of (\ref{t6prima}):  We first write
\bea
\P_\nu(X(\t_\d)=x\,|\,X(\t_\d)\in\cR)
&=&\frac{1}{\P_{\nu}(X(\t_\d)\in\cR)}
\sum_{i\geq 0}\sum_{x_i\in\cX}
\P_\nu(i_0>i, X(\t_i)=x_i)\\
&\times&\P_{x_i}(X\circ \ell^{-1}_{\cR}(\s_\kappa)=x,
\,\ell^{-1}_\cR(\s_\kappa)<\ell^{-1}_{\XR}(\s_\l)\,,\,\s_\kappa> T_{\d,\cR,\l}^*)
\nonumber
\eea
Now, conditioning on $\s_\kappa$ and setting $\P_{x_i}^{\s_\kappa}=\P_{x_i}(\cdot\,|\,\s_\kappa)$,
we get
\begin{equation}\label{pippo}
\begin{split}
\P_\nu(X(\t_\d)&=x\,|\,X(\t_\d)\in\cR)
=\frac{1}{\P_{\nu}(X(\t_\d)\in\cR)}
\sum_{i\geq 0}\sum_{x_i\in\cX}
\P_\nu(i_0>i, X(\t_i)=x_i)\\
&\times\E\left[\P_{x_i}^{\s_\kappa}(X\circ \ell^{-1}_{\cR}(\s_\kappa)=x,
\,\ell^{-1}_\cR(\s_\kappa)<\ell^{-1}_{\XR}(\s_\l)\,,\,\s_\kappa> T_{\d,\cR,\l}^*)\right]\\
&=\frac{1}{\P_{\nu}(X(\t_\d)\in\cR)}
\sum_{i\geq 0}\sum_{x_i\in\cX}
\P_\nu(i_0>i, X(\t_i)=x_i)\\
&\times\E\left[\1_{\{\s_\kappa> T_{\d,\cR,\l}^*\}}\P_{x_i}^{\s_\kappa}(X\circ \ell^{-1}_{\cR}(\s_\kappa)=x\,|\,
\s_\kappa<\t_{\XR,\l})\P_{x_i}^{\s_\kappa}(\s_\kappa<\t_{\XR,\l})\right]\,,
\end{split}
\end{equation}
where the second equality comes from the independence  between $X$, $\s_\kappa$
and $\s_\l$.
Since
\begin{equation}
\P_{\nu}(X(\t_\d)\in\cR)=
\sum_{i\geq 0}\sum_{x_i\in\cX}
\P_\nu(i_0>i, X(\t_i)=x_i)
\E\left[\1_{\{\s_\kappa> T_{\d,\cR,\l}^*\}}\P_{x_i}^{\s_\kappa}(\s_\kappa<\t_{\XR,\l})
\right]\,,
\end{equation}
from (\ref{pippo}) we get
\begin{equation}
\begin{split}
&\frac{\P_\nu(X(\t_\d)=x\,|\,X(\t_\d)\in\cR)}{\m_{\cR,\l}^*(x)}-1
= \frac{1}{\P_{\nu}(X(\t_\d)\in\cR)}
\sum_{i\geq 0}\sum_{x_i\in\cX}
\P_\nu(i_0>i, X(\t_i)=x_i)\\
&\times\E\left[\1_{\{\s_\kappa> T_{\d,\cR,\l}^*\}}
\P_{x_i}^{\s_\kappa}(\s_\kappa<\t_{\XR,\l})
\left(\frac{\P_{x_i}^{\s_\kappa}(X\circ \ell^{-1}_{\cR}(\s_\kappa)=x\,|\,\s_\kappa<\t_{\XR,\l})}
{\m_{\cR,\l}^*(x)}-1\right)\right]\,.
\end{split}
\end{equation}
An analogous expression can be found for
$\frac{\P_\nu(X(\t_\d)=x\,|\,X(\t_\d)\in\XR)}{\m_{\XR,\kappa}^*(x)}-1$.
The result then follows from Theorem~\ref{th:generalization},
and in particular from the equivalent of Theorem~\ref{th:dinamica}.

To prove inequality (\ref{t6seconda}), we first state the following lemma.
\begin{lemma}\label{lem:exponential}
Let $T>0$ and $\{\s_i:\,i\geq 1\}$ be a sequence of independent
exponential random variables of rate $\kappa$ such that $e^{\kappa T}-1<1$.
If $N=\min\{i\geq 1\,:\,\s_i>T\}$, then
\begin{equation}\label{exponential}
\P\left(\sum_{i=1}^N \s_i>\frac{t}{\kappa}\right)\leq
\frac{e^{-t}}{1-(e^{\kappa T}-1)}
\end{equation}
\end{lemma}
\begin{proof}[Proof of Lemma \ref{lem:exponential}]
Using the property of the exponential distribution, we have
\begin{equation}
\begin{split}
\P\left(\sum_{i=1}^N \s_i>\frac{t}{\kappa}\right)
&= \sum_{n\geq 1}\P(N=n)\P\left(\sum_{j=1}^n \s_i>\frac{t}{\kappa}\,
\big{|}\,\s_1<T,\ldots,\s_{n-1}<T, \s_n>T\right)\\
&\leq\sum_{n\geq 1}\P(N=n)\P\left(\s_n>\frac{t}{\kappa}-(n-1)T\,
\big{|}\s_n>T\right)\\
&=\sum_{n\geq 1}\P(N=n)\P\left(\s_n>\frac{t}{\kappa}-nT\right)\\
&=\sum_{n\geq 1} (1-e^{-\kappa T})^{n-1}e^{-\kappa T}e^{-t+n\kappa T}\\
&=e^{-t} \sum_{n\geq 1} (e^{\kappa T}-1)^{n-1}=
\frac{e^{-t}}{1-(e^{\kappa T}-1)}\,,
\end{split}
\end{equation}
which concludes the proof.
\end{proof}

Coming back to the proof of Th. \ref{th:t6}, we first notice
that if $\t_\d>t(\frac{1}{\kappa}+\frac{1}{\l})$, then
$\ell_{\cR}(\t_\d)>\frac{t}{\kappa}$ or $\ell_{\XR}(\t_\d)>\frac{t}{\l}$.
As a consequence, defining
\begin{equation}
\ba{l}
A_\cR=\{\kappa \ell_\cR(\t_\d)\vee \l \ell_{\XR}(\t_\d)=\kappa \ell_\cR(\t_\d)>t \}\\
A_{\XR}=\{\kappa \ell_\cR(\t_\d)\vee \l \ell_{\XR}(\t_\d)=\l \ell_{\XR}(\t_\d)>t \}
\ea
\end{equation}
so that $\P(A_\cR)+\P(A_{\XR})\leq 1$, we have
$$
\P_\nu(\t_\d>t(\sfrac{1}{\kappa}+\sfrac{1}{\l}))=
\P_\nu(\t_\d>t(\sfrac{1}{\kappa}+\sfrac{1}{\l})|\,A_\cR) \P_\nu(A_\cR)+
\P_\nu(\t_\d>t(\sfrac{1}{\kappa}+\sfrac{1}{\l})|A_{\XR})\P_\nu(A_{\XR})\,.
$$
Using the independence between $\s_\kappa$, $\s_\l$ and $X$,
 together with the previous lemma, we finally get
 $$
\P_\nu\left(\t_\d>t\left(\sfrac{1}{\kappa} +\sfrac{1}{\l}\right) \right)
\leq e^{-t}\left\{\sfrac{1}{1-\xi} \right\}
\left(\P_\nu(A_\cR)+\P_\nu(A_{\XR})\right)
\leq e^{-t}\left\{\sfrac{1}{1-\xi} \right\}\,.
  $$

\subsection{Proof of Theorem~\ref{th:t8}}\label{soft_p_5}
To prove the upper bound
we consider the extended electrical network associated
with $C_\kappa^\l(A,B)$
and follow the first steps of the proof of the upper bound
in Theorem \ref{th:phi&cap} (see subsection~\ref{sole}).
We then reach, for some probability measure~$\nu$ on $\mathcal{R}$,
to
\begin{equation}
	1+\kappa\E_{\nu}[\t_{\XR, \lambda}]
	=\frac{\kappa\m(\cR)}{C_\kappa^\lambda(\cR,\XR)}
	\,.
\end{equation}
instead of equation~(\ref{picche}).
Using then, from Theorem~\ref{th:t7},
the analogous of equation~(\ref{banana}),
we obtain
\begin{equation}
	1 + \frac{\kappa}{\phi^*_{\mathcal{R}, \lambda}}\left\{
		1 + \varepsilon^*_{\mathcal{R}, \lambda}
		+ \varepsilon^*_{\mathcal{R}, \lambda} \ln \frac{
			1
		}{
			\varepsilon^*_{\mathcal{R}, \lambda} \zeta_{\mathcal{R}}
		}
	\right\}
	\geq \frac{
		\kappa \mu(\mathcal{R})
	}{
		C_\kappa^\lambda(\mathcal{R}, \mathcal{X}\setminus\mathcal{R})
	}
	\,,
\end{equation}
or
\begin{equation}
	\frac{1}{\phi^*_{\mathcal{R}, \lambda}}\left\{
		1 + \varepsilon^*_{\mathcal{R}, \lambda}
		+ \varepsilon^*_{\mathcal{R}, \lambda} \ln \frac{
			1
		}{
			\varepsilon^*_{\mathcal{R}, \lambda} \zeta_{\mathcal{R}}
		}
		+ \frac{
			\phi^*_{\mathcal{R}, \lambda}
		}{
			\kappa
		}
	\right\}
	\geq \frac{
		\mu(\mathcal{R})
	}{
		C_\kappa^\lambda(\mathcal{R}, \mathcal{X}\setminus\mathcal{R})
	}
	\,,
\end{equation}
which gives the desired upper bound on $\phi^*_{\mathcal{R}, \lambda}$.

The proof of the lower bound will be similar to
the proof of  the lower bound of Theorem \ref{th:phi&cap},
where we used a partial Poincar\'{e} inequality
to control the mean exit time from $\cR$.
The difference here is that we will have to work
on the whole space $\cX$ and not only on $\cR$.
Since $\m_{\cR,\l}^*$ is concentrated on $\cR$,
we will first compare its associated exit time
$\tau_{\mathcal{X}\setminus\mathcal{R}, \lambda}$
with the exit time of another quasi-stationary measure,
$\ti \mu_{\cX}^*$, that spreads on the whole space $\cX$.
Then we will control $\tilde\phi_\cX^*$,
the escape rate from $\cX$,
with the spectral gap
estimated in Theorem \ref{th:gap&cap}.

Let $\ti\mu_\cX^*$ be the quasi-stationary measure on $\cX$ associated
with the Markovian process $\ti X$ on $\ti\cX= \cX\cup \overline{\XR}$ with generator
$\ti\cL$ defined, for some $\ti\l >0$, by
\begin{equation}\label{gen-tilde}
(\ti\cL f)(\ti x) =
\left\{
\ba{ll}
(\cL f)(x) & \mbox{if } \ti x=x\in\cR\\
(\cL f)(x)+\l(f(\bar x)-f(x)) & \mbox{if } \ti x=x\in\XR\\
\tilde\l(f(\ti x)-f(x)) & \mbox{if } \ti x=\bar x\in\overline{\XR}
\ea
\right.\,.
\end{equation}
The associated escape rate $\tilde\phi^*_{\mathcal{X}}$
is, with obvious notation and for any probability measure $\nu$ on $\mathcal{X}$,
the rate of exponential decay of
$P_\nu(\tilde\tau_{\overline{\mathcal{X}\setminus\mathcal{R}}} > t)$
when $t$ goes to infinity.
Since
\begin{equation}
	\begin{split}
		P_{\mu_{\mathcal{R}, \lambda}^*}\left(
			\tilde\tau_{\overline{\mathcal{X}\setminus\mathcal{R}}} > t
		\right)
		& \geq P_{\mu_{\mathcal{R}, \lambda}^*}\left(
			\ell_{\mathcal{R}}\left(
				\tilde\tau_{\overline{\mathcal{X}\setminus\mathcal{R}}}
			\right) > t
		\right)\\
		& = P_{\mu_{\mathcal{R}, \lambda}^*}\left(
			\tau_{\mathcal{X}\setminus\mathcal{R}, \lambda} > t
		\right) = e^{-\phi^*_{\mathcal{R}, \lambda} t}
		\,,
	\end{split}
\end{equation}
we have $\phi^*_{\mathcal{R}, \lambda} \geq \tilde\phi^*_{\mathcal{X}}$.

We then have to estimate $\ti\phi_{\cX}^*$ and we do so
by comparison with the spectral gap.
By Lemma \ref{lem:Dirichlet} applied with $\cR=\cX$ and the correct normalizations,
and taking, with obvious notation, $f =\ti h_{\cX}^*$,
which is indeed the minimizer in the variational principle satisfied by $\tilde\phi_{\cX}^*$,
we have
\begin{equation}\label{erbetta}
	\ti\phi_\cX^*\geq\frac{\cD(\ti h_\cX^*)}{\|\ti h_\cX^*\|^2}
	=\frac{\|\ti h_\cX^*\|^2-1}{\|\ti h_\cX^*\|^2}
	\frac{\cD(\ti h_\cX^*)}{\var_\m(\ti h_\cX^*)}
	\geq \left(1- \frac{1}{\|\ti h_\cX^*\|^2}\right)\g.
\end{equation}
Now,
\begin{equation}\label{pietruccia}
	\begin{split}
		\|\tilde h_\cX^*\|^2
		& \geq
		\sum_{x\in\cR}\mu(x)\left(\frac{\tilde \mu_\cX^*(x)}{\mu(x)}\right)^2
		= \mu(\cR)\sum_{x\in\cR}\mu_\cR(x)\left(\frac{\tilde \mu_\cX^*(x)}{\mu(x)}\right)^2\\
		& \geq
		\mu(\cR)\left(\sum_{x\in\cR}\mu_\cR(x)\frac{\tilde \mu_\cX^*(x)}{\mu(x)}\right)^2
		= \frac{1}{\mu(\cR)} \left(\tilde \mu_\cX^*(\cR)\right)^2
		\,.
	\end{split}
\end{equation}
Since the escape from $\mathcal{X}$ occurs at rate $\lambda$
in each point of $\mathcal{X} \setminus \mathcal{R}$
and there are no direct connections between $\mathcal{R}$
and $\overline{\mathcal{X} \setminus \mathcal{R}}$,
one has
\begin{equation}
	\tilde\mu^*_{\mathcal{X}}(\mathcal{X} \setminus \mathcal{R}) \cdot \lambda
	= \tilde\phi^*_{\mathcal{X}} \leq \phi^*_{\mathcal{R}, \lambda}
\end{equation}
or
\begin{equation}
	\tilde\mu^*_{\mathcal{X}}(\mathcal{R})
	\geq \left\{
		1 - \frac{\phi^*_{\mathcal{R}, \lambda}}{\lambda}
	\right\}
	\,.
\end{equation}
From $\phi^*_{\mathcal{R}, \lambda} \geq \tilde\phi^*_{\mathcal{X}}$,
(\ref{erbetta}), Theorem~\ref{th:gap&cap} and~(\ref{pietruccia})
we obtain
\begin{equation}
	\phi^*_{\mathcal{R}, \lambda}
	\geq \left(
		1 - \frac{\mu(\cR)}{
			\left\{
				1 - \frac{\phi^*_{\mathcal{R}, \lambda}}{\lambda}
			\right\}^2
		}
	\right)
	\frac{C_\kappa^\lambda(\cR,\XR)}{\mu(\cR)(1-\mu(\cR))}
	\left\{
		\frac{1}{
			1 + \max\left(
				\frac{\kappa + \phi_\kappa^\lambda}{\gamma_\cR},
				\frac{\lambda + \phi_\kappa^\lambda}{\gamma_{\XR}}
			\right)
		}
	\right\}
	\,.
\end{equation}
Developing the square, dropping a few terms
and using the convexity of $x \mapsto 1/(1 + x)$,
this implies
\begin{equation}
	\phi^*_{\mathcal{R}, \lambda}
	\geq
	\frac{C_\kappa^\lambda(\mathcal{R}, \mathcal{X} \setminus \mathcal{R})}{\mu(\mathcal{R})}
	\left\{
		\frac{1 - \mu(\mathcal{R}) - 2 \phi^*_{\mathcal{R}, \lambda} / \lambda}{1 - \mu(\mathcal{R})}
	\right\}\left\{
		1 - \max\left(
			\frac{\kappa + \phi_\kappa^\lambda}{\gamma_{\mathcal{R}}},
			\frac{\lambda+ \phi_\kappa^\lambda}{\gamma_{\mathcal{X} \setminus \mathcal{R}}},
		\right)
	\right\}
	\,,
\end{equation}		
which is the desired result.

\subsection{Proof of Theorem~\ref{mix}}\label{soft_p_6}
By Theorem~\ref{th:generalization}, for all $x$ in $\cX$,
\begin{equation}
\max_{y\in\cR}\left|
  \frac{
    \P_x\left(
      X(\cT) = y \Big
      | \sigma_\lambda > T_{\XR,\kappa}^*
    \right)
  }{
    \mu^*_{\XR, \kappa}(y)
  }
  -1
\right|
\leq \varepsilon_{\XR,\kappa}^*.
\end{equation}
Then
\begin{equation}
	\begin{split}
		\|\nu_x-\mu^*_{\XR, \kappa}\|_{\mbox{{\tiny TV}}}
		& \leq \frac{1}{2}\varepsilon_{\XR,\kappa}^* + \P\left(\sigma_\lambda < T_{\XR,\kappa}^*\right)\\
		& = \frac{1}{2}\varepsilon_{\XR, \kappa}^* + 1 - e^{-\lambda T_{\XR, \kappa}^*}\\
		& \leq \frac{1}{2}\varepsilon_{\XR, \kappa}^* + \lambda T_{\XR, \kappa}^*
		\,.
	\end{split}
\end{equation}
Also
\begin{eqnarray}
	\|\mu^*_{\mathcal{X} \setminus \mathcal{R}, \kappa} - \mu_{\mathcal{X} \setminus \mathcal{R}}\|_{\rm TV}
	& \leq & \frac{1}{2}\sqrt{
		\frac{
			\varepsilon^*_{\mathcal{X} \setminus \mathcal{R}}
		}{
			1 - \varepsilon^*_{\mathcal{X} \setminus \mathcal{R}}
		}
	}
	\,,\\
	\|\mu_{\XR} - \mu\|_{\rm TV}
	& = & \mu(\cR)
	\,,
\end{eqnarray}
and the upper on $\|\nu_x - \mu\|_{TV}$
follows from the triangular inequality.

Now, for all $t >0$,
\begin{equation}
\|\P_x(X(t) = \cdot) - \mu\|_{\mbox{{\tiny TV}}} \leq \|\nu_x - \mu\|_{\mbox{{\tiny TV}}} + \P_x(\cT > t)
\end{equation}
and to prove our mixing time estimate, it is sufficient to show
\begin{equation}\label{bruxelles}
\P_x\left(\cT > t\right)
\leq \frac{1}{2}\left(\frac{1}{2} + \eta\right) - \mu(\cR) - \sqrt{
	\frac{
		\varepsilon^*_{\mathcal{X} \setminus \mathcal{R}}
	}{
		1 - \varepsilon^*_{\mathcal{X} \setminus \mathcal{R}}
	}
}
- \lambda T_{\XR,0}^*
= \frac{1}{4} - \frac{\mu(\mathcal{R})}{2}
\end{equation}
for
\begin{equation}\label{semaforo}
	t \geq \frac{2}{
	\left(
		\frac{1}{2} - \mu(\mathcal{R})
	\right) \phi_{\cR,\lambda}^*
}
\left\{
	1
	+ \varepsilon^*_{\mathcal{R}, \lambda}
	+ \varepsilon^*_{\mathcal{R}, \lambda} \ln \frac{1}{\varepsilon^*_{\mathcal{R}, \lambda} \zeta_R}
	+ \frac{\phi^*_{\mathcal{R}, \lambda}}{\lambda}
\right\}
\,.
\end{equation}

To obtain such an estimate
we give an upper bound on the mean value
of $\cT$ and use Markov inequality.
With $\cT' = \sigma_\lambda\wedge\tau_\cR$,
we have, using~(\ref{banana}) adapted to soft measures,
\begin{equation}
	\begin{split}
		\E_x[\cT]
		&\leq \E[\sigma_\lambda] + \E_x\left[\E_{X(\cT')}\left[\tau_{\cR,\lambda}\right]\Big|X(\cT')\in\cR\right]\\
		&\leq \frac{1}{\lambda} + \frac{1}{\phi_{\cR,\lambda}^*}\left\{
			1
			+ \varepsilon_{\cR,\lambda}^*
			+ \varepsilon_{\cR,\lambda}^* \ln \frac{1}{\varepsilon^*_{\mathcal{R}, \lambda} \zeta_{\mathcal{R}}}
		\right\} \\
		& =  \frac{1}{\phi_{\cR,\lambda}^*}\left\{
			1
			+ \varepsilon_{\cR,\lambda}^*
			+ \varepsilon_{\cR,\lambda}^* \ln \frac{1}{\varepsilon^*_{\mathcal{R}, \lambda} \zeta_{\mathcal{R}}}
			+ \frac{\phi^*_{\mathcal{R}, \lambda}}{\lambda}
		\right\} \,,
	\end{split}
\end{equation}
so that~(\ref{semaforo}) implies~(\ref{bruxelles}).

\section{Two examples}\label{trotta}
In this section we want to illustrate
the analysis method of Section~\ref{model_and_results}
with reference to toy models.
We will recover known results for the Glauber dynamics of the Curie-Weiss model
and give sharp asymptotics of its relaxation time,
we will also study a variation on the so-called ``$n$-dog'' theme
considered in~\cite{SC}
that illustrates the variety of scenarios one can encounter in proving our basic hypothesis
on $\varepsilon_{\mathcal{R}}^*$ or controlling $T_{\mathcal{R}}^*$.

\subsection{Metastable behavior of the Curie-Weiss model}\label{sec:CW}
\subsubsection{Model, dynamics and one-dimensional representation.}
Les us consider the Curie-Weiss model which is a mean-field spin system described
by $N$ spin variables, $\s=(\s_1,\ldots,\s_N)\in\cX=\{-1,1\}^N$, % $\s_i\in\{-1,+1\}$, $i=1,\ldots N$,
with  Hamiltonian
\begin{equation}
H_{N,h}(\s)  = -\frac{1}{2N} \sum_{i,j=1}^N
\s_i\s_j-h\sum_{i=1}^N \s_i\,,
\end{equation}
where $h>0$ is called the external field.
The corresponding Gibbs probability measure on $\cX$ is
\begin{equation}
\mu_{N,h,\b}(\s)= \frac{e^{-\b H(\s)}}{Z_{N,h,\b}},
\end{equation}
where $\b >0$ is the inverse of the temperature, and $Z_{N,h,\b}=\sum_{\s\in\cX} e^{-\b H_{N,h}(\s)}$ is
the normalizing factor called the partition function.
%\begin{equation}
%Z_{N,h,\b} =\sum_{\s\in\cX} e^{-\b H_{N,h}(\s)}\,.
%\end{equation}
To make the notation simpler, we  set $H(\s)\equiv H_{N,h}(\s)$, $\mu(\s)\equiv \mu_{N,h,\b}(\s)$
and $Z\equiv Z_{N,h,\b}$.

For every $N\in \N$, and setting $[-1,1]_N :=\{-1,-1+\sfrac{2}{N},\ldots, 1\}$,
let us define the total magnetization, $m_N: \cX\mapsto [-1,1]_N $, as
\begin{equation}
 m_N(\s)\equiv \frac 1N\sum_{i=1}^N\s_i\,.
\end{equation}
Notice that it allows rewriting the Hamiltonian as a function of a one-dimensional parameter, i.e.
\begin{equation}
H(\s)= N u(m_N(\s))\,,
\end{equation}
with $u(m)= -\sfrac{m^2}{2}-hm$,  for $m\in[-1,1]$.
For simplicity, in the sequel we will identify functions defined on the
discrete set $[-1,1]_N$ with functions defined on $[-1,1]$ by setting
$f(m)\equiv f([2 N  m]/2N)$.

For $m\in[-1,1]$,  let us consider the functions
\begin{equation}
f_N(m)= -\frac{1}{\b N}\ln\sum_{\s\,:\, m_N(\s)=m}e^{-\b H(\s)}
\end{equation}
\begin{equation}
f(m)=\lim_{N\rightarrow\infty}f_N(m)\,.
\end{equation}
A standard computation shows that
\begin{equation}\label{def:freeE}
\begin{split}
f_N(m)&= u(m)-\frac{1}{\b}s(m)+\frac{1}{\b N}\ln\left( \sqrt{\sfrac{(1-m)^2\pi N}{2}}(1+o(1))\right)\\
&= f(m)+\frac{1}{\b N}\ln \left(\sqrt{\sfrac{(1-m)^2\pi N}{2}}(1+o(1))\right)\,,
\end{split}
\end{equation}
where $s(m)=-\left(\sfrac{1+m}{2}\ln\sfrac{1+m}{2}+\sfrac{1-m}{2}\ln\sfrac{1-m}{2}\right)$
is the entropy of Bernoulli random variables.
Moreover, the critical points of the function $f(m)$ satisfy the equation
\begin{equation}
m=\tanh(\b(m+h))\,
\end{equation}
and one gets, for $\b>1$ and $0<h< \sqrt{\sfrac{\b-1}{\b}}+\sfrac{1}{\b}\ln\left(\sqrt\b +\sqrt{\b-1}\right)$,
that the graph of $f(m)$ is given by a double-well with two minima $m_-<0<m_+$
and a maximum $m_0<0$.

We then consider the time evolution  of this system provided by a
heat-bath Glauber dynamics. This is a Markov chain on $\cX$,
denoted by $X=(X(t))_{t\geq 0}$, defined through the following (normalized) rates
\begin{equation} \label{def:ratesmicro}
	\left\{
		\begin{array}{l}
			p(\s,\s^i)= \frac{1}{N}\frac{e^{-\b H(\s^i)}}{e^{-\b H(\s)}+e^{-\b H(\s^i)}}\,,
			\quad \mbox{for } i=1,\ldots,N\\
			p(\s,\s)= 1-\sum_{i=1}^N p(\s,\s^i) \\
			p(\s,\s')=0\,,\qquad\qquad\qquad\mbox{elsewhere}
		\end{array}
	\right.
\end{equation}
where $\s^i$ denotes the configuration obtained from $\s$
by a spin-flip at the site $i$. An easy check shows
that $X$ is reversible w.r.t $\m$.

It turns out that the induced  dynamics on the space $[-1,1]_N$,
$\bar X(t):= m_N(X(t))$,  is also Markovian with transition rates
\begin{equation}\label{def:ratesmacro}
	\left\{
		\begin{array}{l}
			\bar p(m,m\pm \sfrac{2}{N})= \left(\frac{1\mp m}{2}\right)\left(\frac{1+\tanh(\b\D_{\pm})}{2}\right)\\
			%\frac{e^{\b \D_{\pm}(m)}}{e^{\b \D_{\pm}(m)}+e^{-\b \D_{\pm}(m)}}\\
			\bar p(m,m)= 1-p(m,m+\sfrac{2}{N})-p(m,m-\sfrac{2}{N}) \\
			\bar p(m,m')=0\,,\qquad\qquad\qquad\mbox{elsewhere}
		\end{array}
	\right.
	\,,
\end{equation}
where $\D_{\pm}:= \pm m \pm h +\sfrac{1}{N}$.
Moreover, an easy check shows that the induced dynamics is reversible w.r.t.
the probability measure $\bar\m$ on $[-1,+1]_N$, given by
\begin{equation}\label{def:measuremacro}
\bar\m(m):=\frac{e^{-\b N f_N(m)}}{Z_N}= \sum_{\s\,:\,m_N(\s)=m}\m(\s)\,,
\end{equation}
where $f_N$ was defined in (\ref{def:freeE}).
When parameterized by $m$, the evolution of our system can be viewed as a
one-dimensional random walk driven by a double-well potential.
We thus consider the metastable region
$\cR\subset \cX$ and the corresponding one-dimensional projection
$\bar\cR\subset [-1,+1]_N$:
\begin{equation}
\cR:= \left\{\s\in \cX\,:\, m_N(\s)\leq m_0\right\}\quad;\quad
\bar\cR:= \left\{m\in [-1,1]_N\,:\, m\leq  m_0\right\}\,.
\end{equation}

\subsubsection{Verifying hypotheses: First part}
In order to apply  our Theorems (\ref{th:tempo}) (and the related inequality (\ref{banana}),
we first have to verify the hypotheses and in particular  provide a suitable
upper bound on $\e_\cR^*=\frac{\phi_\cR^*}{\g_\cR}$ and on $\zeta_\cR^*$.

By Lemma \ref{lem:phiphi} we get
\begin{equation}\label{stimaphi}
\begin{split}
\phi_\cR^*\leq \phi_\cR= \m_\cR(e_\cR)\leq \m_\cR(\partial_- \cR)
= \frac{\bar\mu(m_0)}{\mu(\cR)}\leq
\exp{(-\b N \G)}(1 +o(1)) \,,
\end{split}
\end{equation}
where in the last inequality we set $\G:= f(m_0)-f(m_-)$ and used (\ref{def:freeE}).
The hypothesis on $\e_\cR^*$ then follows immediately
by applying the $(N\log N)^{-1}$ lower bound on $\g_\cR$ that was derived in \cite{LLP} by a very precise computation.
Since we just need a rough control of this quantity,
that will be then compared to $\phi_\cR^*$, we provide here a new simpler argument
that yields a bound of order $N^{-3/2}$.

We first notice that the dynamics defined by (\ref{def:ratesmicro}) can be compared
to a random walk on the hypercube.
This suggests that a simple way to control the spectral gap is by mixing time estimates.
To this aim we consider the dynamics reflected in $\cR$, $X_\cR=(X_\cR(t))_{t\geq 0}$,
and  the related mixing time,
\begin{equation}
\t_{mix,\cR}(\sfrac{1}{4})=
\inf_{t\geq 0}\{\max_{\s\in\cR}
\|\P_\s(X_\cR(t)=\,\cdot\,)-\mu_\cR \|_{TV}\leq \sfrac{1}{4}\}\, .
\end{equation}

In what follows we will denote by $c(\b)$ a constant depending on $\b$
but independent of $N$, whose particular value may change from line to line.
With  the above notation it holds the following:
\begin{proposition}\label{prop:tmixing}
\begin{equation}\label{eq:tmix}
\t_{mix,\cR}(\sfrac{1}{4})\leq c(\b) N^{\frac{3}{2}}\,.
\end{equation}
\end{proposition}
\begin{proof}
The idea of the proof is based on coupling techniques, and we thus
define the following coupling:
\begin{itemize}
\item[a)] We  consider two independent dynamics
$X^\s_\cR$ and $X^\eta_\cR$,
with initial states $\s,\eta \in\cR$, and
let them run independently until they reach the same magnetization.
%we get $m_N(X^\s_\cR(t))=m_N(X^\eta_\cR(t))$.
\item[b)] We then run an analogous coupling  to that defined in \cite{LLP},
 where the only difference here is the reflection on $\cR$.
This coupling is  defined for initial states
$\s,\s'\in \cR$ with $m_N(\s)=m_N(\s')$, and is defined in such a way so as to keep
the magnetizations coupled along the dynamics.
\end{itemize}
Let $T_{\s,\eta}$ denote the coupling time for the global coupled
dynamics $(X^\s_\cR(t), X^\eta_\cR(t))$, that is
\begin{equation}\label{eq:tcoupglobal}
T_{\s,\eta}=\inf\{t\geq 0 \,:\,X_\cR^\s(t)=X_\cR^{\eta}(t)\}\,.
\end{equation}
To provide an estimate on $\t_{mix,\cR}(\sfrac{1}{4})$, it is then enough to find
$t$ such that
\begin{equation}\label{tcoupling}
\max_{\s,\eta\in\cR} \P(T_{\s,\eta}>t)\leq \sfrac{1}{4}\,.
\end{equation}
The  coupling time $T_{\s,\eta}$ can be controlled by first estimating the time
to couple the magnetizations, %that is
%smaller time $t$ to get $m_N(X^\s_\cR(t))=m_N(X^\eta_\cR(t))$,
and then the coupling time
of the dynamics starting in configurations with equal magnetization.

Formally, let $( X^m_{\bar\cR}(t), X^{m'}_{\bar\cR}(t))$ be the induced coupled dynamics
 with initial states $m,m'\in\bar\cR$, and define
\begin{equation}
\bar T_{m,m'}=\inf\{t\geq 0\,:  X^m_{\bar\cR}(t)= X^{m'}_{\bar\cR}(t)\}\,.
\end{equation}
Then, for any $\s,\eta\in\cR$,
\begin{equation}\label{tcoupsplit}
\begin{split}
\P(T_{\s,\eta}>t)&\leq \P\left(\max_{m,m'\in\bar\cR}\bar T_{m,m'}
+\max_{\s,\s'\in\cR:\atop  m_N(\s)=m_N(\s')}T_{\s,\s'}>t\right)\\
&\leq \max_{m,m'\in\bar\cR}\P(\bar T_{m,m'}>\sfrac{t}{2})+ \max_{\s,\s'\in\cR: \atop  m_N(\s)=m_N(\s')}
\P( T_{\s,\s'}>\sfrac{t}{2})
\end{split}
\end{equation}

Following the same argument of \cite{LLP} ({\it see Lemma 2.9 and its proof})
it is easy to prove that, for any $\s,\s'\in\cR$ such that $ m_N(\s)=m_N(\s')$,
\begin{equation}\label{eq:tcoupmacro}
\P(T_{\s,\s'}> c(\b) N\log N)\leq \sfrac{1}{N}\,
\end{equation}
for a constant $c(\b)$  depending on $\b$ but not in $N$.

To control the time $\bar T_{m,m'}$, let $\t_m$ denote the stopping time in $m$ for $X_{\bar\cR}$.
With some abuse of notation,  let $\E_m$ denote the average over the dynamics $X_{\bar\cR}$
with initial state $m\in\bar\cR$ and define
\begin{equation}\label{def:T}
T:= \max_{m\in\bar\cR}\, \E_m(\t_{m_-})= \max\left\{\E_{-1}(\t_{m_-});\E_{m_0}(\t_{m_-})\right\}
\end{equation}
where the second equality is due to an obvious geometric fact.
Then the following Lemmas hold.
\begin{lemma}\label{lem:Tmedio}
With the above notation, it holds
\begin{equation}
T\leq c(\b) N^{\frac{3}{2}}\,.
\end{equation}
\end{lemma}
\begin{lemma}\label{lem:Tcoupling}
For all $t\geq 40 T$, it holds
\begin{equation}\label{eq:lemTC}
\max_{m,m'\in\bar\cR}\P(T_{m,m'}>t)= \P(T_{-1,m_0}> t)\leq \sfrac{1}{2}\,.
\end{equation}
\end{lemma}
Before proving the above Lemmas, let us conclude the proof of Prop. \ref{prop:tmixing}.\\
By inequalities (\ref{tcoupsplit})-(\ref{eq:tcoupmacro}) and  Lemma \ref{lem:Tcoupling}, it follows
that if $t= \max\{240 T,N^{\frac{3}{2}}\}$, $N\geq 8$, and for any $\s,\eta\in\cR$,
\begin{equation}
\begin{split}
\P(T_{\s,\eta}>t)&
\leq \P(\bar T_{-1,m_0}> 120T)+\max_{\s,\s'\in\cR:\atop  m_N(\s)=m_N(\s')}
\P( T_{\s,\s'}> N^{\frac{3}{2}})\\
& \leq \frac{1}{8} +\frac{1}{N}\leq \frac{1}{4}
\end{split}
\end{equation}
By Lemma \ref{lem:Tmedio} the above inequality holds whenever $t> c(\b)N^{\frac{3}{2}}$
and the statement of the Proposition follows.
\end{proof}
We now come back to the proofs of the two Lemmas.
\begin{proof}[Proof of Lemma \ref{lem:Tmedio}]
Since $X_{\bar\cR}$ is a one-dimensional dynamics, for any two states $x,y\in\bar\cR$
we have the formula
\begin{equation}\label{uno}
\E_{x} (\t_{y})=\frac{\bar\m(V_{x,y})}{C(x,y)}
\end{equation}
where $V_{x,y}$ and $C(x,y)$ are, respectively, the equilibrium potential and
the capacity between $x$ and $y$. Moreover, if $x< y$, we have
\begin{equation}\label{due}
V_{x,y}(m)= \P_m(\t_x<\t_y)=\left\{
\ba{ll}
1 & \mbox{if } m\leq x\\
0 & \mbox{if } m\geq y\\
\frac{C(x,y)}{C(m,y)} &  \mbox{if } x<m< y
\ea
\right.\,,
\end{equation}
\begin{equation}\label{tre}
C(x,y)^{-1}= \sum_{k=0}^{(y-x)\sfrac{N}{2}-1}
\left(\bar c\left(x+\sfrac{2k}{N}, x+\sfrac{2(k+1)}{N}\right)\right)^{-1}\,,
\end{equation}
with $\bar c(x,y)=\bar\m(x)\bar p(x,y)$.
Analogous formulas hold when $x> y$.
\begin{remark}
Since we are considering the dynamics reflecting in $\bar\cR$,
the classical version for the mean exit time would be
$\E_{x} (\t_{y})=\frac{\bar\m_{\bar\cR}(V_{x,y})}{C_{\bar\cR}(x,y)}$
rather than Eq. (\ref{uno}),
where $C_{\bar\cR}(x,y)$ is the capacity defined through conductances
$\bar c(x,y)=\bar\m_{\bar\cR}(x)\bar p_{\bar\cR}(x,y)$.
However, it is easy to verify that for points $x,y\in\bar\cR$
the two formulas are equivalent.
\end{remark}

In Appendix B we will show that, if there are no local maxima of $f_N$ in $[x,y]$,
\begin{equation}\label{quattro}
C(x,y)^{-1}\leq
 c(\b)\sqrt{N} Z_N \max_{z\in\{x,y\}}e^{\b N f_N (z)}\,.
\end{equation}
Putting together (\ref{uno})-(\ref{quattro}), and since $f(y)>f(m_-)$ for any $y\in[-1,m_-)$,
we get
\begin{equation}
\begin{split}
\E_{-1}(\t_{m_-})&=\sum_{j=0}^{(m_-+1)\sfrac{N}{2}-1}
\bar\mu_\cR(-1+\sfrac{2j}{N})C(-1+\sfrac{2j}{N}, m_-)^{-1}
\leq c(\b) N^{\frac{3}{2}}\end{split}
\end{equation}
Analogous computations can be done for  $\E_{m_0}(\t_{m_-})$, providing the same
estimate. This concludes the proof of the Lemma.
\end{proof}
\begin{proof}[Proof of Lemma \ref{lem:Tcoupling}]
The first identity of (\ref{eq:lemTC}) is obvious, due to the geometry of the problem.

We then focus on the two dynamics $X_{\bar\cR}^{-1}$ and $X_{\bar\cR}^{m_0}$, and
define recursively the stopping times $s_k, \t_k$ and $s_k'$, for $k\geq 1$:
\begin{equation}\label{def:stoppingtimes}
\begin{split}
&s_1:= \inf_{t\geq 0}\{X_{\bar\cR}^{-1}(t)=m_-\}\\
&\t_k:= \inf_{t\geq s_k}\{X_{\bar\cR}^{m_0}(t)=m_-\}\\
&s_{k+1}:= \inf_{t\geq \t_k}\{X_{\bar\cR}^{-1}(t)=m_-\}\\
&s'_{k}:= \sup_{t\leq \t_k}\{X_{\bar\cR}^{-1}(t)=m_-\}
\end{split}
\end{equation}
Letting $\t(t)$ denote the first clock ring after time $t$ , we can define the event
\begin{equation}\label{evento}
A:= \{\exists k\leq 2\,:\, s'_k=\t_k\,\,\mbox{or}\,\, X_{\bar\cR}^{-1}(s_k'+\t(s_k'))>m_-\}\,.
\end{equation}
Notice that, since $s'_k$ is the time of the last visit in $m_-$ of the dynamics
$X_{\bar\cR}^{-1}$ before $\t_k$, and because $-1<m_-<m_0$, the occurrence of the event
$\{X_{\bar\cR}^{-1}(s_k'+\t(s_k'))>m_-\}$ implies that $T_{-1,m_0}<\t_k$.
In particular, we have $A\cup\{\t_2\leq t\}\subset\{T_{-1,m_0}\leq t\}$ and then
\begin{equation}\label{bo}
\P(T_{-1,m_0}>t)\leq \P(A^c)+\P(\t_2>t)\,.
\end{equation}
From the definition (\ref{def:ratesmacro}) of  rates $\bar p$, and using
that $\sfrac{1-m_-}{2}>\sfrac{1+m_-}{2}$ together with the properties
of the hyperbolic tangent, one can show that
\begin{equation}
\bar p(m_-,m_--\sfrac{2}{N})\leq \bar p(m_-,m_-+\sfrac{2}{N}) \Longleftrightarrow
 \bar p(m_-,m_--\sfrac{2}{N})\leq \sfrac{1}{2}(1- \bar p(m_-,m_-))\,.
\end{equation}
Thus
\begin{equation}
\begin{split}
\P(A^c)&\leq \P(X_{\bar\cR}^{-1}(s'_k+\t(s_k'))<m_-\,,\mbox{ for } k=1,2)\\
& \leq\left(\P(X_{\bar\cR}^{-1}(t+\t(t))=m_--\sfrac{2}{N}|
X_{\bar\cR}^{-1}(t)= m_-, X_{\bar\cR}^{-1}(t+\t(t))\ne m_-)\right)^{2}\\
&=\left(\frac{\bar p(m_-,m_--\sfrac{2}{N})}{1-\bar p(m_-,m_-)}\right)^2\leq \frac{1}{4}
\end{split}
\end{equation}

In order to estimate $\P(\t_2>t)$, we divide the interval $[0,t]$ in
$k=\lfloor \sfrac{t}{8T}\rfloor$ subintervals of length $8T$, where $T$
was defined in  (\ref{def:T}). The event $\{\t_2>t\}$ is then included
in the event that, in at least $k-2$ subintervals, at most one of the process
has arrivals in $m_-$.
On each interval, this happens with probability bounded above by
\begin{equation}
\P_{-1}(\t_{m_-}> 8T)+\P_{m_0}(\t_{m_-}> 8T)\leq \frac{1}{8T}(\E_{-1}(\t_{m_-})+\E_{m_0}(\t_{m_-}))\leq \frac{1}{4},
\end{equation}
 by  Markov's inequality, and then
 \begin{equation}
 \P(\t_2>t)\leq \binom{k-2}{k}\left(\frac{1}{4}\right)^{k-2}\leq 2^{-k+3} \,.
 \end{equation}
The statement follows taking  $t\geq 40T$, so that $k\geq 5$ and $\P(\t_2>t)\leq \sfrac{1}{4}$,
 and finally, from (\ref{bo}) and the previous estimates,
 \begin{equation}
 \P(T_{-1,m_0}>t)\leq \frac{1}{2}\,.
 \end{equation}
\end{proof}

Coming back to the hypotheses on $\e_\cR^*$,  from
the well known inequality $\g_\cR^{-1}\leq \t_{mix,\cR}(\sfrac{1}{4})$,
and  by  (\ref{stimaphi}) and (\ref{eq:tmix}), we obtain
\begin{equation}
\e_\cR^*=\frac{\phi_\cR^*}{\g_\cR}\leq c(\b)N^{\frac{3}{2}}\exp\left(-\b N\Gamma \right)(1+o(1))\,,
\end{equation}
which goes to 0 for any $N$ large enough, and thus satisfies
the hypothesis of our main theorems.
Moreover, from Lemma \ref{lem:zeta} and  the trivial bounds
$\D_\cR\geq N$ and $D_\cR\geq c(\b)$, we get
\begin{equation}
\zeta_\cR^*\geq e^{-\D_\cR D_\cR}\geq e^{-Nc(\b)}
\end{equation}
By inequality \ref{delta*phi*} and the previous estimates, this implies
that for $N$ large enough the condition $\phi_\cR^*\cdot T_\cR^*\ll 1$ is satisfied.

\subsubsection{Asympotic law of the exit time}
Applying Th. \ref{th:tempo} and by the related inequality (\ref{banana}),
 we get that in the limit  $N\rightarrow\infty$ and for all distributions $\nu$ on $\cR$,
\begin{equation}\label{CW:exittime}
	\left\{
		\begin{array}{l}
			\E_\nu (\t_{\XR})\leq {\phi_\cR^*}^{-1}(1+o(1))\\
			\E_\nu (\t_{\XR})\geq (1-\pi_\cR(\nu)){\phi_\cR^*}^{-1}(1+o(1))
		\end{array}
	\right.
\end{equation}
and for all $t>0$,
\begin{equation}\P_\nu (\phi_\cR^*\t_{\XR} >t )= (1-\pi_\cR(\nu))e^{-t}(1+o(1))\,.
\end{equation}
In particular, for $\nu=\m_\cR$,
\begin{equation}\label{CW:exittime2}
\E_{\m_\cR} (\t_{\XR})= {\phi_\cR^*}^{-1}(1+o(1))\end{equation}
\begin{equation}\P_{\m_\cR} (\phi_\cR^*\cdot\t_{\XR}>t )= e^{-t}(1+o(1))\,,\quad \forall t\geq 0\,.
\end{equation}

The next step concerns the estimation of $\phi_\cR^*$.
By  Th. \ref{th:phi&cap}, assuming that $N$ is large enough to have $\e_\cR^*+\phi_\cR^*\cdot T_\cR^*\ll 1$,
and choosing $k$ such that $\phi_\cR^*\ll k\ll \g_\cR$, we have
\begin{equation}\label{phicap}
\phi_\cR^*= \frac{C_k(\cR,\XR)}{\m(\cR)}\,.
\end{equation}
In order to estimate $C_k(\cR,\XR)$, we use its two variational characterizations,
(\ref{capacity}) and (\ref{cuore}), with suitable test functions.
The one-dimensional nature of the model suggests
that the capacities of the dynamics over $\cX$ could be well approximated
by the analogous quantities computed for the induced dynamics over $[-1,1]_N$.
This has the advantage that the equilibrium potential of the one-dimensional chain,
namely the minimizer in (\ref{capacity}) for $k,\l=\infty$,
can be explicitly given.

Following this idea, for the upper bound we consider the test function  $V(\s):= V_{m_-,m_0}(m_N(\s))$,
where $V_{m_-,m_0}$ is the function defined in (\ref{due}).
In other words, $V(\s)$ is the equilibrium potential associated to the
one-dimensional chain, with boundary conditions $V(m_-)=1$ and $V(m_0)=0$.
Explicitly,
for $m\in[-1,1]_N$,
\begin{equation}\label{potential}
V_{m_-,m_0}(m)= \P_m(\t_{m_-}<\t_{m_0})=\left\{
\ba{ll}
1 & \mbox{if } m\leq m_-\\
0 & \mbox{if } m\geq m_0\\
\frac{C(m_-,m_0)}{C(m,m_0)} &  \mbox{otherwise}
\ea
\right.\,,
\end{equation}
where $\displaystyle C(x,y)^{-1}= \sum_{k=0}^{(y-x)\sfrac{N}{2}-1}
\left(\bar c\left(x+\sfrac{2k}{N}, x+\sfrac{2(k+1)}{N}\right)\right)^{-1}$
and $\bar c(x,y)=\bar\m(x)\bar p (x,y)$.\\
Then we have
\begin{equation}\label{stimacapk}
\begin{split}
C_{k}&(\cR,\XR)\leq \cD(V)+ k \sum_{\s\in\cR}\m(\s)(V(\s)-1)^2\\
&= C(m_-,m_0)
+ k \sum_{m=m_-}^{m_0}\frac{e^{-\b N f_N(m)}}{Z_N}
\left(\frac{C(m_-,m_0)}{C(m_-,m)}\right)^2\\
&\leq  C(m_-,m_0) + k c(\b)N C(m_-,m_0)\,,
\end{split}
\end{equation}
where the last inequality is due to (\ref{quattro})
which is derived in Appendix B.
Since $k\ll \g_\cR\leq c(\b)N^{-3/2}$,
it holds
\begin{equation}
C_{k}(\cR,\XR)\leq C(m_-,m_0)(1+ o(1))\,.
\end{equation}

Similarly, for the lower bound on $C_{k}(\cR,\XR)$ we consider
a unitary test  flow $\psi$ which is constant on all couples $(\s,\s')$ of given magnetization.
Specifically  we set $\psi(\s,\s'):= \Psi(m_N(\s), m_N(\s'))$  and define
\begin{equation}\label{flow}
\left\{
\ba{ll}
\Psi(m, m+\sfrac{2}{N})= \left(S(m)\frac{(1-m)N}{2}\right)^{-1}& \forall\, m\in[m_-,m_0]_N \,,\\
\Psi(m,m')=0 & \mbox{otherwise}
\ea
\right.\,,
\end{equation}
with $S(m)=\left|\{\s\,:\, m_N(\s)=m\}\right|$.
With this definition, the flow $\Psi$ is the unitary flow from $m_-$ to $m_0$
that realized the minimum in the Thompson principle for the one-dimensional chain.
Inserting the test flow in (\ref{cuore}), we then  have
\begin{equation}\label{stimacapklow}
\begin{split}
C_{k}(\cR,\XR)^{-1}&\leq \cD(\psi) +\frac{\m(\cR)}{k}S(m_-)
\frac{e^{-\b N u(m_-)}}{\m(\cR)Z_N} \left(\frac{Z_N\cdot e^{-\b N u(m_-)}}{S(m_-)} \right)^2\\
&= \sum_{m=m_-}^{m_0} \frac{Z_N\cdot e^{\b N f_N(m)}}{\bar p(m,m+\sfrac{2}{N})}
+ \frac{1}{k}Z_N\cdot e^{\b N f_N (m_-)}\\
&= C(m_-,m_0)+ \frac{1}{k}\mu(m_-)^{-1}  \,.
\end{split}
\end{equation}
Since $k^{-1}\ll (\phi_\cR^*)^{-1}\leq \mu(m_-) C_{k}(\cR,\XR)^{-1}$
(by (\ref{phicap})), we get
\begin{equation}
C_{k}(\cR,\XR)\geq C(m_-,m_0)(1+ o(1))\,.
\end{equation}
From (\ref{phicap})) and with the above estimate, we then have
\begin{equation}
\phi_\cR^*= \frac{C(m_-,m_0)}{\mu(\cR)} (1+o(1))\,.
\end{equation}
Finally, $\m(\cR)$ and the capacity $C(m_-,m_0)$ defined in (\ref{tre})
can be both evaluated for large $N$ (see Appendix B),
providing the following asymptotic expressions:
\begin{equation}\label{capacityasynt}
C(m_-,m_0)=
\displaystyle\frac{\sqrt{(1-m_0^2)\b|f''(m_0)|}}{\pi N}\frac{e^{-\b N f(m_0)}}{Z_N}(1+o(1))
\end{equation}
\begin{equation}\label{m(R)}
\m(\cR)= \frac{e^{-\b Nf(m_-)}}{Z_N \cdot \sqrt{\b f''(m_-)(1-m_-^2)}}(1+o(1))\,.
\end{equation}
Altogether, under the same hypotheses of before, we have
\begin{equation}\label{lola}
\E_{\m_\cR}(\t_{\XR})=\frac{\pi N\cdot e^{\b N
(f(m_0)-f(m_-))}}{\b\sqrt{|f''(m_0)|f''(m_-)(1-m_0^2)(1-m_-^2)} }(1+o(1))
\end{equation}
\subsubsection{Verifying hypotheses: Second part}
To move to the second part of the analysis, which goes from
Th. \ref{th:gap&cap} to Th. \ref{mix}, we first have to estimate the quantities
$\phi_{\XR}^*$, $\g_{\XR}$ and $\e_\XR^*$ related to the dynamics over $\XR$.
As for $\phi_\cR^*$, we can find easily a rough (but sufficient) upper bound over $\phi_\XR^*$
by Lemma \ref{lem:phiphi}.
By trivial estimates we get
\begin{equation}\label{stimaphic}
\begin{split}
\phi_{\XR}^*\leq \phi_{\XR}= \m_{\XR}(e_{\XR})\leq \m_{\XR}(\partial_+\cR)
= \frac{\bar\mu(m_0+\sfrac{2}{N})}{\mu(\XR)}\leq
\exp{(-\b N \G')}(1 +o(1)) \,,
\end{split}
\end{equation}
where in the last inequality we set $\G':= f(m_0+\sfrac{2}{N})-f(m_+)$ and used (\ref{def:freeE}).

To get a lower bound over $\g_{\XR}$, as for $\g_\cR$ we proceed by
first estimating the mixing time of the dynamics reflected in $\XR$,
$X_{\XR}=(X_{\XR}(t))_{t\geq 0}$.
With obvious notation, it holds the following:
\begin{proposition}\label{prop:tmixingc}
\begin{equation}\label{eq:tmix1}
\t_{mix,\XR}(\sfrac{1}{4})\leq c(\b) N^{\frac{3}{2}}\,.
\end{equation}
\end{proposition}
\begin{proof}
The proof is the same as for Prop. \ref{prop:tmixing},
and can write down just replacing $\cR$ with $\XR$, the states
$-1, m_-$ and $m_0$ respectively with $m+$, $-1$ and $m_0+\sfrac{2}{N}$
and the time $T$ defined in (\ref{def:T}) with
$$T'= \max_{m\in\overline\XR}\, \E_m(\t_{m_+})=
\max\left\{\E_{+1}(\t_{m_+});\E_{m_0+\frac{2}{N}}(\t_{m_+})\right\}\,.
$$
\end{proof}
As a consequence of (\ref{stimaphic}) and Prop. \ref{prop:tmixingc},
we obtain
\begin{equation}\label{rosso}
\e_{\XR}^*=\frac{\phi_{\XR}^*}{\g_{\XR}}\leq c(\b)N^{\frac{3}{2}}\exp\left(-\b N\Gamma' \right)(1+o(1))\,,
\end{equation}
and also
\begin{equation}\label{blu}
\frac{\phi_{\cR}^*}{\g_{\XR}}\leq c(\b)N^{\frac{3}{2}}\exp\left(-\b N\Gamma \right)(1+o(1))\,,
\end{equation}
which are both $\ll 1$ for any $N$ large enough.

\subsubsection{Relaxation, transition and mixing times.}
From inequalities (\ref{rosso}) and (\ref{blu}),
we can choose $k,\l$ in Theorems \ref{th:gap&cap}-\ref{mix},
such that $\phi_{\cR}\ll k\ll \g_{\cR}$ and
$\phi_{\cR} + \phi_{\XR} \ll \l\ll \g_{\XR}$,
and then get matching upper and lower bound over on the relaxation time
$\g$ and the mean transition time.
Explicitly,  by Th. \ref{th:gap&cap}, \ref{th:t7} and \ref{th:t8},
it holds that in the limit  $N\rightarrow\infty$ and for $k,\l$ such that $\phi_{\cR}^*\ll k\ll \g_{\cR}$
and $\max\{\phi_{\cR}^*, \phi_{\XR}^*\}\ll \l\ll \g_{\XR}$,
\begin{itemize}
\item[i)]
\begin{equation}
\g^{-1}= \frac{\m(\cR)\m(\XR)}{C_k^\l(\cR,\XR)}(1+o(1))
\end{equation}
\item[ii)] For all distribution $\nu$ over $\cR$,
\begin{equation}\label{CW:transtime}
	\left\{
		\begin{array}{l}
			\E_\nu (\t_{\XR,\l})\leq {\phi_{\cR,\l}^*}^{-1}(1+o(1))\\
			\E_\nu (\t_{\XR,\l})\geq (1-\pi_\cR(\nu)){\phi_{\cR,\l}^*}^{-1}(1+o(1))
		\end{array}
	\right.
\end{equation}
and for all $t>0$,
\begin{equation}\P_\nu (\phi_{\cR,\l}^*\t_{\XR,\l} >t )= (1-\pi_\cR(\nu))e^{-t}(1+o(1))\,.
\end{equation}
In particular, for $\nu=\m_\cR$,
\begin{equation}\label{CW:exittime3}
\E_{\m_\cR} (\t_{\XR,\l})= {\phi^*_{\cR,\l}}^{-1}(1+o(1))\end{equation}
\begin{equation}\P_{\m_\cR} (\phi_{\cR,\l}^*\t_{\XR,\l}>t )= e^{-t}(1+o(1))\,,\quad \forall t\geq 0\,.
\end{equation}
\item[iii)]
\begin{equation}\label{phicap2}
\phi_{\cR,\l}^*= \frac{C_k^\l(\cR,\XR)}{\m(\cR)}(1+o(1))\end{equation}
\end{itemize}

To provide quantitative estimates on the relaxation and transition time,
it thus remains to estimate the capacity $C_k^\l(\cR,\XR)$.
As for $C_k(\cR,\XR)$,  we make use of the variational characterizations
(\ref{capacity}) and (\ref{cuore}), with suitable test functions.
The functions that we consider are extensions of
those defined for $C_k(\cR,\XR)$, in the sense that they are defined
similarly but on a bigger support.

Explicitly,   let $\widetilde V(\s):= V_{m_-,m_+}(m_N(\s))$,
with $V_{m_-,m_+}$ defined in (\ref{due}).
Plugging $\widetilde V$ into (\ref{capacity}), we  obtain the upper bound
\begin{equation}\label{stimacapkl}
C_{k}^\l(\cR,\XR)\leq \cD(\widetilde V)+ k \sum_{\s\in\cR}\m(\s)(\widetilde V(\s)-1)^2
+ \l\sum_{\s\in\XR}\m(\s)(\widetilde V(\s))^2\,
\end{equation}
Since $\widetilde V$ is defined as the equilibrium potential
of the one-dimensional chain, with boundary condition $\widetilde V (m_-)=1$ and $\widetilde V(m_+)=0$,
we have that $\cD(\widetilde V)=C(m_-,m_+)$.
Using inequality (\ref{quattro}) and choosing
$k,\l\ll \g_{\XR}\leq c(\b) N^{-\sfrac{3}{2}} $, the second and third terms
of (\ref{stimacapkl}) are bounded as
\begin{equation}\nonumber
\begin{split}
k \sum_{\s\in\cR}\m(\s)(\widetilde V(\s)-1)^2&=
k \sum_{m=m_-}^{m_0}\frac{e^{-\b N f_N(m)}}{Z_N}
\left(\frac{C(m_-,m_+)}{C(m_-,m-\sfrac{2}{N})}\right)^2\\
&\leq kc(\b) N^{\sfrac{3}{2}} \left(C(m_-,m_+)\right)^2 Z_N\, e^{\b N f_N(m_0)}\\
&\leq c(\b)\left(C(m_-,m_+)\right)^2 Z_N\, e^{\b N f_N(m_0)}\\
\end{split}
\end{equation}
\begin{equation}\nonumber
\begin{split}
\l\sum_{\s\in\XR}\m(\s)(\widetilde V(\s))^2 &=
\l\sum_{m=m_0}^{m^+}\frac{e^{-\b N f_N(m)}}{Z_N}
\left(\frac{C(m_-,m_+)}{C(m,m_+)}\right)^2\\
&\leq \l c(\b) N^{\sfrac{3}{2}} \left(C(m_-,m_+)\right)^2 Z_N\, e^{\b N f_N(m_0)}\\
&\leq c(\b)\left(C(m_-,m_+)\right)^2 Z_N\, e^{\b N f_N(m_0)}
\end{split}
\end{equation}

In Appendix B, the  capacity $C(m_-,m_+)$ is evaluated for large $N$
and  the following formula is obtained
\begin{equation}\label{capacitytot}
C(m_-,m_+)=
\displaystyle\frac{\sqrt{(1-m_0^2)\b|f''(m_0)|}}{2\pi N}\frac{e^{-\b N f(m_0)}}{Z_N}(1+o(1))\,.
\end{equation}
This implies that the second and third terms above are $o(C(m_-,m_+))$ and then
\begin{equation}
C_{k}^\l(\cR,\XR)\leq C(m_-,m_+)(1+o(1))\,.
\end{equation}

For the lower bound we consider a test unitary flow $\widetilde\psi(\s,\s'):=
\widetilde\Psi(m_N(\s), m_N(\s'))$  with
\begin{equation}\label{flow2}
\left\{
\ba{ll}
\Psi(m, m+\sfrac{2}{N})= \left(S(m)\frac{(1-m)N}{2}\right)^{-1}& \forall m\in[m_-,m_+]_N \,,\\
\Psi(m,m')=0 & \mbox{otherwise}
\ea
\right.\,.
\end{equation}
Inserting the test flow in (\ref{cuore}), we then  have
\begin{equation}\label{stimacapklow2}
\begin{split}
C_{k}^\l(\cR,\XR)^{-1}&\leq \cD(\widetilde\psi) +\frac{\m(\cR)}{k}S(m_-)
\frac{e^{-\b N u(m_-)}}{\m(\cR)Z_N} \left(\frac{Z_N\cdot e^{-\b N u(m_-)}}{S(m_-)} \right)^2\\
&+\frac{\m(\XR)}{\l}S(m_+)
\frac{e^{-\b N u(m_+)}}{\m(\XR)Z_N} \left(\frac{Z_N\cdot e^{-\b N u(m_+)}}{S(m_+)} \right)^2\\
&= \sum_{m=m_-}^{m_0} \frac{Z_N\cdot e^{\b N f_N(m)}}{\bar p(m,m+\sfrac{2}{N})}
+ \frac{1}{k}Z_N\cdot e^{\b N f_N (m_-)}+\frac{1}{\l}Z_N\cdot e^{\b N f_N (m_+)}\\
&\leq C(m_-,m_+)^{-1}(1+o(1)) \,,
\end{split}
\end{equation}
where in the last step we used that $k^{-1}\ll {\phi_\cR^*}^{-1}=\mu(m_-) C(m_-,m_0)^{-1}$,
$\l^{-1}\ll {\phi_{\XR}^*}^{-1}=\mu(m_+) C(m_0,m_+)^{-1}$, and the fact that
$C(m_-,m_0)\,,C(m_0,m_+)$ and $C(m_-,m_+)$ are all of order $N^{-1}e^{-\b N f(m_0)}$ (see Appendix B).
From (\ref{phicap2}) and with the above estimates, we then have
\begin{equation}
\phi_{\cR,\l}^*= \frac{C(m_-,m_+)}{\mu(\cR)} (1+o(1))\,.
\end{equation}
Finally, if we choose $\l$ so that $\l T_{\XR}^*\ll 1$, for example $\l\ll N^{-5/2}$,
then we can apply Th. \ref{mix} and get an upper bound on the mixing time of the same order
of the transition and relaxation times.
Altogether, in the limit  $N\rightarrow\infty$ and for $k,\l$ such that $e^{-\b N\Gamma}\ll k\ll N^{-3/2}$
and $e^{-\b N\Gamma}\}\ll \l\ll N^{-5/2}$, it holds
\begin{itemize}
\item[i)]
\begin{equation}\label{viola}
\begin{split}
\g^{-1}&=\E_{\m_\cR}(\t_{\XR,\l})(1+o(1))\\
&=\frac{2\pi N\cdot e^{\b N
(f(m_0)-f(m_-))}}{\b\sqrt{|f''(m_0)|f''(m_-)(1-m_0^2)(1-m_-^2)} }(1+o(1))
\end{split}
\end{equation}
\item[(ii)]
\begin{equation}
\t_{mix} (\sfrac 14) \leq 4 \g^{-1} (1+o(1))
\end{equation}
\end{itemize}

\begin{remark}
Notice that in the Curie-Weiss model
the mean exit time and the mean transition time
differ asymptotically only by a factor $2$  (see Eqs. (\ref{lola})and (\ref{viola})).
This is a slight difference
but one  that clarifies the different rule of the exit time from
the transition time. Notice also that by the well-known bound
$\t_{mix} (\sfrac 14)\geq \g^{-1}$,
the second result shows that the mixing time and the relaxation time
are of the same order, which is $N\cdot e^{\b N(f(m_0)-f(m_-))}$.
\end{remark}

\subsection{The wasp graph} \label{naso}

Given three positive real numbers
$r_a$, $r_t$ and $r_w$
and a positive integer $n$,
we set $l_a = \lfloor n r_a \rfloor$,
$l_t = \lfloor n r_t \rfloor$
and $l_w = \lfloor n r_w \rfloor$.
We then consider two cubic lattices
with vertices indexed by $\{0, \dots, l_a\}^3$
and $\{0, \dots, l_t\}^3$, four copies
of the square lattice with vertices indexed
by $\{0, \dots, l_w\}^2$ and we attached them together
by identifying some corners as in Figure~\ref{ciccio},
forming then the ``wasp graph'' with its four ``wings''
and its ``abdomen'' attached to its central ``thorax''.
We finally place ourself in the regime $n \gg 1$
and consider the random walk
with constant fixed rate $\alpha$
between nearest-neighbour, with $\alpha \leq 1/6$
to satisfy our hypothesis~(\ref{lampione}).

\begin{figure}[htbp]
	\centering
	\scalebox{.5}{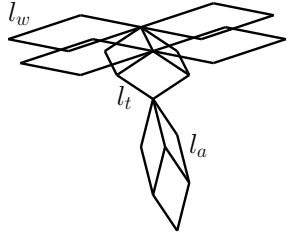}
	\caption{
		\label{ciccio}
		A wasp without a head, and maybe misplaced wings.
	}
\end{figure}

Without wings and with $r_a = r_t = 1$
our wasp would be the three-dimensional ``n-dog'' of~\cite{SC}
and we would have a relaxation time and mixing time of order $n^3$.
We will reprove this result by using our $(\kappa, \lambda)$-capacities,
actually considering the same kind of flows
as those used in~\cite{SC}
but, as we will see, with some more flexibility
in building such flows.
We will also show that,
as one could expect,
adding the wings will not change the spectral gap
and mixing time asymptotics.
This is, in particular, to illustrate how one can recursively
apply Theorem~(\ref{th:gap&cap}): $\gamma_{\mathcal{R}}$
can be estimated by applying the theorem to the restricted dynamics itself.
The last reason why we introduced this toy model
that combines two and three-dimensional graphs
is that it illustrates some limits of our result:
while using the three-dimensional parts of our graph
we will be able to estimate easily the mixing time
of our random walk and prove asymptotic exponential laws
for exit and transition times,
the two-dimensional ``wing pair restricted'' random walk
is associated with a too slowly decreasing $\varepsilon^*_{\mathcal{R}}$
to control more than the relaxation time:
we are in the regime $\varepsilon^*_{\mathcal{R}} \ll 1$
but {\em outside} the regime $\phi^*_{\mathcal{R}} T^*_{\mathcal{R}} \ll 1$.

To fix some notation, let us call $\mathcal{R}_t$ the cubic lattice
$\{0, \dots, l_t\}^3$ and $\mathcal{R}_a$ this other cubic lattice
from which one corner is removed to have a partition of the vertices
$\mathcal{X}_b = \mathcal{R}_t \cup \mathcal{R}_a$
that describe the wasp body obtained after remotion of the wings.
In the same way we call $\mathcal{R}_1$, $\mathcal{R}_2$, $\mathcal{R}_3$
and $\mathcal{R}_4$ the four square lattices from which one point has been removed
to obtain a partition of $\mathcal{R} = \mathcal{R}_t \cup \bigcup_{i = 1}^4 \mathcal{R}_i$
that is the front part of the wasp.
We then have a partition of the whole graph $\mathcal{X} = \mathcal{R} \cup \mathcal{R}_a$

Let us start with the study of the $\mathcal{X}_b$-restricted random walk.
We will write $\phi^*_t$, $\gamma_t$ and $\varepsilon^*_t$ instead of
$\phi^*_{\mathcal{R}_t}$, $\gamma_{\mathcal{R}_t}$ and $\varepsilon^*_{\mathcal{R}_t}$.
From Lemma~\ref{lem:phiphi} we have, with obvious notation,
$\phi^*_t \leq \phi_t = 3\alpha / (1 + l_t)^3$.
As far as $\gamma_t$ is concerned we can use the following lemma.

\begin{lemma}
	For $d \geq 1$, if $\gamma_d$ is the spectral gap
	of the random walk on the $d$-dimensional lattice $\{0, \dots, l\}^d$
	with nearest-neighbour jump rate $\alpha$, then
	\begin{equation} \label{green}
		\frac{1}{\gamma_d} \leq \frac{dl(l + 1)}{2\alpha / e}
		\,.
	\end{equation}
	In addition, if we call 0 the all 0 coordinate vertex
	and $\gamma'_d$ the spectral gap of restricted random walk
	on $\{0, \dots, l\}^d \setminus \{0\}$, then
	\begin{equation} \label{white}
		\frac{1}{\gamma'_d} \leq \frac{d (l + 1)^2}{2 \alpha^{d + 1} / e}
		\,.
	\end{equation}
\end{lemma}

\begin{proof}
	The estimate~(\ref{green}) is obtained by the standard
	coordinate by coordinate coupling for the lazy version
	of the original random walk.
	The same coupling can be used for the random walk
	on the graph with one removed corner to bring each coordinate
	of two lazy random walks at distance 1 at most.
	Since one can then build another coupling making them meet
	in $d$ steps at most with probability $\alpha^d$ at least,
	and start the coupling again from the beginning if they do not,
	the mean meeting time of these lazy random walks
	is bounded from above by
	\begin{equation}
		\frac{1}{\alpha^d}\left(
			\frac{dl(l + 1)}{2\alpha} + d
		\right)
		\leq \frac{d(l + 1)^2}{2\alpha^{d + 1}}
		\,.
	\end{equation}
	Markov's inequality makes then possible to bound
	the mixing time of the lazy walk,
	from which one deduces~(\ref{white}) for the original walk.
\end{proof}

The first part of the lemma together with the previous estimate
on $\phi^*_t$ gives then
\begin{equation}
	\varepsilon^*_t
	\leq \frac{3 l_t (l_t + 1)}{2 \alpha / e} \frac{3 \alpha}{(l_t + 1)^3}
	\leq \frac{9 e}{1 + l_t}
	\,.
\end{equation}

To prove that $l_t^3$ is the correct order of the exit time from $\mathcal{R}_t$
we apply Theorem~\ref{th:phi&cap} to estimate $\phi^*_t$ from below
and then just have to build a unitary flow from $\mathcal{R}_a$ to $\mathcal{R}_t$
to estimate from below a $(\kappa, \lambda)$-capacity with $\lambda = + \infty$.
We send a flow of strength 1 to the junction corner (this will be modified
when working with $\lambda < +\infty$) and have to absorb it in $\mathcal{R}_t$.
Since we know the probabilistic meaning of the optimal flow and $\kappa$ has,
heuristically, to be small enough to be close to local equilibrium at absorption,
we should absorb a fraction of order $1 / (1 + l_t)^3$ of this unitary flow
in each vertex of $\mathcal{R}_t$. But we are not constrained to realize this
exactly and this is where we have some flexibility that helps in computation.
Since we also know from the electrical network picture that the optimal flow
should in some sense be radially distributed,
we build our flow as the mean of a random simple flow with some spherical symmetry.
Let us first explain how two build a certain random path $\xi$.
We begin by choosing a point $Q$ with positive coordinates
in the origin centered euclidean ball of radius $(1 + l_t)$
according to the normalized Lebesgue measure.
This point belongs to some unitary cube
with integer coordinates corners
and we call $Q'$ the corner with the smallest coordinates.
We then approximate the radius $[0,Q]$
by a coordinate non-decreasing path
that starts from $0$, is only made of edges
along the unitary cubes crossed by $[0, Q]$,
and ends in $Q'$.
Such a path is in particular a shortest path
on the lattice that links $0$ with $Q'$
and the fact that is exists can be shown
by shown by recurrence on the dimension
and by using projections along coordinate axes.
For such a random path $\xi$ (since $Q$ is random) we define a flow $\psi_\xi$
by $\psi_\xi(x, y) = \mathbbm{1}_{\{(x, y) \in \xi\}} - \mathbbm{1}_{\{(y, x) \in \xi\}}$.
We finally use as test flow in Thomson's principle
the associated mean flow, that is the flow $\psi$ such that,
for $x$ and $y$ nearest neighbours with $\|y\|_2 > \|x\|_2$,
$\psi(x, y) = \mathbbm{P}((x, y) \in \xi)$.
Since the distance between these approximating paths
and their associated radius is smaller than $\sqrt{3}$,
the chosen point has to be in cone of half angle $\alpha$,
with $\sin \alpha = \sqrt{3} / \|y\|_2$ for $y$ to be used in the approximating path.
It follows that
\begin{equation}
	\psi(x, y)
	\leq \frac{1}{\frac{1}{8} \frac{4 \pi (1 + l_t)^3}{3} }\frac{2 \pi}{3} (1 + l_t)^3 (1 - \cos \alpha)
	= 4 \left(
		1 - \sqrt{1 - \frac{3}{\|y\|_2^2}}
	\right)
	\leq \frac{12}{\|y\|_\infty^2}
	\,.
\end{equation}
Also, for all $x \in \mathcal{R}_t$, we have ${\rm div}_x \Psi = \mathbbm{P}(Q' = x)$
and
\begin{equation}
	\begin{split}
		\mathbbm{E}_{\mu_{R_t}}\left[
			\left(
				\frac{{\rm div} \Psi}{\mu_{\mathcal{R}_t}}
			\right)^2
		\right]
		& = \sum_{x \in \mathcal{R}_t} \frac{1}{(1 + l_t)^3} (1 + l_t)^6
		\mathbbm{P}^2\left(
			Q' = x
		\right) \\
		& = \sum_{x \in \mathcal{R}_t}  (1 + l_t)^3 \frac{
			{\rm Vol}(C(x) \cap B_{1 / 8})^2
		}{
			\left(
				\frac{1}{8} \frac{4}{3} \pi (1 + l_t)^3
			\right)^2
		}
	\end{split}
\end{equation}
where $C(x)$ is the unitary cube with $x$ as smallest coordinate coordinate corner,
$B_{1 / 8}$ is the positive coordinate part of the ball of radius $(1 + l_t)$
and ${\rm Vol}$ stands for the Lebesgue measure.
It follows that
\begin{equation}
	\mathbbm{E}_{\mu_{R_t}}\left[
		\left(
			\frac{{\rm div} \Psi}{\mu_{\mathcal{R}_t}}
		\right)^2
	\right]
	\leq \sum_{x \in \mathcal{R}_t}  \frac{1}{(1 + l_t)^3} \frac{
		{\rm Vol}(C(x) \cap B_{1 / 8})
	}{
		\left(
			\frac{1}{8} \frac{4}{3} \pi
		\right)^2
	}
	= \frac{6}{\pi}
\end{equation}

Thomson's principle gives then, with $\mu_b$ the uniform measure on $\mathcal{X}_b$, $\kappa > 0$
and $C_{\kappa, b}(R_t, R_a)$ the $\kappa$-capacity between $R_t$ and $R_a$ that is computed relatively
to the restricted random walk in $\mathcal{X}_b$,
\begin{equation}
	\begin{split}
		\frac{
			\mu_b\left(
				\mathcal{R}_t
			\right)
		}{
			C_{\kappa, b}\left(
				\mathcal{R}_t, \mathcal{R}_a
			\right)
		}
		& \leq \frac{(1 + l_t)^3}{\alpha} + \sum_{k = 0}^{l_t - 1} 3 \frac{(1 + l_t)^3}{\alpha} \left[
			3(1 + k)^2 - 3 (1 + k) + 1
		\right] \frac{144}{(1 + k)^4}
		+ \frac{1}{\kappa} \frac{1}{8} \frac{6}{\pi} \\
		& \leq \frac{(1 + l_t)^3}{\alpha} \left(
			1 + 432 \sum_{k \geq 1} \frac{3}{k^2}
		\right)
		+ \frac{6}{\kappa\pi} \\
		& \leq 2161 \frac{(1 + l_t)^3}{\alpha}
		+ \frac{6}{\kappa\pi}
		\,.
	\end{split}
\end{equation}
Choosing $1 / \kappa \ll n^3$ this shows that $l_t^3 = \lfloor r_t n \rfloor^3$ is the correct order for the exit time.

To estimate transition and relaxation time with the same tools,
we have to estimate $(\kappa, \lambda)$-capacities with finite $\lambda$.
Using not only randomly chosen sinks but randomly chosen sources also,
we can define a mean flow as previously to prove, with obvious notation,
\begin{equation}
	\begin{split}
		\frac{
			\mu_b\left(
				\mathcal{R}_t
			\right) \mu_b\left(
				\mathcal{R}_a
			\right)
		}{
			C_{\kappa, b}^\lambda\left(
				\mathcal{R}_t, \mathcal{R}_a
			\right)
		}
		& \leq \mu_b\left(
				\mathcal{R}_a
		\right) \sum_{k = 0}^{l_t - 1} 3 \frac{(1 + l_t)^3}{\alpha} \left[
			3(1 + k)^2 - 3 (1 + k) + 1
		\right] \frac{144}{(1 + k)^4} \\
		&\qquad + \mu_b\left(
				\mathcal{R}_t
		\right) \sum_{k = 0}^{l_a - 1} 3 \frac{(1 + l_a)^3}{\alpha} \left[
			3(1 + k)^2 - 3 (1 + k) + 1
		\right] \frac{144}{(1 + k)^4} \\
		&\qquad\quad + \mu_b\left(
				\mathcal{R}_a
		\right)\frac{6}{\kappa\pi}
		+ \mu_b\left(
				\mathcal{R}_t
		\right)\frac{6}{\lambda\pi}  \\
		& \leq 2160\left(
			 \mu_b\left(
				\mathcal{R}_a
			\right)\frac{(1 + l_t)^3}{\alpha}
			+ \mu_b\left(
				\mathcal{R}_t
			\right)\frac{(1 + l_a)^3}{\alpha}
		\right)
		+ \frac{6}{\kappa\pi}
		+ \frac{6}{\lambda\pi}
		\,,
	\end{split}
\end{equation}
to get, by choosing also $1 / \lambda \ll n^3$,
\begin{equation}
	\frac{
		\mu_b\left(
			\mathcal{R}_t
		\right) \mu_b\left(
			\mathcal{R}_a
		\right)
	}{
		C_{\kappa, b}^\lambda\left(
			\mathcal{R}_t, \mathcal{R}_a
		\right)
	}
	\leq 2160 \frac{2r_t^3 r_a^3}{r_t^3 + r_a^3} \frac{n^3}{\alpha}
	+ o(n^3)
	\,.
\end{equation}
By Theorem~\ref{th:gap&cap}, using~(\ref{white}) and choosing
$\kappa, \lambda \ll 1 / n^2$, this gives an upper bound
on the relaxation time with the same asymptotics.
Theorem~\ref{th:t8} also provides a similar upper bound on the mean transition time.
Going to lower bounds on $1/ \gamma_b = 1 / \gamma_{\mathcal{X}_b}$ and $1 / \phi^*_{t, \lambda}$
one could estimate $(\kappa, \lambda)$-capacities through Dirichlet principle,
but it is better to recall that
$\mu_b(\mathcal{R}_t)/ C_{\kappa, b}^\lambda(\mathcal{R}_t, \mathcal{R}_a) \geq (1 - \epsilon) / \phi^*_{t, \lambda}$
for any $\epsilon$ and large enough $n$ and $1 / \phi^*_{t, \lambda} \geq 1 /\phi^*_t \geq (1 + l_t)^3/(3\alpha)$.

As far as exponential asymptotic laws and mixing time asymptotics are concerned,
our results depend on our ability, with obvious notation, to control $\zeta^*_t$
and show that $\varepsilon^*_t \ln(1 / \zeta^*_t)$
goes to zero and ensure $\phi_t^* T^*_t \ll 1$.
This cannot be achieved by using Lemma~\ref{lem:zeta},
since $\varepsilon_t^*/\gamma_t = \phi_t^*/ \gamma_t^2 \gg 1$
and $\varepsilon_t^* D_t$ is of order 1. (Estimates provided by~(\ref{tartaruga})
and~(\ref{zidane}) would actually be enough in dimension four and five respectively.)
We are, however, in the special case where~(\ref{brown}) holds
and proves, since $\phi^*_t$ and $\phi_t$ are of the same order,
that $\varepsilon^*_t \ln(1 / \zeta^*_t) \ll 1$.
This proves local thermalization on time scale $n^2 \ln n$
and exponential asymptotic laws immediately follow.
This also proves that the mixing time goes like $n^3$
as soon as $r_t \neq r_a$.

We prove now that these asymptotics
on relaxation, transition, exit and mixing times
are still valid on the full wasp graph, wings included.
To do so we note that our previous flow used to estimate
$(\kappa, \lambda)$-capacities between $\mathcal{R}_t$
and $\mathcal{R}_a$ in $\mathcal{X}_b$
can still be used to estimate
$(\kappa, \lambda)$-capacities between $\mathcal{R}$ (wings included)
and $\mathcal{R}_a$ in the full space $\mathcal{X}$.
The key point now is to control $\gamma_{\mathcal{R}}$.
If our wasp had only one wing $\mathcal{R}_1$,
we could have use Theorem~\ref{th:gap&cap} directly
on $\mathcal{R} = \mathcal{R}_1 \cup (\mathcal{R} \setminus \mathcal{R}_1)$.
We will use instead Lemma~\ref{giallo}
and, anyway, will have to estimate $(\kappa, \lambda)$-capacities
between $\mathcal{R}_1$ and $\mathcal{R}_t$ and compare it with
$\gamma_1 = \gamma_{\mathcal{R}_1}$ and $\gamma_t = \gamma_{\mathcal{R}_t}$.

Let us start by estimating $\phi^*_1 = \phi^*_{\mathcal{R}_1}$.
In this two-dimensional case the easy bound $\phi^*_1 \leq \phi_1$
is not a good one. We then use the variational principle satisfied by $\phi^*_1$
(see Lemma \ref{lem:phiphi}) with the same kind of test function
we would have use to estimate $C_\kappa(\mathcal{R}_1, \mathcal{R}_t)$.
With $V(x) = (\ln(\|x\|_\infty))/(1 + \ln l_1)$ for $x \in \mathcal{R}_1$,
we have, with obvious notation,
\begin{equation}
	\begin{split}
		\mathcal{D}_1(V)
		& = \sum_{k = 0}^{l_1 - 1} 2 (k + 1) \frac{\alpha}{(1 + l_1)^2} \left[
			\frac{
				\ln(k + 2) - \ln(k + 1)
			}{
				1 + \ln l_1
			}
		\right]^2 \\
		& \leq \frac{2 \alpha}{
			\left(
				1 + \ln l_1
			\right)^2 \left(
				1 + l_1
			\right)^2
		} \sum_{k = 0}^{l_1 - 1} \frac{1}{k + 1} \\
		& \leq \frac{2 \alpha}{
			\left(
				1 + \ln l_1
			\right)^2 \left(
				1 + l_1
			\right)^2
		} \left(
			1 + \ln l_1
		\right)
		= \frac{2 \alpha}{
			\left(
				1 + \ln l_1
			\right) \left(
				1 + l_1
			\right)^2
		}
	\end{split}
\end{equation}
We also have
\begin{equation}
	\begin{split}
		\mu_1(V^2)
		& = \sum_{k = 0}^{l_1} (2k + 1) \frac{1}{(1 + l_1)^2}
		\frac{
			\ln^2 (1 + k)
		}{
			(1 + \ln l_1)^2
		}\\
		& \geq \frac{1}{(1 + l_1)^2 (1 + \ln l_1)^2} \int_1^{1 + l_1} (2x -1) \ln^2 x \,dx \\
		& \geq \frac{1}{3} \qquad \mbox{for $l_1 \geq 20$.}
	\end{split}
\end{equation}
Since $V_{|\partial_- \mathcal{R}_t} \equiv 0$ it follows that $\phi_1^* \leq 6\alpha/((1 + l_1)^2 (1 +\ln l_1))$ for $l_1 \geq 20$,
and $\varepsilon_1^*$ decreases at least like $1 / \ln l_1$.

To see that we have found the right order for $\phi_1^*$
we estimate the $\kappa$-capacity between $\mathcal{R}_1$ and $\mathcal{R}_t$
by using the same kind of flow as previously.
For nearest neighbours $x$ and $y$ with $\|y\|_2 > \|x\|_2$
such a flow $\psi$ satisfies, with $\sin \alpha \leq \sqrt{2}/\|y\|_2$ and $\alpha \leq \pi / 4$,
so that $\sin \alpha \geq \alpha / \sqrt{2}$
\begin{equation}
	\begin{split}
		\psi(x, y)
		& \leq \frac{1}{\frac{1}{4} \pi l_1^2}
		\alpha l_1^2
		\leq \frac{4}{\pi} \sqrt{2} \frac{\sqrt{2}}{\|y\|_2}
		\leq \frac{8}{\pi \|y\|_\infty}
		\,.
	\end{split}
\end{equation}
Thomson principle then gives
\begin{equation}
	\begin{split}
		\frac{\mu_{\mathcal R}(\mathcal{R}_1)}{C_\kappa(\mathcal{R}_1, \mathcal{R}_t)}
		& \leq \frac{(1 + l_1)^2}{\alpha}
		+ \sum_{k = 0}^{l_1 - 1} 2 (2k +1) \frac{(1 + l_1)^2}{\alpha} \left(\frac{8}{\pi (k + 1)}\right)^2
		+ \frac{1}{\kappa}\frac{1}{\frac{1}{4} \pi} \\
		& \leq \frac{(1 + l_1)^2}{\alpha} \left(
			1 + \frac{256}{\pi^2} \sum_{k = 0}^{l_1 - 1} \frac{1}{k + 1}
		\right)
		+ \frac{4}{\kappa \pi}\\
		& \leq \frac{(1 + l_1)^2}{\alpha} \left(
			1 + 26 (1 + \ln l_1)
		\right)
		+ \frac{4}{\kappa \pi}
		\,,
	\end{split}
\end{equation}
which proves, choosing $1/\kappa \ll n^2 \ln n$ that $1 / \phi^*_1$ is of order $n^2 \ln n$.

Combining these two- and three-dimensional flows we get,
denoting by $C_{\kappa, \mathcal{R}}^\lambda(\cdot, \cdot)$
the $(\kappa, \lambda)$-capacity
that is computed relatively to the random walk restricted in $\mathcal{R}$.
\begin{equation}
	\begin{split}
		\frac{
			\mu_{\mathcal{R}}(\mathcal{R}_1) \mu_{\mathcal{R}}(\mathcal{R}_t)
		}{
			C_{\kappa, \mathcal{R}}^\lambda\left(
				\mathcal{R}_1, \mathcal{R}_t
			\right)
		}
		& \leq 26 \mu_{\mathcal{R}}(\mathcal{R}_t) \frac{(1 + l_1)^2(1 + \ln l_1)}{\alpha}
		+ 2160 \mu_{\mathcal{R}}(\mathcal{R}_1) \frac{(1 + l_t)^3}{\alpha}\\
		& \qquad + \mu_{\mathcal{R}}(\mathcal{R}_t)\frac{4}{\kappa \pi}
		+ \mu_{\mathcal{R}}(\mathcal{R}_1)\frac{\pi}{6 \lambda}
		\,.
	\end{split}
\end{equation}
With $1 / \kappa \ll n^2 \ln n$ as previously and $1 / \lambda \ll n^3$,
since $\mu_{\mathcal{R}}(R_t)$ is of order 1
and $\mu_{\mathcal{R}}(R_1)$ is of order $1 / n$,
this leads to
\begin{equation}
	\begin{split}
		\frac{
			\mu_{\mathcal{R}}(\mathcal{R}_1) \mu_{\mathcal{R}}(\mathcal{R}_t)
		}{
			C_{\kappa, \mathcal{R}}^\lambda\left(
				\mathcal{R}_1, \mathcal{R}_t
			\right)
		}
		& \leq 26 c_1^2 \frac{n^2 \ln n}{\alpha} + O(n^2)
	\end{split}
\end{equation}
and, choosing also $1 / \kappa, 1 / \lambda \gg n^2$,
one has in the same way, using the previous test function $V$ in Dirichlet principle,
\begin{equation}
	\begin{split}
		\frac{
			\mu_{\mathcal{R}}(\mathcal{R}_1) \mu_{\mathcal{R}}(\mathcal{R}_t)
		}{
			C_{\kappa, \mathcal{R}}^\lambda\left(
				\mathcal{R}_1, \mathcal{R}_t
			\right)
		}
		& \geq \frac{1}{2} c_1^2 \frac{n^2 \ln n}{\alpha} + O(n^2)
		\,.
	\end{split}
\end{equation}
From Lemma~\ref{giallo}, (\ref{green}) and~(\ref{white})
it follows that $1 / \gamma_{\mathcal{R}} = o(n^3)$ and
the results obtained for the wasp without wings
holds with the wings also.

We note however that when applying the previous two-dimensional computation
to the random walk restricted to a pair of wings,  we obtain similarly a good spectral gap control
but we are not able to show the asymptotic exponential law or derive  mixing time estimates,
since, in this case $T^*_{\mathcal{R}_1}$ and $1 / \phi^*_1$ are of the same order.

\begin{appendix}
\section{Estimating $\zeta_\cR^*$}\label{crude}
\subsection{Crude and very crude estimates}
We prove here lemma~\ref{lem:zeta}.

\subsubsection*{Proof of i)}
One has, for all $x$ in $\cR$ and $t > 0$,
\begin{equation}
	\mu_\cR^*(x) = \P_{\mu_\cR^*}\left(
		X(t) = x \big|\, \tau_{\cX\setminus\cR} > t
	\right)
	\geq \P_{\mu_\cR^*}\left(
		X(t) = x , \tau_{\cX\setminus\cR} > t
	\right)
	\,.
\end{equation}
By the natural coupling between $X$ and $X_\cR$
up to time $\tau_{\XR}$ and stochastic
domination of $\tau_{\XR}$ by an exponential
random variable with parameter $\alpha_\cR$
that is independent from $X_R$,
it follows
\begin{equation}
	\mu_\cR^*(x) \geq \P_{\mu_\cR^*}\left(
		X_\cR(t) = x
	\right) e^{-\alpha_\cR t}
	\,.
\end{equation}
By Cauchy-Schwarz inequality,
Proposition \ref{prop:varH},
and the standard trick to control
$\ell_\infty(\mu_\cR)$ norms with $\ell_2(\mu_\cR)$ norms
(the same we used in the proof of Theorem \ref{th:dinamica})
we get
\begin{equation}\label{pollo}
	\mu_\cR^*(x) \geq \left(
		1 - e^{-\gamma_\cR t}\sqrt{\frac{\varepsilon_\cR^*}{(1 - \varepsilon_\cR^*)\mu_\cR(x)}}\,
	\right)
	e^{-\alpha_\cR t} \mu_\cR(x)
	\,.
\end{equation}
To make this bound useful, we notice that the term inside the bracket is larger than or equal to $1/2$
if
$$t\geq t_0:=\frac{1}{2\g_\cR}\ln\left(\frac{4\e_\cR^*}{(1-\e_\cR^*)\m_\cR(x)}\right)$$
If $t_0>0$ for all $x\in\cR$, that is if $\frac{4\e_\cR^*}{(1-\e_\cR^*)\z_\cR}>1$, then we can plug in
its value in (\ref{pollo}) and get, by definition of $\z_\cR^*$,
\begin{equation}
\z_\cR^*\geq \min_{x\in\cR} \frac{\m_\cR(x)}{4}
\left(\sqrt{\frac{4\varepsilon_\cR^*}{(1 - \varepsilon_\cR^*)\mu_\cR(x)}}\,\right)
^{-\frac{2\a_\cR}{\g_\cR}}
\geq \frac{\z_{\cR}}{4}
\left(\sqrt{\frac{4\varepsilon_\cR^*}{(1 - \varepsilon_\cR^*)\z_\cR}}\,\right)
^{-\frac{2\a_\cR}{\g_\cR}}\,.
\end{equation}
 On the other hand, if $\frac{4\e_\cR^*}{(1-\e_\cR^*)\z_\cR}\leq1$,
 we can just take the value $t=0$ in (\ref{pollo}) and get
\begin{equation}
\z_\cR^*\geq \min_{x\in\cR}
\left(1-
\sqrt{\frac{\varepsilon_\cR^*}{(1 - \varepsilon_\cR^*)\mu_\cR(x)}} \,
\right)^2 \m_\cR(x)
\geq \frac{\z_\cR}{4}\,.
\end{equation}
Taking the logarithm of $1/\z_\cR^*$ and putting things together,
we obtain the stated inequality.

\subsubsection*{Proof of ii)}
The first inequality is obvious from the definition of $\z_\cR^*$.
Let $\hat X$ denote the discrete version of $X$ like in Section~\ref{quasi_restr_p_2}
and let $N(t)$ be the number of rings up to time $t$.
 Then, for $z\in\cR$, we have
 \begin{equation}
\begin{split}
\m_\cR^*(z)&=
\lim_{t\to\infty} \P_x(X(t)=z\,|\, \t_{\XR}>t)\\
&=\lim_{t\to\infty} \sum_{k\geq0}
\P_x(\hat X(k)=z\,|\,\hat\t_{\XR}>k)\P(N(t)=k)\\
&\geq \lim_{t\to\infty} \sum_{k\geq0}
\P_x(\hat X(k+D_\cR)=z\,|\,\hat\t_{\XR}>k+D_\cR)\P(N(t)=k+D_\cR)\\
&= \lim_{t\to\infty} \sum_{k\geq0}
\sum_{y\in\cR}\P_x(\hat X(k)=y\,|\,\hat\t_{\XR}>k)
\P_y(\hat X(D_\cR)=z\,|\,\hat\t_{\XR}>D_\cR)\\
&\qquad\qquad\qquad\qquad\times\P(N(t)=k+D_\cR)\,,
\end{split}
 \end{equation}
 where we used the notation $\hat\t_{\XR}$ for the hitting time
 of the chain $\hat X$ on $\XR$.
Since for all $y\in\cR$ we have
$$\P_y(\hat X(D_\cR)=z\,|\,\hat\t_{\XR}>D_\cR)\geq
\P_y(\hat X(D_\cR)=z\,,\,\hat\t_{\XR}>D_\cR)\geq
e^{-\D_\cR D_\cR}\,,
$$
we get
\begin{equation}
\begin{split}
\m_\cR^*(z)&
\geq e^{-\D_\cR D_\cR}
\lim_{t\to\infty} \sum_{k\geq0}\sum_{y\in\cR}
\P_x(\hat X(k)=y\,|\,\hat\t_{\XR}>k)\P(N(t)=k+D_\cR)\\
&=e^{-\D_\cR D_\cR}\lim_{t\to\infty}\P(N(t)\geq D_\cR)
= e^{-\D_\cR D_\cR}\,.
\end{split}
\end{equation}

\subsection{Superharmonicity of $h_\cR^*$}
To prove that $h_\cR^*$ is a super-harmonic function,
notice that, for all $x\in\cR$,
\begin{equation}
\begin{split}
(\cL h_\cR^*)(x)&= -h_\cR^*(x) +
\sum_{y\in\cX} p(x,y)h_\cR^*(y)\\
&= -h_\cR^*(x)+ \sum_{y\in\cR} p(x,y)\frac{\m_\cR^*(y)}{\m_\cR(y)}\\
&= -h_\cR^*(x)+ \sum_{y\in\cR} p_\cR^*(x,y)\frac{\m_\cR^*(y)}{\m_\cR(y)}\\
&= -h_\cR^*(x)+ \sum_{y\in\cR} p_\cR^*(y,x)\frac{\m_\cR^*(y)}{\m_\cR(x)}\\
&= -\phi_\cR^*h_\cR^*(x)\leq 0\,,
\end{split}
\end{equation}
where in the last two lines we used the reversibility
of $p_\cR^*$ w.r.t. $\m_\cR$ and that $\m_\cR^*p_\cR^*=(1-\phi_\cR^*)\m_\cR^*$.

\section{Computation of relevant quantities in the Curie-Weiss model}\label{App:CW}
Here we provide some accurate estimates over relevant quantities
in the characterization of the metastable behavior for the Curie-Weiss model.

\subsection{Measure of the metastable set}
Here we prove formula (\ref{m(R)}) which provides the asymptotic expression of $\m(\cR)$.
By definition on $\mu$ and $\cR$ and using (\ref{def:freeE}),
we have
\begin{equation}\label{m(R)1}
\begin{split}
Z_N \cdot\bar\m(\cR)&= \sum_{k=(-1-m_-)\frac{N}{2}}^{(m_0-m_-)\sfrac{N}{2}}
e^{-\b N f_N(m_- +\frac{2k}{N})}\\
&=\sqrt{\frac{2}{N}}\frac{1}{\sqrt{\pi (1-m_-^2)}}\sum_{k=(-1-m_-)\sfrac{N}{2}}^{(m_0-m_-)\sfrac{N}{2}}
e^{-\b N f(m_- +\frac{2k}{N})}(1+o(1))\\
&= \sqrt{\frac{2}{N}}\frac{1}{\sqrt{\pi (1-m_-^2)}} \,e^{-\b N f(m_-)}
\sum_{k=-\lceil N^{\frac{2}{3}}\rceil}^{\lfloor N^{\frac{2}{3}}\rfloor}
e^{-\frac{\b N f''(m_-)}{2}(\frac{2k}{N})^2} (1+o(1))
\end{split}
\end{equation}
where in the last step we use Taylor approximation and observe that
$$\sum_{|k|\geq \lceil N^{\frac{2}{3}}\rceil}
e^{-\frac{\b N f''(m_-)}{2}(\frac{2k}{N})^2}\leq N e^{-c(\b)N^{\frac{1}{3}}}\,.$$
Approximating the sum in (\ref{m(R)1}) with an integral, we finally get
\begin{equation}\label{m(R)2}
\begin{split}
Z_N\cdot\bar\m(\cR)&=
\frac{1}{\sqrt{\pi (1-m_-^2)}}e^{-\b N f(m_-)}
\int_{\R}e^{-\b  f''(m_-)x^2}dx(1+o(1)) \\
&= \frac{1}{\sqrt{\b f''(m_-)(1-m_-^2)}}
e^{-\b N f(m_-)}(1+o(1))\,.
\end{split}
\end{equation}

\subsection{Capacities between points in the macroscopic scale}
Here we provide an asymptotic expression for the capacity between points
in the one-dimensional dynamics with transition rates (\ref{def:ratesmacro}),
that is the dynamics induced on $[-1,1]_N$ by the Curie-Weiss heath-bath dynamics.

As recalled in Section \ref{sec:CW}, for points $x<y\in [-1,1]_N$
it holds
$$
C(x,y)^{-1}= \sum_{k=0}^{(y-x)\sfrac{N}{2}-1}
\left(\bar c\left(x+\sfrac{2k}{N}, x+\sfrac{2(k+1)}{N}\right)\right)^{-1}\,,
$$
with $\bar c(x,y)=\bar\m(x)\bar p(x,y)$.
In the following, we will first provide an asymptotic approximation
for $C(x,y)^{-1}$ when  $m_0\not\in [x,y]$, and then  compute the asymptotic
formulas of  $C(m_-,m_0)^{-1}$, $C(m_0,m_+)^{-1}$ and $C(m_-,m_0)^{-1}$ .

If $m_0\not\in[x,y]$, we can assume w.l.o.g. that $f(x)>f(z)$ for all $z\in (x,y)$.
Bounding below the rates $\bar p$
with a positive constant $c(\b)$ and from (\ref{def:freeE}), we get
\begin{equation}\label{nomax}
\begin{split}
C(x,y)^{-1}&\leq
 c(\b)\sqrt N Z_N \sum_{k=0}^{(y-x)\sfrac{N}{2}-1}e^{\b N f (x+\sfrac{2k}{N})}\\
& \leq   c(\b)\sqrt{N}Z_N e^{\b N f(x)}\sum_{k=0}^{(y-x)\sfrac{N}{2}-1} e^{-\b |f'(\xi_k)||2k}\\
& \leq  c(\b)\sqrt{N}Z_N e^{\b N f(x)}
\end{split}\,,
\end{equation}
where in the second line we used $f (x+\sfrac{2k}{N})-f (x)= -|f(\xi_k)|\sfrac{2k}{N}$ for some $\xi_k\in(x, x+\sfrac{2k}{N})$
and that there exists a constant $c>0$ such that $|f(\xi_k)|>c$ uniformly in $k$.
If instead $f(y)>f(z)$ for all $z\in (x,y)$, then  is enough to switch $x$ and $y$ and run the
argument above, as $C(x,y)=C(y,x)$.  Altogether, this provides inequality (\ref{quattro}).

To  compute $C(m_-,m_0)$, we first notice that, since $m_0$ is a critical point of $f$,
we can write
$$\tanh(\b\D_{\pm}(x+\sfrac{2k}{N}))=\pm \tanh(\b(m_0+h))(1+o(1))=\pm m_0 (1+o(1))$$
and then get the approximation
$\bar p(x\pm\sfrac{2k}{N},x\pm\sfrac{2(k+1)}{N})=\frac{(1-m_0^2)}{4}(1+o(1))$.
Proceeding as for the computation of $\bar\mu(\cR)$,
we have
\begin{equation}\label{simax}
\begin{split}
C(m_-,m_0)^{-1}&= \frac{4}{(1-m_0^2)} Z_N \sum_{k=0}^{(m_0-m_-)\sfrac{N}{2}-1}e^{\b N f_N (m_0-\sfrac{2k}{N})}(1+o(1))\\
%&= 2\sqrt{\frac{2N\pi}{(1-m_0^2)}} Z_N \sum_{k=0}^{(m_0-m_-)\sfrac{N}{2}-1}e^{\b N f (m_0-\sfrac{2k}{N})}(1+o(1))\\
& =  2\sqrt{\frac{2N\pi}{(1-m_0^2)}} Z_N e^{\b N f(m_0)}
\sum_{k=0}^{\lfloor N^{\frac{2}{3}}\rfloor}
e^{-\frac{-\b N |f''(m_0)|}{2}(\frac{2k}{N})^2} (1+o(1))\\
&= 2N\sqrt{\frac{\pi}{(1-m_0^2)}}Z_N e^{\b N f(m_0)}
\int_{0}^{+\infty}e^{-\b |f''(m_0)|x^2}dx\,(1+o(1)) \\
&= \frac{\pi N}{\sqrt{(1-m_0^2)\b |f''(m_0)|}}Z_N
e^{\b N f(m_0)}(1+o(1))\,.
\end{split}
\end{equation}
This provides formula (\ref{capacityasynt}).

Similarly we can compute $C(m_0, m_+)^{-1}$ and $C(m_-, m_+)^{-1}$. In the first case
we let the sum over $k$ of (\ref{simax}) run from $(m_0-m_+)\sfrac{N}{2}$ to $0$,
and then get the same result as for  $C(m_-, m_0)^{-1}$.
When computing $C(m_-, m_+)^{-1}$, we let the sum over $k$  run from
$(m_0-m_+)\sfrac{N}{2}$ to $(m_0-m_-)\sfrac{N}{2}-1$. We then
approximate the sum by an integral over all $\R$ that finally produces
an extra factor of $2$ with respect to (\ref{simax}).
This yields formula (\ref{capacitytot}).
\end{appendix}


\begin{thebibliography}{99}
\bibitem{TT}
  {\sc W. Thomson, P. G. Tait},
  {\sl Treatise on Natural Philosophy\/},
  Oxford (1867).
\bibitem{Y}
	{\sc A. M. Yaglom},
	Certain limit theorems of the theory of branching stochastic processes,
	{\sl Dokl. Akad. Nauk. SSSR\/} {\bf 56}, 795-798 (1947).
\bibitem{DSe}
  {\sc J. N. Darroch, E.  Seneta},
  On quasi-stationary distributions in absorbing discrete-time finite
  Markov chains,
  {\sl J. Appl. Probability\/} {\bf 2}, 88-100 (1965).
\bibitem{LP}
  {\sc O. Penrose,  J. L. Lebowitz},
  Rigorous treatment of metastable states in the van der Waals-Maxwell Theory,
  {\sl J. Stat. Phys.\/} {\bf 3}, 211-241 (1971).
\bibitem{W}
  {\sc A. D. Wentzell},
  Formulas for eigenfunctions and eigenmeasures that are connected with a Markov process,
  {\sl Teor. Verojatnost. i Primenen.\/} {\bf 18}, 3â€“29 (1973).
\bibitem{CGOV}
  {\sc M. Cassandro, A. Galves, E. Olivieri and M. E. Vares},
  Metastable behaviour of stochastic dynamics: a pathwise approach,
  {\sl J. Stat. Phys\/} {\bf 35}, 603-634 (1984).
\bibitem{WF}
  {\sc M. I. Freidlin, A. D. Wentzell},
  {\sl  Random Perturbations of Dynamical Systems\/},
 Springer-Verlag (1984).
\bibitem{NS1}
  {\sc E. J. Neves, R. Schonmann},
  Critical droplets and metastability for a Glauber dynamics at very low temperature,
  {\sl Comm. Math. Phys.\/} {\bf 137}, 209-230 (1991).
\bibitem{AB}
   {\sc D. J. Aldous, M. Brown},
   Inequalities for rare events in time-reversible Markov chains I,
   {\sl Stochastic Inequalities}, ed. by M. Shaked and Y. L. Tong,
   IMS Lecture Notes in Statistics {\bf 22}, 1-16 (1992).
\bibitem{NS2}
  {\sc E. J. Neves, R. Schonmann},
  Behaviour of droplets for a class of Glauber dynamics at very low temperature,
  {\sl Prob. Theory Relat. Fields\/} {\bf 91}, 331-354 (1992).
\bibitem{S1}
  {\sc  E. Scoppola},
  Metastability for Markov chains: a general procedure based on renormalization group ideas,
  in {\sl Probability theory of spatial disorder and phase transition\/}
  Cambridge (1993) Ed. G.R.Grimmett - Kluwer Acad.Publ.(1994) 303-322.
\bibitem{S2}
  {\sc E. Scoppola},
  Renormalization and graph methods for Markov chains,
  in {\sl Advances in Dynamical Systems and Quantum Physics }
  S.Albeverio, R.Figari, E.Orlandi, A.Teta Ed. - World Scientific 1995.
\bibitem{BC}
  {\sc G. Ben Arous, R. Cerf},
  Metastability of the three-dimensional Ising model on a torus at very low temperature,
  {\sl Electron. J. Prob.} {\bf 1} (1996).
\bibitem{Mi2}
  {\sc L. Miclo},
  Sur les probl\`emes de sortie discrets inhomog\`enes,
  {\sl Ann. Appl. Probab.\/} {\bf 6}, 1112-1156 (1996).
\bibitem{DS}
  {\sc P. Dehghanpour, R. Schonmann},
  Metropolis dynamics relaxation via nucleation and growth,
  {\sl Comm. Math. Phys.} {\bf 188}, 89-119 (1997).
\bibitem{N}
  {\sc J. R. Norris},
  {\em Markov Chains}, Cambridge University Press (1997).
\bibitem{SC}
  {\sc L. Saloff-Coste},
  Lectures on finite Markov chains
  {\sl Lecture Notes in Mathematics\/} {\bf 1665}, 301-413 (1997).
\bibitem{MP}
  {\sc P. Mathieu, P. Picco},
  Metastability and convergence to equilibrium for the random field Curie-Weiss model,
  {\sl  J. Statist. Phys.\/} {\bf 91}, 679-732  (1998).
\bibitem{SS}
  {\sc R. Schonmann, S. Shlosman},
  Wulff droplets and the metastable relaxation of kinetic Ising models,
  {\sl Comm. Math. Phys.} {\bf 194}, 389-462 (1998).
\bibitem{BEGK1}
  {\sc A. Bovier, M. Eckhoff, V. Gayrard and M. Klein},
  Metastability and low lying spectra in reversible Markov chains,
  {\sl Commun. Math. Phys.\/} {\bf  228}, 219-255 (2002).
\bibitem{BM}
  {\sc A. Bovier, F. Manzo},
  Metastability in Glauber dynamics in the low temperature limit: Beyond exponential asymptotics
  {\sl J. Statist. Phys.\/} {\bf 107}, 757-779 (2002).
\bibitem{BBG}
  {\sc G. Ben Arous, A. Bovier and  V. Gayrard},
  Glauber  dynamics of the random energy model. 1. Metastable  motion on the extreme states,
  {\sl Commun. Math. Phys.\/} {\bf 235}, 379-425 (2003)
\bibitem{BEGK2}
  {\sc A. Bovier, M. Eckhoff, V. Gayrard and M. Klein},
  Metastability in reversible diffusion processes 1. Sharp estimates for capacities and exit times.
  {\sl J. Eur. Math. Soc. \/} {\bf 6}, 399-424 (2004).
\bibitem{HMS}
  {\sc W. Huisinga, S. Meyn, C. Sch\"utte},
  Phase transitions and metastability in Markovian and molecular systems,
  {\sl Ann. Appl. Prob.\/} {\bf 14}, 419-458 (2004).
\bibitem{JSTV}
  {\sc M. Jerrum, J. B. Son, P. Tetali and E. Vigoda},
  Elementary bounds on Poincar\'e and log-Sobolev constants for decomposable Markov chains,
  {\sl Ann. Appl. Prob.\/} {\bf 14}, 1741-1765 (2004).
\bibitem{BEGK3}
  {\sc A. Bovier, V. Gayrard and M. Klein},
  Metastability in reversible diffusion processes. 2. Precise estimates for small eigenvalues
  {\sl J. Eur. Math. Soc.\/} {\bf 7}, 69-99 (2005).
\bibitem{GOS}
  {\sc A. Gaudilli\`ere, E. Olivieri and E. Scoppola},
  Nucleation pattern at low temperature for local Kawasaki dynamics in two dimensions,
  {\sl Markov Processes Relat. Fields\/} {\bf 11}, 553-628 (2005).
\bibitem{OV}
  {\sc E. Olivieri, M. E. Vares},
  {\sl Large deviations and metastability\/},
  Cambridge University  Press (2005).
\bibitem{BdHN}
  {\sc A. Bovier, F. den Hollander and F. Nardi},
  Sharp asymptotics for Kawasaki dynamics on a finite box with open boundary conditions
  {\sl Probab. Theor. Rel. Fields.\/} {\bf 135}, 265-310 (2006).
\bibitem{ISS}
  {\sc A. Iovanella, B. Scoppola and E. Scoppola},
  Some spin glass ideas applied to the clique problem,
  {\sl J. Stat. Phys.\/}  {\bf 126}, 895-915 (2007).
\bibitem{BBF}
  {\sc O. Bertoncini, J. Barrera M. and R. Fern\'andez},
  Cut-off and exit from metastability: two sides of the same coin,
  {\sl C. R. Math.\/} {\bf 346}, 691-696 (2008).
\bibitem{BF}
  {\sc A. Bovier, A. Faggionato},
  Spectral analysis of Sinai's walk for small eigenvalues
  {\sl Ann. Probab.\/} {\bf 36}, 198-254 (2008).
\bibitem{CNS}
  {\sc E. N. M. Cirillo,  F. R. Nardi and  C. Spitoni},
  Metastability for reversible probabilistic cellular automata with self-interaction,
  {\sl J. Stat. Phys.\/} {\bf 132}, 431-471 (2008).
\bibitem{BBI1}
  {\sc A. Bianchi, A. Bovier and D. Ioffe},
  Sharp asymptotics for metastability in the Random Field Curie-Weiss model,
  {\sl EJP\/} {\bf 14}, 1541-1603 (2009).
\bibitem{G}
  {\sc  A. Gaudilli\`ere},
 Condenser physics applied to Markov chains - A brief introduction to potential theory,
 arXiv:0901.3053.
\bibitem{GdHNOS}
  {\sc A. Gaudilli\`ere, F. den Hollander, F. R. Nardi, E. Olivieri and E. Scoppola},
  Ideal gas approximation for a two-dimensional rarefied gas under Kawasaki dynamics,
  {\sl Stoch. Proc. and Appl.\/} {\bf 119}, 737-774 (2009).
\bibitem{BL1}
  {\sc J. Beltr\'an, C. Landim},
  Tunneling and metastability of continuous time Markov chains
  {\sl J. Stat. Phys.\/} {\bf 140}, 1065-1114 (2010).
\bibitem{BdHS}
  {\sc A. Bovier, F. den Hollander and C. Spitoni},
  Homogeneous nucleation for Glauber  and Kawasaki dynamics in large volumes and low temperature
  {\sl Ann. Prob.\/} {\bf 38}, 661-713 (2010).
\bibitem{LLP} \textsc{D.A. Levin, M. Luczak, Y. Peres,},
Glauber dynamics for the mean field Ising model: cut-off, critical power
 law, and metastability, {\sl Probab.\ Theory. and Related Fields} \textbf{146}, no. 1-2,
223-265 (2010).
\bibitem{Mi1}
  {\sc L. Miclo},
  On absorption times and Dirichlet eigenvalues,
  {\sl ESIAM : PS\/} {\bf 14}, 117-150 (2010).
\bibitem{BL3}
  {\sc J. Beltr\'an, C. Landim},
  Metastability of reversible finite state Markov processes,
  {\sl Stoch. Proc. and Appl.\/} {\bf 121}, 1633-1677 (2011).
\bibitem{GSSV}
  {\sc A.  Gaudilli\`ere, B.  Scoppola, E. Scoppola and M. Viale},
  Phase transitions for the cavity approach to the clique problem on random graphs,
  {\sl Journal of Stat. Phys.\/} {\bf 145}, 1127-1155 (2011).
\bibitem{BL2}
  {\sc J. Beltr\'an, C. Landim},
  Metastability of reversible condensed zero range processes on a finite set,
  {\sc Probab. Theory and Related Fields\/} {\bf 152}, 781-807 (2012).
\bibitem{BBI2}
   {\sc A. Bianchi, A. Bovier and D. Ioffe},
   Point-wise estimates and exponential laws in metastable systems via coupling methods,
   {\sl Ann. Prob.\/} {\bf 40}, 339-371 (2012).
\bibitem{CM}
  {\sc R. Cerf, F. Manzo},
  Nucleation and growth for the Ising model in d dimensions at very low temperatures,
  {\sl Ann. Prob. \/} {\bf 41}, 3697-3785 (2013).
\bibitem{BL4}
   {\sc J. Beltr\'{a}n, C. Landim},
   A Martingale approach to metastability,
   {\sl Prob. Therory and Related Fields\/} (2014).
\bibitem{FMNSS}
   {\sc R. Fernandez, F. Manzo, F. Nardi, E. Scoppola and J. Sohier}
   Conditioned, quasi-stationary, restricted measures and escape from metastable states,
   arXiv:1410.4814.
\bibitem{BdH}
	A. Bovier and F. den Hollander,
	{\em Metastability: a potential-theoretic approach.}
	Monograph, xx + 569 pp.  to appear in  Grundlehren der mathematischen Wissenschaften,
	Vol. 351 Springer (2015), draft version available from A. Bovier's web page.
\end{thebibliography}
\end{document}